  \documentclass{amsart}

\usepackage{amsmath}
\usepackage{amssymb}
\usepackage{graphicx}
\usepackage{stmaryrd}
\usepackage{tikz}
\usepackage{psfrag}
\usepackage{bm}

\newtheorem{theorem}{Theorem}[section]
\newtheorem{corollary}[theorem]{Corollary}
\newtheorem{lemma}[theorem]{Lemma}
\newtheorem{proposition}[theorem]{Proposition}

\theoremstyle{definition}
\newtheorem{definition}{Definition}[section]

\newtheorem{remark}[definition]{Remark}

\numberwithin{equation}{section}

\newcommand{\B}{\mathbb B}
\newcommand{\R}{\mathbb R}

\newcommand{\K}{\mathcal K}
\newcommand{\U}{\mathcal U}
\newcommand{\V}{\mathcal V}
\newcommand{\A}{\mathbb A}

\newcommand{\Z}{\mathcal Z}
\newcommand{\T}{\mathcal T}
\renewcommand{\u}{\mathbf u}

\newcommand{\z}{\mathbf z}
\newcommand{\x}{\mathbf x}
\newcommand{\y}{\mathbf y}
\newcommand{\e}{\mathbf e}

\newcommand{\0}{\mathbf 0}

\renewcommand{\P}{\mathcal P}
\newcommand{\tr}{\operatorname{tr}}
\newcommand{\rank}{\operatorname{rank}}
\newcommand{\dist}{\operatorname{dist}}
\newcommand{\supp}{\operatorname{supp}}
\newcommand{\dv}{\operatorname{div}}
\renewcommand{\mod}{\operatorname{mod}}

\newcommand{\rad}{\operatorname{rad}}

\newcommand{\imply}{\Longrightarrow}

\title[Irregular Diffusions and  Polyconvex Gradient Flows]{Irregular Diffusions and Loss of Regularity \\in Polyconvex Gradient Flows}

\date{January 2, 2026}

 \author{Bin Guo}

\address{School of Mathematics, Jilin University, Changchun, Jilin Province 130012, China}
   \email{\tt bguo@jlu.edu.cn}

\author{Seonghak Kim}
\address{Department of Mathematics\\
Kyungpook National University\\
Daegu 41566, Republic of Korea}
\email{shkim17@knu.ac.kr   {\rm (Corresponding Author)}}

 \author{Baisheng Yan}
 \address{Department of Mathematics\\ Michigan State University\\ East Lansing, MI 48824, USA}
   \email{yanb@msu.edu}



\begin{document}

\begin{abstract}

We investigate diffusion-type partial differential equations that are irregular in the sense that they admit weak solutions which are nowhere smooth, even for prescribed smooth data.  By reformulating these equations as first-order partial differential relations and adapting the method of convex integration, we develop a construction scheme based on new geometric structures, referred to as $\mathcal{T}_N$-configurations, together with a simplified structural hypothesis on the diffusion functions, termed Condition $O_N$.  Under this condition, we show that the associated initial and boundary value problems with certain smooth initial-boundary data admit infinitely many Lipschitz weak solutions that are nowhere $C^1$.  We further analyze specific $\mathcal{T}_N$-configurations and establish nondegeneracy conditions that are essential for verifying Condition $O_N$. As an application, we construct examples of strongly polyconvex energy functionals whose gradient flows generate irregular diffusion equations, thereby revealing a failure of regularity and uniqueness even within the class of polyconvex gradient flows.

 \end{abstract}

\maketitle


\section{Introduction}

Let  $m, n\ge 1$ be integers,  let $\Omega\subset \R^n$ be a bounded  domain, and let $T>0$ be  fixed.  We study  diffusion-type partial differential equations of the form
\begin{equation}\label{DE}
 \partial_t u(x,t)  =\dv \sigma(Du(x,t)) \quad  \text{in $\, \Omega_T=\Omega \times (0,T),$}
\end{equation}
where   $\sigma\colon \R^{m\times n}\to \R^{m\times n}$ is a given function representing the diffusion flux, and \[
u=u(x,t)=(u^1,\dots ,u^m)\colon\Omega_T\to \R^m
\]
is the unknown function,  with time derivative $\partial_t u(x,t)=(\partial_t u^1,\dots , \partial_t u^m)\in\R^m$ and spatial gradient matrix  $Du(x,t)=\big(\frac{\partial u^i}{\partial {x_j}}\big) \in \R^{m\times n}.$

Here and throughout, $\R^{m\times n}$  denotes   the Euclidean space of real $m\times n$  matrices,  equipped  with the standard  inner product $\langle A,B\rangle=\tr(A^TB)$ and associated  norm $|A|=\langle A,A\rangle^{1/2}.$

Note that when $m=1$, equation  \eqref{DE}  represents a single equation, whereas when $m\ge 2,$ it becomes a system of $m$  equations.
In general, a function
\[
u\in L^1_{loc}(0,T; W^{1,1}_{loc}(\Omega;\R^m))
\]
 is said to be  a (very) {\emph{weak solution}} of  \eqref{DE}  provided that $\sigma(Du) \in L^1_{loc}(\Omega_T;\R^{m\times n})$  and
\begin{equation}\label{weak-sol}
 \int_{\Omega_T} \big ( u \cdot \partial_t  \varphi  - \langle  \sigma(D u), D\varphi \rangle \big )\, dxdt  =0\quad  \forall\, \varphi\in  C^\infty_0(\Omega_T;\R^m).
\end{equation}

For many partial differential equations, {\em standard} weak solutions are typically required to satisfy additional structural conditions and, in practice, exhibit higher regularity than these conditions alone would suggest; such equations are referred to as {\em regular.} In contrast, an equation is called {\em irregular} if its standard weak solutions fail to display any improved regularity.

This paper is primarily concerned  with the phenomenon of such irregularity for equation \eqref{DE}. We do not dwell on the precise definition of standard weak solutions of \eqref{DE} (see, e.g.,  \cite{BDM13,DiB93, LSU, Ln}). Instead, we focus on {\em Lipschitz} weak solutions
\[
u\in W^{1,\infty}(\Omega_T;\R^m),
\]
  which  qualify as  standard weak solutions    in any reasonable sense.

In the study of equation \eqref{DE}, in addition to standard smoothness and growth conditions on the function $\sigma$, certain structural hypotheses on $\sigma$ also play a crucial role,  especially for systems with $m,n \ge 2$. For example,  $\sigma$ is called strongly  \emph{rank-one monotone} if
\begin{equation}\label{r-mono}
 \langle \sigma(A+p\otimes a) -\sigma(A), \, p\otimes a\rangle \ge \nu |p|^2|a|^2
  \end{equation}
 for all $A\in \R^{m\times n},\, p\in\R^m$ and $a\in \R^n,$ where   $\nu>0$ is a constant.  In this case,  \eqref{DE}  is usually called   (strongly) {\em parabolic.}
 Moreover,  $\sigma$ is  called strongly {\em quasimonotone} if
 \begin{equation}\label{q-mono}
\int_\Omega \langle \sigma(A+D\phi), \, D\phi  \rangle \, dx  \ge \nu \int_\Omega |D\phi|^2\,dx
  \end{equation}
 for all $A \in \R^{m\times n}$ and $\phi\in  C^\infty_0(\Omega;\R^m);$ see   \cite{BDM13,Ha95,Zh86}.
It is  standard in the calculus of variations that  \eqref{q-mono} implies condition \eqref{r-mono}.
Furthermore,  \eqref{q-mono} is automatically satisfied  if $\sigma$ is (fully) strongly \emph{monotone}, in the sense that
 \begin{equation}\label{f-mono}
 \langle \sigma(A+B) -\sigma(A), \, B\rangle \ge \nu |B|^2 \quad \forall\,  A,\,B \in \R^{m\times n}.
  \end{equation}
When $m = 1$ or $n = 1$, the three conditions \eqref{r-mono}–\eqref{f-mono} are equivalent.

  Under full monotonicity \eqref{f-mono}, equation \eqref{DE} can be analyzed using standard parabolic theory; see \cite{Bar,Br,DiB93,LSU,Ln}. Under quasimonotonicity \eqref{q-mono} and suitable growth assumptions, the results in \cite{BDM13} show that every standard (and hence every Lipschitz) weak solution $u$ of \eqref{DE} enjoys improved \emph{partial regularity}. Specifically,
 $Du \in C^{\alpha,\alpha/2}_{\mathrm{loc}}(U;\R^{m \times n})$ for some $\alpha \in (0,1)$ and some open set $U \subset \Omega_T$ of full measure (i.e., $|\Omega_T \setminus U| = 0$). In this sense, the equation is regarded as regular.

Nowhere $C^1$  Lipschitz weak solutions were constructed in \cite{MRS05, MSv03, Sz04} for certain elliptic and parabolic systems in two dimensions ($n=2$) using the convex integration method \cite{Gr86}. This method has since been adapted to construct highly irregular solutions to scalar equations of the form \eqref{DE} with certain forward-backward diffusion functions
$\sigma$; see \cite{KY15, KY17, KY18, Zh06}. It has also been extended to systems of the form \eqref{DE} in which
$\sigma$  satisfies stronger parabolicity conditions; see \cite{Ya20, Ya22}.

The convex integration method has also achieved remarkable success across a wide range of important PDEs and applications. These include, among others, the incompressible Euler and Navier–Stokes equations \cite{BV19, DLSz09, DLSz17, Is18}, porous medium equations \cite{CFG11}, the Monge-Amp\`ere equations \cite{LP17}, active scalar equations \cite{Sh11}, and, most recently, the construction of wild weak solutions for elliptic equations, including the
$p$-Laplace equation \cite{CT22, Jo23}.

In this paper, we adapt  the convex integration scheme  developed in \cite{KY15,Ya20,Zh06} to  study  equation \eqref{DE} as a first-order  partial differential relation
\begin{equation}\label{pdr1}
u= \dv v\quad{\mbox{and}}\quad  (Du,\partial_t v)\in \K   \;\; \mbox{a.e. in $\Omega_T$}
\end{equation}
for  $(u,v)\colon\Omega_T\to\R^m\times   \R^{m\times n},$ where
$
 \K=\{(A,\sigma(A)):A\in \R^{m\times n}\}
$
  is the graph of the diffusion function $\sigma.$

The crux of  convex integration  for \eqref{pdr1} is  to construct a sequence of   uniformly bounded open sets $\{\U_\nu\}_{\nu=1}^\infty$ in $\R^{m\times n}\times \R^{m\times n}$ that converges  to the graph set $\K$ in a suitable sense, together with a   corresponding sequence of Lipschitz functions $\{(u_\nu,v_\nu)\}_{\nu=1}^\infty$ satisfying
\begin{equation}\label{pdr2}
u_\nu= \dv v_\nu\quad{\mbox{and}}\quad  (Du_\nu,\partial_t v_\nu)\in \U_\nu \;\; \mbox{a.e. in $\Omega_T$}
\end{equation}
such that the sequence $\{u_\nu\}_{\nu=1}^\infty$ is Cauchy   in $W^{1,1}(\Omega_T;\R^m).$
The limit function $u$ then yields  a Lipschitz weak solution of equation (\ref{DE}). Moreover, if the essential oscillation of $\{Du_\nu\}_{\nu=1}^\infty$ at each point in $\Omega_T$ is uniformly bounded below by a positive constant, the resulting solution $u$ is nowhere $C^1$ in $\Omega_T$.

We now elucidate the construction scheme for the sets $\{\U_\nu\}$ and the sequence $\{(u_\nu,v_\nu)\}.$
Suppose  that we are given  a function  $(u,v)\colon\Omega_T\to\R^m\times   \R^{m\times n}$  satisfying  $u=\dv v$ and $(Du,\partial_t v)\in \U$ almost everywhere in $\Omega_T,$ where $\U\subset \R^{m\times n}\times \R^{m\times n}$ is  a bounded  open set.
  To construct a new bounded open set $\tilde \U$  and a new function  $(\tilde u, \tilde v)$ satisfying
\begin{equation}\label{pert-00}
\tilde u=\dv \tilde v\quad \mbox{and} \quad (D\tilde u,\partial_t \tilde v)\in \tilde\U \;\;\mbox{a.e. in $\Omega_T,$}
 \end{equation}
with  $\tilde \U$ lying closer to the graph set $\K$ than $\U$ does (in a suitable sense), we introduce a perturbation of the form
\[
\tilde u =u +\varphi\quad \mbox{and} \quad \tilde v =v +\psi+g,
\]
where  $(\varphi,\psi,g)\in C^\infty_0(\Omega_T; \R^m\times \R^{m\times n}\times \R^{m\times n})$ is chosen to satisfy
\begin{equation}\label{pert-0}
 \dv \psi =0 \quad \mbox{and} \quad  \dv g=\varphi.
 \end{equation}
This guarantees $\tilde u=\dv \tilde v.$ To ensure
\[
(D\tilde u,\partial_t\tilde v)=(Du+D\varphi, \partial_t v +\partial_t\psi+\partial_t g)\in \tilde \U,
\]
the perturbations  $(D\varphi,\partial_t\psi)$ must be carefully designed to satisfy  \eqref{pert-0}, while $\|\partial_t g\|_{L^\infty}$ is kept sufficiently small.

Some  building blocks  for constructing  the set  $\tilde \U$ and the functions  $(\varphi,\psi,g)$  were previously developed by working with the full space-time gradients  of $(\varphi,\psi)$ and the associated so-called {\em $\tau_N$-configurations},  together with  a suitable anti-divergence operator that produces the {\em corrector} $g$ by solving  $\dv g=\varphi;$ see \cite{Ya20,Ya22,Ya23}. However, the heavy reliance on full-gradient structures in those works makes the resulting construction rather intricate.

In this paper, we investigate  new and simpler  building blocks motivated by perturbations  $(\varphi,\psi,g)$ of the form
\begin{equation}\label{block}
 \varphi= (a\cdot Dh)p,\;\;
  \psi=(a\cdot Dh)B-(BDh )\otimes a,\;\; \mbox{and}\;\;
  g=h p\otimes a,
\end{equation}
where $p\in \R^m,$ $a\in\R^n\setminus\{0\},$  $B\in\R^{m\times n}$ with $Ba=0$, and $h \in C^\infty_0(\Omega_T).$
  We focus on the case  of  {\em  localized plane-waves}  in which
\[
h= \zeta (x,t)f( a\cdot x+st),
\]
where $\zeta\in C^\infty_0(\Omega_T)$ is mostly  constant on its support and $f\in C^\infty(\R)$ is  highly-oscillatory  and periodic, with $f''$ mostly equal to, while remaining between,  two fixed constants. Under these choices,  the quantity   $(D\varphi,\partial_t\psi)$ will lie  in the neighborhood of a line-segment in the direction  $\gamma=(p\otimes a, B)$ (see Lemma \ref{lem-0}).   Using these directions $\gamma=(p\otimes a,B),$ we introduce the  {\em $\T_N$-configurations}  (see Definition \ref{def-tau-N}),  which are considerably simpler than  the  $\tau_N$-configurations studied in \cite{Ya20,Ya22,Ya23}.

Based on  these simpler  $\T_N$-configurations, we develop a  more efficient construction scheme  for the differential relation \eqref{pdr1}. We also reformulate  a  crucial structural hypothesis on the function $\sigma,$ termed  {\em Condition $O_N$} (see Definition \ref{O-N}), which substantially simplifies the earlier Condition $(OC)_N$ introduced in  \cite{Ya22,Ya23}.

The first main result of this paper concerns  the  initial and boundary  value problem (IBVP)
\begin{equation}\label{ibvp-0}
\begin{cases} \partial_t u=\dv \sigma(Du)  & \mbox{in $\Omega_T,$}\\
u=\bar u & \mbox{on $\partial'\Omega_T,$}
\end{cases}
\end{equation}
associated with equation  \eqref{DE},
where $\bar u\in W^{1,1}(\Omega_T;\R^m)$ is a given function representing the {\em initial-boundary datum}, and
\[
\partial'\Omega_T =(\Omega \times\{0\}) \cup [ \partial\Omega \times [0,T)]\subset \partial \Omega_T
\]
is the \emph{parabolic boundary} of $\Omega_T.$
 By a \emph{weak solution} to the IBVP  \eqref{ibvp-0}, we mean a  function  $u\in W^{1,1}(\Omega_T;\R^m)$ that is a weak solution  of equation \eqref{DE} and satisfies the initial-boundary condition $u=\bar u$ on $\partial'\Omega_T$ in the sense of traces in  $W^{1,1}(\Omega_T;\R^m).$

 In Theorem \ref{mainthm1}, we show that if $\sigma$ satisfies  Condition $O_N$, then the IBVP (\ref{ibvp-0})  (indeed,   even the   \emph{full Dirichlet problem} with $u=\bar u$ prescribed on the entire $\partial \Omega_T$) admits  infinitely many Lipschitz  weak solutions that are nowhere $C^1$, for certain  smooth  initial-boundary  data $\bar u.$ This demonstrates  the intrinsic irregularity of such  equations.

The second main result concerns  a special class of equations of the form \eqref{DE} with $m,n\ge 2,$ which arise as the {\em $L^2$-gradient flows} of   energy functionals of the form:
\begin{equation}\label{EN}
I(u)=\int_\Omega F(Du(x))\,dx,
\end{equation}
where $F\colon\R^{m\times n}\to\R$ is a $C^1$  function.
In this setting, the $L^2$-gradient flow of $I(u)$ is precisely equation \eqref{DE},  with  diffusion flux  $\sigma=DF.$

For such  gradient flows,  condition (\ref{f-mono}) is equivalent to the uniform strong {\em convexity} of   $F$, while condition (\ref{r-mono}) is equivalent to the uniform strong {\em  rank-one convexity} (or {\em Legendre--Hadamard condition}) of  $F.$
Moreover, between \eqref{r-mono} and \eqref{q-mono} lies another   important structural hypothesis, namely  the uniform strong {\em Morrey quasiconvexity} of   $F$, which can be expressed as
\begin{equation}\label{q-conv}
\int_\Omega [F(A+D\phi) -F(A)]\,dx \ge \frac{\nu}{2}\int_\Omega |D\phi|^2 dx\quad \forall\, A\in \R^{m\times n}, \;  \phi\in C^\infty_0(\Omega;\R^m).
\end{equation}
This condition is known to hold, in particular, if $F$ is  {\em strongly polyconvex} in the sense that
\begin{equation}\label{poly}
F(A)=\frac{\nu}{2}|A|^2 +G(J(A)),
\end{equation}
where $G(J)$ is a $C^1$  convex function of $J\in\R^q$ and  $J(A)\in\R^q$ denotes a fixed arrangement of  $q$ subdeterminants of $A\in \R^{m\times n}.$
 We refer the reader to \cite{AF84, Ba77, Da08, Ev86, Mo52} for further results and applications of these concepts in the calculus of variations and nonlinear elasticity.

Gradient flows associated with convex functionals can be analyzed using standard analytical techniques; see, for example, \cite{BDDMS20, KY25}. In contrast, the study of gradient flows for general {\em nonconvex} functionals,  particularly with respect to existence, uniqueness  and regularity,  remains considerably more challenging; see \cite{AGS05} for some general theory. The work \cite{ESG05} addresses a special polyconvex gradient flow in the class of diffeomorphism solutions.

In this paper, we show that for all $m,n\ge 2,$ there exist smooth, strongly polyconvex functions $F$ on $\R^{m\times n}$ such that  the initial and boundary value problem  for the gradient flow of  the energy functional $I(u)$  in \eqref{EN}  admits   both a smooth solution and   infinitely many Lipschitz  weak solutions that are nowhere $C^1$ for certain smooth initial-boundary data; see Corollary \ref{mainthm3}. This  significantly strengthens the results of \cite{MRS05},   where  nowhere $C^1$ Lipschitz  weak solutions $u$ with  zero initial-boundary condition were constructed for a parabolic system
\[
\partial_t u-\dv DF(Du) =f
\]
in the case $m=n=2,$  with a smooth strongly quasiconvex  $F\colon\R^{2\times 2}\to\R$  and a small H\"older forcing $f$ with $\|f\|_{C^\alpha}<\eta$ for given $\alpha,\eta \in (0,1).$  In contrast, our result applies to all $m, n\ge 2$, with a smooth strongly polyconvex   $F$ and the zero forcing $f\equiv 0$; this demonstrates a fundamental failure of both regularity and uniqueness, even in the setting of polyconvex gradient flows.

We conclude   by outlining the structure of the paper. The main results are stated in Section \ref{s-2}, with the necessary notations introduced in the subsequent sections. In Section \ref{s-3}, we present the definitions of $\T_N$-configurations and Condition $O_N$, along  with several key results concerning convex integration. Section \ref{s-4} is devoted to the proof of Theorem \ref{mainthm1}, which builds on the main stage theorem for convex integration (Theorem \ref{thm1}). In Section \ref{s-5}, we analyze  certain specific $\T_N$-configurations and establish the nondegeneracy conditions  essential for verifying Condition $O_N$; although this section is long, it  consists primarily of elementary setup and computations.  Finally, in Section \ref{s-6}, we construct smooth, strongly polyconvex functions $F$ on $\mathbb{R}^{2 \times n}$ for which the flux function $\sigma = DF$ satisfies Condition $O_5$, thereby completing the proof of Theorem \ref{mainthm2}.


\section{Statement  of Main Results}\label{s-2}

We state the main results below and refer the reader  to Section~\ref{s-3} for the relevant definitions and notations. The proofs of   Theorems~\ref{mainthm1} and \ref{mainthm2} are presented  in Sections~\ref{s-4} and \ref{s-6}, respectively.

\begin{theorem}\label{mainthm1} Let $\sigma\colon\R^{m\times n}\to \R^{m\times n}$ be locally Lipschitz and satisfy  Condition $O_N$  for some  $N\ge2,$ with $\Sigma(1)$ denoting the corresponding open set  as  in Definition \ref{O-N}.

Assume that $(\bar u, \bar v)\in C^1(\bar \Omega_T;\R^m\times\R^{m\times n})$ satisfies
\begin{equation}\label{subs}
\bar u=\dv \bar v \quad \mbox{and}\quad   (D\bar u,\partial_t \bar v)\in \Sigma(1) \;\;\mbox{on  $\bar \Omega_T.$}
\end{equation}
Then for any $\delta\in (0,1),$  the \emph{Dirichlet problem}
\begin{equation}\label{ibvp-5}
\begin{cases}  \partial_t u=\dv \sigma(Du )   & \mbox{in $\Omega_T,$}\\
 u  |_{\partial\Omega_T}=\bar u &
\end{cases}
\end{equation}  admits  a  Lipschitz weak solution  $u$ with  $\|u-\bar u\|_{L^\infty}+\|\partial_t u-\partial_t \bar u\|_{L^\infty} <\delta$  such that $Du$ is nowhere essentially continuous in  $\Omega_T.$

 In particular,  problem \eqref{ibvp-5}  possesses {\em infinitely many}  Lipschitz weak solutions  $u$ with $Du$  being nowhere essentially continuous  in  $\Omega_T.$

\end{theorem}

\begin{remark}\label{rk-1}
From the definition of  Condition $O_N$  in Definition \ref{O-N}, clearly,  $(O,O)=\pi_1(0)\in S^{r_0}_1(0)\subset \Sigma(1);$ thus $(\bar u,\bar v)\equiv (0,0)$ satisfies  condition (\ref{subs}).

More generally, if $(A,B)\in \Sigma(1),$ with $A=(a_{ij})$ and $B=(b_{ij}),$ then the function $(\bar u,\bar v)$ satisfies  condition (\ref{subs}), where  $\bar u=(\bar u^1,\dots ,\bar u^m)\colon \Omega_T\to \R^m$ and $\bar v=(\bar v_{ij})\colon \Omega_T\to\R^{m\times n}$ are defined by
\begin{equation}\label{IB-0}
\bar u^i(x,t) =u^i_0(x)=\sum_{k=1}^n a_{ik}x_k\; \mbox{ and }\;  \bar v_{ij}(x,t)=\frac12 a_{ij}x_j^2 +t b_{ij} \quad (i=1,\dots ,m;\, j=1,\dots ,n).
\end{equation}
 With such  time-independent data $\bar u,$   problem (\ref{ibvp-5})  clearly admits the {\em stationary}   solution $u\equiv \bar u.$ However,  Theorem \ref{mainthm1} shows that the problem  also possesses  infinitely many {\em non-stationary} Lipschitz  weak solutions that are  nowhere $C^1$. These time-dependent solutions are of greater interest for the diffusion problem.
\end{remark}

\begin{theorem}\label{mainthm2}
 Let $n\ge 2.$  Then there exist   strongly polyconvex functions $F\colon \R^{2\times n}\to\R$ of the form
 \begin{equation}\label{poly-00}
  F(A)=\frac{\nu}{2} |A|^2 + G(A, \delta(A)) \quad \forall\,A=(a_{ij}) \in \R^{2\times n},
  \end{equation}
  where  $\nu>0,$  $G\colon \R^{2\times n}\times \R\to\R$  is  smooth  and  convex, and $\delta(A)= a_{11}a_{22}- a_{12} a_{21},$   such that  the function $\sigma=DF$ satisfies {\em  Condition $O_5$} and that the growth condition
\[
|D^kF(A)|\le C_k (|A|+1) \quad (A\in \R^{2\times n},\; k=1,2,\dots)
\]
holds for some constants $C_k>0.$
\end{theorem}

\begin{corollary}\label{mainthm3}
For all $m,n\ge 2$, there exist {smooth and strongly polyconvex}  functions $F$ on $\R^{m\times n}$  satisfying the growth condition
\begin{equation}\label{growth-00}
|D^kF(A)|\le C_k (|A|+1) \quad (A\in \R^{m\times n},\; k=1,2,\dots)
  \end{equation}
for some constants $C_k>0,$ such that the Dirichlet problem
\begin{equation}\label{ibvp}
\begin{cases}  \partial_t  u=\dv DF(D u )   & \mbox{in $\Omega_T,$} \\
 u  |_{\partial \Omega_T}=\bar u, &
 \end{cases}
\end{equation}
 with  certain smooth stationary weak solutions  $\bar u=u_0(x)$ of equation \eqref{DE},  possesses  infinitely many  Lipschitz  weak solutions that are nowhere $C^1$  in $\Omega_T.$
 \end{corollary}

\begin{proof}
We first treat the case $m=2.$ Let  $\tilde F=F\colon \R^{2\times n}\to\R$ be a smooth and strongly polyconvex  function as in (\ref{poly-00}) of Theorem \ref{mainthm2}. Then $\sigma=D\tilde F$  satisfies Condition $O_5;$ accordingly, let $\Sigma(1)$ be  defined by the function $\sigma=D\tilde F$ as in Definition \ref{O-N}.

Fix any $(A,B)\in \Sigma(1)$ and define $(\bar u,\bar v)$ as  in \eqref{IB-0}.  Let  $\tilde u_*(x,t)=\bar u(x,t)=\tilde u_0(x)\in\R^2.$
Then, by Theorem \ref{mainthm1} and Remark \ref{rk-1},  the  Dirichlet  problem
\begin{equation}\label{DP}
\begin{cases}  \partial_t  \tilde u=\dv D\tilde F(D \tilde u )   \;\; \mbox{ \rm in $\Omega_T,$} \\
\tilde u  |_{\partial \Omega_T}=\tilde u_*
\end{cases}
\end{equation}
admits  infinitely many  Lipschitz weak solutions $\tilde{u}:\Omega_T\to \R^2$ that are nowhere $C^1$  in $\Omega_T.$
Thus,  the result holds when $m=2.$

Now assume $m\ge 3.$ Define $F\colon \R^{m\times n}\to\R$ by
\[
F(A)=\frac{\nu}{2}|A_2|^2 +\tilde F(A_1) \quad \forall\,A=\begin{bmatrix}A_1\\A_2\end{bmatrix},\;\; A_1\in \R^{2\times n},\;\; A_2\in \R^{(m-2)\times n},
\]
 and define $\bar u \colon \Omega_T\to\R^2\times \R^{m-2}=\R^m$  by
\begin{equation}\label{DP-2}
\bar u(x,t)=(\tilde u_*(x,t), \,0)=(\tilde u_0(x),\, 0):=u_0(x) \in \R^m.
\end{equation}
Then $\bar u$ is a smooth stationary weak solution  of equation \eqref{DE}. Moreover,
\[
DF(A)= \begin{bmatrix} D\tilde F(A_1)\\ \nu A_2 \end{bmatrix}\quad\text{and}\quad F(A)=\frac{\nu}{2}|A|^2 +G(A_1,\delta(A_1)).
\]
Hence  $F$ is of the form (\ref{poly}) and satisfies the growth condition \eqref{growth-00}.

 Let  $ u=(\tilde u,\, 0)\colon \Omega_T\to\R^2\times \R^{m-2}=\R^m,$ where $\tilde{u}:\Omega_T\to \R^2$ is a Lipschitz weak solution to problem \eqref{DP} that is nowhere $C^1$  in $\Omega_T.$ Then
\[
 \partial_t u=(\partial_t\tilde u, \,  0),   \quad Du=\begin{bmatrix} D \tilde u\\O\end{bmatrix}, \quad\mbox{and}\quad  DF(Du)=\begin{bmatrix} D\tilde F(D\tilde u)\\O\end{bmatrix}.
\]
Therefore,  $u$ is a Lipschitz weak solution  to problem (\ref{ibvp}) with initial–boundary datum $\bar u$ given by \eqref{DP-2}, and it is nowhere $C^1$  in $\Omega_T.$
 \end{proof}


\section{Convex Integration,  $\T_N$-configurations and  Condition $O_N$}\label{s-3}

We begin by introducing some notations. Let $q\ge1$ be an integer. We denote the Lebesgue measure of a measurable set $G\subset \R^q$  by $|G|=\mathcal L_q(G).$

Given a set {$S\subset\R^q$}  and a number $\epsilon>0$,  we denote by $S_\epsilon$   the $\epsilon$-neighborhood of $S;$ i.e.,
\[
S_\epsilon=\{\xi \in \R^q : \dist(\xi,S)<\epsilon\}
=\bigcup_{\eta\in S}\{\xi \in \R^q : |\eta-\xi|<\epsilon\}.
\]
Given $\xi,\eta\in \R^q$,  we denote  by $[\xi,\eta]=\{\lambda \xi+(1-\lambda)\eta: \lambda\in [0,1]\}$  the closed line segment {in $\R^q$} connecting $\xi$ and $\eta.$

Throughout this section,   we always assume that $m,n\ge 1$ and $N\ge2$ unless otherwise specified.

\subsection{Convex Integration Lemma} The following result provides the initial  building blocks  for  convex integration of the differential relation (\ref{pdr1}).

\begin{lemma}\label{lem-0} Let   $\gamma=(p\otimes a, B)\in  \R^{m\times n}\times \R^{m\times n},$ where
\[
p\in\R^m,\; \; a\in\R^n\setminus \{0\}, \; \;  B\in \R^{m\times n}  \quad\mbox{and}\quad  Ba=0.
\]
 If $0\le\lambda\le 1,$  then  for each  bounded open  set $G\subset \R^{n+1}$ and $0<\epsilon<1,$   there exists a  function  $(\varphi,  \psi,g) \in C^\infty_0(G;  \R^m\times  \R^{m\times n}\times  \R^{m\times n})$  with the following properties:
\begin{itemize}
\item[(a)]  $\dv\psi =0$, $\dv g =\varphi$ and $(D\varphi,\partial_t\psi)\in [-\lambda\gamma, (1-\lambda)\gamma]_\epsilon$  in $G;$
\item[(b)]  $\|\varphi\|_{L^\infty(G)} + \|\partial_t\varphi\|_{L^\infty(G)}+ \|\partial_tg\|_{L^\infty(G)} <\epsilon;$
\item[(c)] there exist two disjoint open  sets $G', G''\subset\subset G$ such that
\[
\begin{cases}
|G'| \ge (1-\epsilon) \lambda |G|, & (D\varphi,\partial_t\psi)=(1-\lambda) \gamma  \;\;  \text{\rm in $G'$;}  \\
|G''| \ge (1-\epsilon) (1-\lambda) |G|, & (D\varphi,\partial_t\psi)=-\lambda \gamma \; \;   \mbox{\rm in $G''$.}
\end{cases}
\]
\end{itemize}
\end{lemma}

\begin{proof}   If $\lambda=0$ or $\lambda=1$, the result follows by taking  $(\varphi,  \psi,g)\equiv 0$ and choosing  $G'=\emptyset$ or $G''=\emptyset,$ respectively.
Hence, in what follows, we assume $0<\lambda<1.$

Let $\zeta\in C^\infty_0(G)$ be such that
\[
\mbox{$0\le \zeta\le 1$ in $G,$  \;  $\zeta=1$ in some open set $\tilde G\subset\subset G$ \,with    $|\tilde G|\ge (1- \epsilon)^{1/3}|G|,$ }
\]
and let $I_1$ and $I_2$ be  two  disjoint closed intervals   in $(0,1)$ such that
\[
  |I_1| = (1-\epsilon)^{1/3} \lambda \quad\mbox{and}\quad
 |I_2| = (1-\epsilon)^{1/3} (1-\lambda).
\]
Select a function $q\in C^\infty_0(0,1)$ such that $-\lambda \le q \le 1-\lambda$ on $ (0,1)$ and
\[
  \{\tau\in (0,1):q(\tau)=1-\lambda\}=I_1,\quad
 \{\tau\in (0,1):q(\tau)=-\lambda\}=I_2,\quad
    \int_0^1 q(\tau)d\tau =0.
\]
Extend  $q$ to all of  $\R$ as a 1-periodic function; then $q\in C^\infty(\R).$
Define
  \[
  f(\tau)=\int_0^\tau\int_0^t q(s)\,ds\,dt-\tau \int_0^1\int_0^t q(s)\,ds\,dt\quad\forall\, \tau\in\R.
  \]
 Then $f\in C^\infty(\R)$ is  1-periodic and satisfies $f''(\tau)=q(\tau)$ for all $\tau\in\R;$ consequently,
 \[
   \{\tau\in (0,1):  f''(\tau)=1-\lambda\}=I_1 \quad\mbox{and}\quad  \{\tau \in (0,1):   f''(\tau)=-\lambda\}=I_2.
 \]

Let $s, \delta>0$ be numbers to be  chosen later, and let
\[
h =\delta^2 \zeta(x,t)f\Big(\frac{a\cdot x +st}{\delta}\Big ).
\]
  Then $
 Dh=\delta^2 fD\zeta + \delta  \zeta f' a.$ Define
 \[
\begin{cases} \varphi= |a|^{-2}(a\cdot Dh)p,\\
   \psi= \frac{1}{s}|a|^{-2}[ (a\cdot Dh)B-(BDh)\otimes a],\\
   g=|a|^{-2}hp\otimes a,
   \end{cases}
       \]
It is easily verified that
\[
 \begin{cases}
 \dv\psi=0,\;\; \dv g=\varphi,\\
 \varphi=\delta^2|a|^{-2}  (D\zeta \cdot a) fp +\delta  \zeta f' p,\\
 \psi=\frac{1}{s}[\delta^2|a|^{-2}   (D\zeta \cdot a)f B +\delta  \zeta f' B -\delta^2 f |a|^{-2}   (BD\zeta)\otimes a],\\
 g=\delta^2|a|^{-2} \zeta fp\otimes a,
 \end{cases}
\]
where, in the formula for $\psi$, we have used $Ba=0.$
Hence
\[
 \begin{split} D\varphi   & =   \delta^2|a|^{-2} fp\otimes D (D\zeta \cdot a) +\delta |a|^{-2}  (D\zeta \cdot a) f' p\otimes a \\
 &\quad  +\delta   f' p\otimes D\zeta + \zeta f''p\otimes a,\\
 \partial_t\varphi  & = \delta^2|a|^{-2}  (D\zeta \cdot a)_t fp+\delta |a|^{-2}  (D\zeta \cdot a) f' sp +\delta  \zeta_t f' p+ \zeta f'' s p,\\
 \partial_t g  & =    \delta^2|a|^{-2} \zeta_t  fp\otimes a+ \delta |a|^{-2} \zeta f' sp\otimes a,\\
\partial_t    \psi  &  =    \frac{\delta^2}{s} |a|^{-2}   (D\zeta \cdot a)_t f B+\delta |a|^{-2}   (D\zeta \cdot a)f'  B +\frac{\delta}{s} \zeta_t f' B+ \zeta f''  B  \\
    &\quad
  -\frac{\delta^2}{s} f |a|^{-2}   ((BD\zeta)\otimes a)_t -\delta  f'  |a|^{-2}   (BD\zeta)\otimes a.
 \end{split}
\]

We now select  $s\in (0,1)$ so that
$
s\|\zeta f''  p\|_{L^\infty(G)} <\epsilon/2.
$
Then
 \[
\begin{cases} \varphi =O(\delta),\; \partial_t\varphi  =\zeta f'' s p +O(\delta),\;
\partial_t g =O(\delta),\\
   D\varphi =f''\zeta p\otimes a +O(\delta),\;
\partial_t\psi  =f''\zeta B +O(\delta),
\end{cases}
  \]
 where $O(\delta)$'s denote the terms  such that $|O(\delta)|\le C\delta$ for  a uniform constant $C$ independent of $\delta\in (0,1).$ Thus
\[
(D\varphi,\partial_t\psi ) =f''\zeta (p\otimes a,B) +O(\delta) =f''\zeta \gamma +O(\delta).
\]
Since $f''\zeta \gamma \in [-\lambda\gamma,(1-\lambda)\gamma]$ in $G$, it follows that, for all sufficiently small $\delta>0,$
\[
 \begin{cases}
 \|\varphi\|_{L^\infty(G)} + \| \partial_t\varphi \|_{L^\infty(G)}+ \| \partial_t g \|_{L^\infty(G)}<\epsilon,\\
 (D\varphi,\partial_t\psi) \in [-\lambda\gamma,(1-\lambda)\gamma]_\epsilon \;\;\mbox{in $G.$}\end{cases}
 \]

Next, let $I_k^o$ be the interior of $I_k$ for $k=1,2.$ Let $\chi_k=\chi_{\tilde I_k}$  be the characteristic function  of the open set $
\tilde I_k=\bigcup\{j+ I_k^o: j\in \{0,\pm 1,\pm 2,\dots\}\}
$
on $\R,$  and define
  \[
  \tilde G_k^\delta= \Big \{y\in \tilde G: \,  \frac{\alpha\cdot y}{\delta}\in \tilde I_k \Big\}= \Big \{y\in \tilde G: \, \chi_k ( \frac{\alpha\cdot y}{\delta}) =1 \Big\},
  \]
where  $y=(x,t)$ and $\alpha=(a,s).$ Then $\tilde G_k^\delta$ is open for  $k=1,2$ and $\delta>0.$
Moreover,
\[
 (D\varphi,\partial_t\psi)|_{\tilde G_1^\delta} =f''\gamma  =(1-\lambda)\gamma \quad\text{and}\quad
  (D\varphi,\partial_t\psi)|_{\tilde G_2^\delta} =f''\gamma =-\lambda \gamma.
  \]
Let $\{\alpha_i\}_{i=1}^{n+1}$, with $\alpha_1= \frac{\alpha}{|\alpha|}$, be an orthonormal basis of $\R^{n+1}.$ Then, by the change of variables  $ y=\sum_{i=1}^{n+1} z_i\alpha_i$  and Fubini's  theorem, we  have that, with $\delta'=\delta/|\alpha|,$ for each $k=1,2,$
 \[
\begin{split}  |\tilde G_k^\delta|&=\int_{\tilde G}  \chi_k( \frac{\alpha\cdot y}{\delta})\,dy=\int_{\R^{n+1}} \chi_{\tilde G}(y) \chi_k( \frac{\alpha_1\cdot y}{\delta'})\,dy\\
&  =\int_{\R^{n+1}} \chi_{\tilde G'}(z) \chi_k( \frac{z_1}{\delta'})\,dz   =\int_{\R} \Big ( \int_{\R^n}\chi_{G'(z_1)}(z') \chi_k( \frac{z_1}{\delta'})\,dz'\Big) \,dz_1\\
& =\int_{\R} |G'(z_1)|\chi_k(\frac{z_1}{\delta'})\,dz_1,\end{split}
  \]
where $\tilde G'=\{z\in\R^{n+1} : y=\sum_{i=1}^{n+1} z_i\alpha_i\in \tilde G\}$ and  $ G'(z_1)= \{z' \in\R^n: (z_1,z')\in  \tilde G'  \}.$

As each $\chi_k$ is 1-periodic, it follows that  sequence $\{\chi_k(m\tau)\}_{m=1}^\infty$ weak* converges  in $L^\infty(\R)$ to the  constant  $C_k=\int_0^1 \chi_k(\tau)\,d\tau=|I_k|$.
So, taking $\delta=\delta_m=\frac{|\alpha|}{m}$, and letting $m\to\infty,$  we have that
\[
  |\tilde G_k^{\delta_m}| =\int_{\R} |G'(z_1)| \, \chi_k(m  z_1)\,dz_1 \to |I_k|  \int_\R |G'(z_1)|\,dz_1=|I_k||\tilde G'|=|I_k||\tilde G|\quad (k=1,2).
  \]
Since $  |I_1|=(1-\epsilon)^{1/3} \lambda, \, |I_2|=(1-\epsilon)^{1/3} (1-\lambda)$ and $|\tilde G|\ge (1-\epsilon)^{1/3}|G|,$ it follows that
 \[
  |I_1||\tilde G| >(1-\epsilon)  \lambda |G| \quad\text{and}\quad
 |I_2| |\tilde G| >(1-\epsilon) (1-\lambda)|G|.
 \]
Therefore,  we choose  a sufficiently small $\delta=\delta_m >0$ to  complete the proof,  by taking $G'=\tilde G_1^{\delta}$ and $G''=\tilde G_2^{\delta}.$
\end{proof}

\subsection{$\T_N$-configurations}

In view of  Lemma \ref{lem-0}, we define
\begin{equation}\label{set-G}
\Gamma=\big\{(p\otimes a, B):   p\in\R^m,  \;  a\in\R^n\setminus \{0\},\; B\in \R^{m\times n} \;\; \mbox{and  $Ba=0$}\big \}
\end{equation}
to be the   {\em wave cone} associated with the  differential relation (\ref{pdr1}).  Since $p\otimes a=|a|p\otimes \frac{a}{|a|},$ it is readily seen that $\Gamma$ is a closed cone.

When $n=2$, the set $\Gamma$ is equivalent to the  cone consisting of the zero matrix and all  rank-one  matrices   in $\R^{2m\times 2}.$ Indeed, if $a\in\R^2\setminus \{0\}$ and $ B\in \R^{m\times 2}$ satisfy $Ba=0,$ then  $B=q\otimes a^\perp$ for some $q\in\R^m,$ where $a^\perp=(-a_2,a_1)$ if $a=(a_1,a_2)\in\R^2.$ Hence,
\[
(p\otimes a, B)=(p\otimes a, q\otimes a^\perp)
\]
 can be identified with  the  matrix $ (p,q)\otimes a$ in $\R^{2m\times 2},$ where $p,q\in\R^m$ and $ a\in\R^2\setminus \{0\}.$

 \smallskip

The following definition generalizes the $T_N$-configurations studied in \cite{KMS03, MRS05, MSv03, Sz04, Ta93} and also  substantially simplifies the $\tau_N$-configurations introduced  in \cite{Ya20,Ya22,Ya23}.

\begin{definition}[$\T_N$-configuration] \label{def-tau-N}
Let $\xi_1,\dots,\xi_N\in\R^{m\times n}\times \R^{m\times n}.$

We say that the ordered $N$-tuple $(\xi_1,\dots ,\xi_N)$ is a  {\em  $\T_N$-configuration} in $\R^{m\times n}\times \R^{m\times n}$ provided that   there exist $\rho\in \R^{m\times n}\times \R^{m\times n}$, $\gamma_1,\dots,\gamma_N\in \Gamma$ and $\kappa_1,\dots,\kappa_N>1$ such that
\begin{equation}\label{sum-0}
 \gamma_1+\gamma_2+\dots+\gamma_N=0
\end{equation}
and
\begin{equation}\label{t-N}
\xi_1=\rho+\kappa_1\gamma_1 \quad\text{and}\quad  \xi_i=\rho+\gamma_1+\dots+\gamma_{i-1}+\kappa_i \gamma_i \quad  (2\le i\le N).
\end{equation}
In this case, define  $\pi_1=\rho$ and $ \pi_i=\rho+\gamma_1+\dots  +\gamma_{i-1}$ ($2\le i\le   N$), and set
\[
\T(\xi_1,\dots,\xi_N)=\bigcup_{j=1}^N [\xi_j,\pi_j] \quad \text{(\emph{See Figure \ref{fig1}}).}
\]
Moreover,  if $\mathcal E\subset \R^{m\times n}\times \R^{m\times n}$ and  $\xi_i\in \mathcal E$ for all $1\le i\le N,$   we say that $(\xi_1,\dots,\xi_N)$ is {\em supported on}  $\mathcal E.$
\end{definition}

\begin{figure}[ht]
\begin{center}
\begin{tikzpicture}[scale =1]

 \draw[thick] (-5,-2)--(1,-1);
  \draw(-5,-2) node[below]{$\xi_1$};
   \draw(1,-1) node[right]{$\pi_1=\pi_{N+1}=\rho$};
 \draw(-1,-3) node[below]{$\xi_N$};
 \draw[thick] (1,-1)--(3,1);
  \draw[thick] (-1,-3)--(1,-1);
  \draw[thick] (3,1)--(1,3);
  \draw[thick] (3,1)--(4,0);
      \draw[thick] (1,3)--(-2,4);
     \draw[thick] (1,3)--(4,2);
     \draw(4,0) node[right]{$\xi_{N-1}$};
     \draw(3,1) node[right]{$\pi_N$};
      \draw[thick] (-2,4)--(-4,3);
   \draw[thick] (-4,3)--(-5, 1);
     \draw[thick] (-5,1) --(-4.5,-0.5);
       \draw[thick] (-4.5,-0.5) --(-2,-1.5);
        \draw(-5,1) node[left]{$\pi_4$};
         \draw(-2,-1.5) node[below]{$\pi_2$};
   \draw[thick] (-4.5,-0.5) --(-7,0.5);
   \draw(-4.5,-0.5) node[below]{$\pi_3$};
   \draw[thick] (-5,1)--(-5.5,2.5);
    \draw(-5.5,2.5) node[left]{$\xi_3$};
   \draw(-7,0.5) node[below]{$\xi_2$};
   \draw[thick] (-4,3)--(-3,5);
   \draw[thick] (-2,4)--(0,5);
         \draw[fill] (-5,-2) circle (0.07);
          \draw[fill] (-1,-3) circle (0.07);
         \draw[fill] (4,0) circle (0.07);
         \draw[fill] (-7,0.5) circle (0.07);
         \draw[fill] (-5.5,2.5) circle (0.07);
            \draw[fill] (-3,5) circle (0.07);
              \draw[fill] (0,5) circle (0.07);
                \draw[fill] (4,2) circle (0.07);
\end{tikzpicture}
\end{center}
\caption{Illustration of a $\T_N$-configuration $(\xi_1,\xi_2,\dots ,\xi_N)$ determined by $\pi_1=\pi_{N+1}=\rho$ and $\gamma_i=\pi_{i+1}-\pi_i\in \Gamma$ $(1\le i\le N).$
The set $\T(\xi_1,\xi_2,\dots ,\xi_N)$ is the union of all the closed line segments shown.}
\label{fig1}
\end{figure}
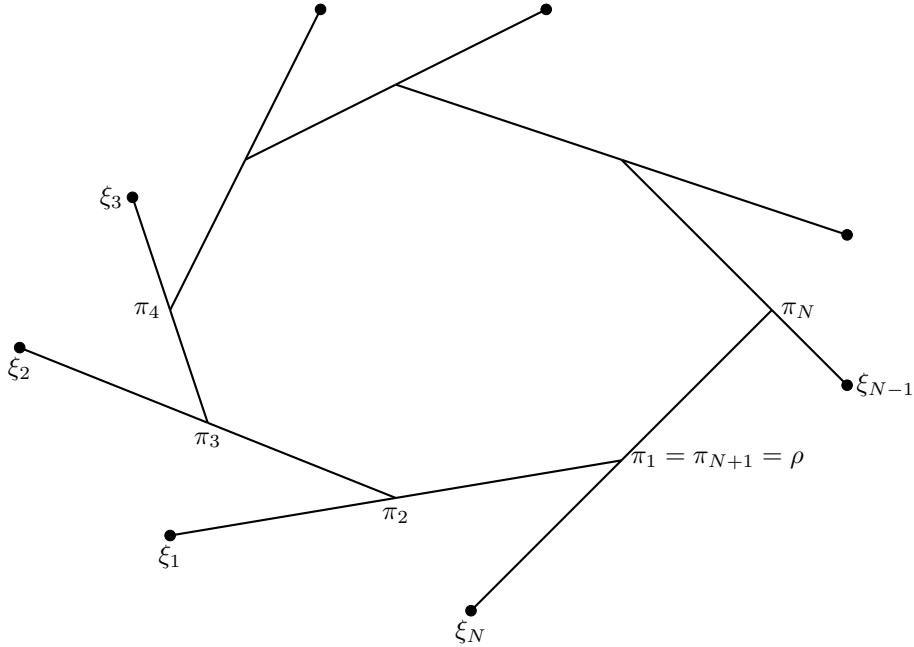

 \begin{remark}\label{rk-3}
Let $(\xi_1,\dots,\xi_N)$ be a $\T_N$-configuration  determined by
$(\rho,\,  \gamma_1,\dots,\gamma_N,  \kappa_1,\dots, \kappa_N)$  as  above.

(i) Define $\chi_i=1/\kappa_i\in (0,1)$ for $1\le i\le N,$  and  set $\pi_{N+1}=\pi_1=\rho.$ Then
\begin{equation}\label{t-N-2}
 \pi_{i+1}=  \chi_i \xi_i+(1-\chi_i)\pi_i
 \quad \forall\, 1\le i\le N.
 \end{equation}

(ii)  Assume $\chi_i<\lambda_i \le 1$ for  $1\le i\le N,$ and define $\xi_i'=\lambda_i\xi_i+(1-\lambda_i)\pi_i.$ Then  $(\xi'_1,\dots,\xi'_N)$ is  a $\T_N$-configuration,  determined  by $(\rho,\,  \gamma_1,\dots,\gamma_N, \lambda_1\kappa_1,\dots,\lambda_N\kappa_N).$ Moreover,
\[
\T(\xi_1',\dots,\xi_N')\subseteq \T(\xi_1,\dots,\xi_N).
\]

(iii) For each $1\le i\le N$,  consider the cyclically shifted  $N$-tuple
 \[
 (\xi_i, \,  \xi_{i+1},\,  \dots, \, \xi_{i+N-1}),
 \]
{\em where the indices are taken modulo  $N.$}  This is  also a $\T_N$-configuration, now determined  by
\[
( \pi_i,\; \gamma_i,\dots,\gamma_{i+N-1}, \kappa_i, \dots,\kappa_{i+N-1}),
\]
 with  indices  taken modulo $N.$
 In particular,  $ \T(\xi_1,\dots,\xi_N)=\T(\xi_i,\dots,\xi_{i+N-1})$ for all $1\le i\le N.$

   \end{remark}

The following elementary lemma  clarifies  some of the computations below based on (\ref{t-N-2});  its proof  follows directly by a simple cycle of iteration and is therefore  omitted.

\begin{lemma}\label{lem0} Let   $P_k$, $X_k$ and  $0<t_k<1$ be periodic on integers $k$ modulo  $N$ and satisfy the relation
\[
 P_{k+1}=t_kX_k+(1-t_k)P_k\quad  \forall\;k \; \mod N.
\]
 Then  $P_i=\sum_{j=1}^N \nu_{i}^{j}X_{j},$ where the coefficients $\nu_i^j,$  with $ i,j\; \mod N,$  are given  by
\begin{equation}\label{co-eff-0}
 \begin{split}    \nu_{i}^{i-1} & =\frac{t_{i-1}}{1-(1-t_1)\cdots(1-t_N)},\\
  \nu_{i}^{j} &  =\frac{(1-t_{i-1})(1-t_{i-2})\cdots (1-t_{j+1})t_{j}}{1-(1-t_1)\cdots(1-t_N)}\quad  (i-N\le j\le i-2). \end{split}
  \end{equation}

  Note that $\sum_{j=1}^N \nu_{i}^{j}=1$ for all $i\;\mod N.$
\end{lemma}

The following result provides  the building blocks for convex integration from a $\T_N$-configuration.

\begin{theorem} \label{convex-block-thm} Let   $(\xi_1,\xi_2,\dots  ,\xi_N)$ be a $\T_N$-configuration and
$
\eta= \lambda \xi_i +(1-\lambda) \pi_i$ for some $1\le i\le N$ and $0\le \lambda\le 1.$  Then $ \eta=\sum_{j=1}^N  \nu_j \xi_j,$ where
\[
\nu_i= \lambda +  (1-\lambda) \nu_{i}^i, \quad  \nu_j=(1-\lambda)  \nu_{i}^j\quad  \forall\,j\ne i,
\]
and $\{\nu_{i}^j\}_{j=1}^N$ are given by  formula \eqref{co-eff-0}  with $t_k=\chi_k$ $(1\le k\le N)$.

Moreover, for each bounded open set $G\subset \R^{n+1}$  and   $0<\delta<1,$  there exists  a  function  $(\varphi,  \psi,g) \in C^\infty_0(G; \R^m\times  \R^{m\times n}\times  \R^{m\times n})$   satisfying the following properties:
\begin{itemize}
\item[(a)] $\dv\psi=0$, $\dv g =\varphi$ and $\eta+(D\varphi,  \partial_t \psi)\in  [\T(\xi_1,\dots,\xi_N)]_\delta$ in $G;$
\item[(b)]  $\|\varphi\|_{L^\infty(G)}+\|\partial_t \varphi\|_{L^\infty(G)}+\|\partial_t g\|_{L^\infty(G)}<\delta;$
\item[(c)] there exist pairwise disjoint open  sets $G_1,\dots  ,G_N\subset\subset G$ such that
\[
\mbox{$\eta+ (D\varphi,  \partial_t \psi)=\xi_j$ in $G_j$ \, and \,
$|G_j|\ge (1-\delta) \nu_j |G|$}  \quad \forall\, 1\le j\le N;
\]
in particular, $|\cup_{j=1}^N G_j|=\sum_{j=1}^N  |G_j|\ge(1-\delta)|G|.$
\end{itemize}
 \end{theorem}

\begin{proof}
  If $\lambda=1$, then $ \eta=\xi_i=  \sum_{j=1}^N  \nu_j \xi_j,$ where $\nu_i=1$ and $\nu_j=0$ ($j\ne i$). Moreover, (a)-(c) hold with $(\varphi,  \psi,g)\equiv 0,$ where  $G_j=\emptyset$ for all $j\ne i$ in (c).

Now  assume $0\le \lambda<1.$
 Without loss of generality,  we assume $i=1;$ namely, $\eta= \lambda \xi_1 +(1-\lambda) \pi_1.$

 Since $\pi_{k+1}=\chi_k\xi_k+(1-\chi_k)\pi_k$ for all $k \,\mod N$, by Lemma \ref{lem0},  we have  $\pi_1= \sum_{j=1}^N  \nu_{1}^j \xi_j$ and  thus
\[
 \eta= ( \lambda +  (1-\lambda) \nu_{1}^1  )\xi_1+ \sum_{j=2}^N  (1-\lambda)  \nu_{1}^j \xi_j =\sum_{j=1}^N \nu_j \xi_j,
 \]
where $\nu_1^j$ is defined by (\ref{co-eff-0}) with $t_k=\chi_k.$
Note that, with  $\tau=(1-\chi_1)\dots(1-\chi_N),$
\begin{equation}\label{co-eff-1}
\chi_N=\nu_1^{N}(1-\tau), \;\;  \chi_{N-j+1}(1-\chi_{N-j+2})\cdots   (1-\chi_N)  =\nu_{1}^{N-j+1}(1-\tau)\;\; (2\le j\le N).
\end{equation}

Let $G\subset \R^{n+1}$ be a bounded open set  and   $0<\delta<1$.  Assume that $\epsilon \in (0,\delta)$ is a number to be chosen later.

  Since $\gamma= \xi_1-\pi_1  = \kappa_1\gamma_1\in\Gamma,$ we have that $\eta+(1-\lambda) \gamma=\xi_1$ and $\eta-\lambda \gamma=\pi_1.$
Thus, applying Lemma \ref{lem-0} with  $0\le \lambda<1$,  $\gamma \in\Gamma$,  $G\subset \R^{n+1}$ and $\epsilon \in (0,\delta)$,  we obtain  a smooth function  $(\varphi_0,\psi_0,g_0)$ compactly supported in $G_0=G$ and  satisfying (a)--(c) of the lemma.  Hence
\begin{equation}\label{start-0}
\begin{cases}
\mbox{$\dv\psi_0=0$, $\dv g_0=\varphi_0$,  $
\eta+(D\varphi_0,\partial_t\psi_0) \in [\xi_1,\pi_1]_{\epsilon}$ in $G_0;$}\\
\|\varphi_0\|_{L^\infty(G_0)}+\|\partial_t \varphi_0\|_{L^\infty(G_0)}+\|\partial_t g_0\|_{L^\infty(G_0)}<\epsilon;\\
 \eta+(D\varphi_0,\partial_t\psi_0) =\eta+(1-\lambda) \gamma=\xi_1\; \mbox{in $G'_0;$}\\
 \eta+(D\varphi_0,\partial_t\psi_0) =\eta-\lambda \gamma=\pi_1\; \mbox{in $G_0'',$}
\end{cases}
\end{equation}
 where $G_0',G_0''\subset\subset G_0$  are pairwise disjoint open sets with
\begin{equation}\label{start-01}
  |G_0'|\ge (1-\epsilon) \lambda |G_0| \quad\mbox{and}\quad   |G_0''|\ge (1-\epsilon)(1-\lambda) |G_0|.
\end{equation}

 Let  $\tilde G_1=G_{11}=G_{0}''$; then $\eta+(D\varphi_0,\partial_t\psi_0) =\pi_1=\pi_{N+1}={\chi_N}\xi_N+(1-{\chi_N})\pi_N.$ Thus, applying Lemma \ref{lem-0}    with $\lambda= \chi_N\in (0,1)$, $\gamma=\xi_N-\pi_N=\kappa_N\gamma_N\in\Gamma$,   $G=G_{11}$ and $\epsilon\in(0,\delta)$,  we obtain a smooth  function $(\varphi_{11},\psi_{11},g_{11})$ compactly supported in $ G_{11}$ and satisfying (a)--(c) of the lemma. Hence,
  $\pi_1+(D\varphi_{11},\partial_t\psi_{11})  \in [\xi_N,\pi_N]_{\epsilon}$ in $ G_{11}$,
  \[
 \mbox{$\pi_1+(D\varphi_{11},\partial_t\psi_{11})  =\xi_N$ in $ G_{11}'$ \; and \; $\pi_1+(D\varphi_{11},\partial_t\psi_{11})=\pi_N$ in $ G_{11}'',$}
  \]
   where  $G_{11}',G_{11}''\subset\subset G_{11}$ are disjoint  open sets  with
\[
| G_{11}'|\ge (1-\epsilon) {\chi_N} | G_{11}| \quad \mbox{and}\quad | G_{11}''|\ge (1-\epsilon){(1-\chi_N)} | G_{11}|.
\]

Let $ G_{12}= G_{11}''$ in which $\pi_1+(D\varphi_{11},\partial_t\psi_{11})=\pi_N={\chi_{N-1}}\xi_{N-1}+(1-{\chi_{N-1}})\pi_{N-1}.$ Thus,  repeating the above procedure, we obtain  a smooth  function  $(\varphi_{12},\psi_{12},g_{12})$ compactly supported in $G_{12}$ such that
\[
\begin{cases} \pi_N+(D\varphi_{12},\partial_t\psi_{12}) \in [\xi_{N-1},\pi_{N-1}]_{\epsilon} \;\; \text{ in $G_{12}$,}\\
\pi_N+(D\varphi_{12},\partial_t\psi_{12}) =\xi_{N-1} \;\; \text{in $G_{12}'$, and }
\pi_N+(D\varphi_{12},\partial_t\psi_{12}) =\pi_{N-1} \;\; \text{in $G_{12}'',$}\end{cases}
 \]
  where $G_{12}',G_{12}''\subset\subset G_{12}$ are disjoint open sets  with
\[
\begin{split} & |G_{12}'|  \ge (1-\epsilon) {\chi_{N-1}} | G_{12}| \ge (1-\epsilon)^2  {\chi_{N-1}} (1-\chi_N) | G_{11}|,
\\
&| G_{12}''|\ge (1-\epsilon)(1- \chi_{N-1}) | G_{12}|\ge (1-\epsilon)^2  (1-\chi_{N-1})(1-  \chi_N ) | G_{11}|.\end{split}
\]

Continuing  this procedure $N$ times, we  obtain smooth functions  $(\varphi_{1j},\psi_{1j},g_{1j})$  compactly supported in $ G_{1j}= G_{1(j-1)}''$ ($1\le j\le N$), where $G_{10}''=G_0''$,  such that
\begin{equation}\label{scheme-1j}
\begin{cases}
\dv\psi_{1j}=0,\; \dv g_{1j}=\varphi_{1j},  \\
\mbox{$\pi_{N-j+2}+(D\varphi_{1j},\partial_t\psi_{1j}) \in [\xi_{N-j+1},\pi_{N-j+1}]_{\epsilon}$ in $G_{1j};$}\\
\| \varphi_{1j}\|_{L^\infty( G_{1j})}+\|\partial_t   \varphi_{1j}\|_{L^\infty( G_{1j})}+\|\partial_t  g_{1j}\|_{L^\infty( G_{1j})}<\epsilon;\\
\mbox{$\pi_{N-j+2}+(D\varphi_{1j},\partial_t\psi_{1j}) =\xi_{N-j+1}$ in $G_{1j}'$;}\\
\mbox{$\pi_{N-j+2}+(D\varphi_{1j},\partial_t\psi_{1j})=\pi_{N-j+1}$ in $G_{1j}'' $,}
 \end{cases}
\end{equation}
where  $G_{1j}',G_{1j}''\subset\subset G_{1j}$ are disjoint  open sets  with
\begin{equation}\label{sm-0}
\begin{split} & |G_{1j}'|  \ge  (1-\epsilon)^{j}  \chi_{N-j+1}(1-\chi_{N-j+2})\cdots   (1-\chi_N)  | G_{11}|,
\\
&| G_{1j}''|\ge  (1-\epsilon)^{j} (1-\chi_{N-j+1})\cdots   (1-\chi_N)  | G_{11}|.
\end{split}
\end{equation}
We then define  a function  $(\varphi_1,\psi_1,g_1)\in C^\infty_0( \R^{n+1}; \R^m\times  \R^{m\times n}\times  \R^{m\times n})$  by
\[
 \varphi_1=\varphi_{11}+\dots  +\varphi_{1N},\; \;  \psi_1=\psi_{11}+\dots  +\psi_{1N},\;\;\mbox{and}\;\; g_1=g_{11}+\dots  +g_{1N}.
\]
Then   $(\varphi_1, \psi_1, g_1)$ is  compactly supported in $\tilde G_1$ and satisfies
\begin{equation}\label{scheme-0}
\begin{cases}
 \mbox{$\dv \psi_1=0$, $\dv  g_1= \varphi_1$,  $\pi_1+(D \varphi_1,\partial_t \psi_1)\in [\T(\xi_1,\dots  ,\xi_N)]_{\epsilon}$ in $\tilde G_1;$} \\
 \mbox{$\| \varphi_1\|_{L^\infty(\tilde G_1)}+\|\partial_t \varphi_1\|_{L^\infty(\tilde G_1)}+\|\partial_t  g_1\|_{L^\infty(\tilde G_1)}<N\epsilon;$}\\
\mbox{$\pi_1+(D \varphi_1,\partial_t \psi_1) = \xi_{N-j+1}$ in $G_{1j}' \quad\forall\, 1\le j\le N;$}\\
 \mbox{$(D \varphi_1,\partial_t \psi_1) =0$ in $G_{1N}'',$}
 \end{cases}
\end{equation}
where, by (\ref{co-eff-1}) and (\ref{sm-0}),  $G_{11}',\dots  , G_{1N}',G_{1N}''\subset\subset\tilde G_1$ are  pairwise disjoint open and \begin{equation}\label{meas-G-j}
\begin{cases} | G'_{1j}|\ge (1-\epsilon)^j  (1-\tau)\nu_{1}^{N-j+1}  |\tilde G_1| \quad  (1\le j\le N), \\
 | G_{1N}''|\ge  (1-\epsilon)^{N} \tau |\tilde G_1|.
 \end{cases}
\end{equation}

We now iterate  the  scheme (\ref{scheme-0})   inductively to obtain  open sets
\[
\text{$\tilde G_{k}=  G_{ (k-1)N}''$, \; and \,  $G'_{k1},\dots  ,G'_{ kN}, G_{kN}''\subset\subset\tilde G_{k}$ disjoint,}
\]
 and smooth  functions  $(\varphi_{k}, \psi_{k}, g_k)$ compactly supported in  $\tilde G_{k},$ for all $k=1,2,\dots,$ such that
 \begin{equation}\label{scheme-1}
\begin{cases}
\mbox{$\dv \psi_{k} =0$, $\dv g_k= \varphi_{k}$, $\pi_1+(D \varphi_{k},\partial_t \psi_{k})\in [\T(\xi_1,\dots  ,\xi_N)]_{\epsilon}$ in $\tilde G_{k};$}\\
\| \varphi_k\|_{L^\infty(\tilde G_{k})}+\|\partial_t  \varphi_k\|_{L^\infty(\tilde G_{k})}+\|\partial_t  g_k\|_{L^\infty(\tilde G_{k})}<N\epsilon;\\
\mbox{$\pi_1+(D \varphi_{k},\partial_t \psi_{k})= \xi_{N-j+1}$ in $G_{kj}'$ and}  \\
 |G'_{kj}|\ge (1-\epsilon)^j  (1-\tau)\nu_{1}^{N-j+1}  | \tilde G_{k}|
\quad \forall j=1,\dots  ,N;\\
 \mbox{$(D \varphi_{k},\partial_t \psi_{k})=0$ in $G_{kN}''$ and  $|G_{kN}''|\ge (1-\epsilon)^{N} \tau | \tilde G_k|.$}
\end{cases}
\end{equation}

Note that for $k=1,2,\dots ,$
\[
 |\tilde G_k|  =|G_{(k-1)N}''|\ge (1-\epsilon )^{N} \tau |\tilde G_{k-1}|\ge\dots  \ge  (1-\epsilon)^{(k-1)N} \tau^{k-1} |\tilde G_1|;
\]
 hence,    for all $k=1,2,\dots  $ and $1\le j\le N,$
\begin{equation}\label{meas-G}
 |G'_{kj}|   \ge (1-\epsilon)^j  (1-\tau)\nu_{1}^{N-j+1}  | \tilde G_{k}|
   \ge (1-\epsilon)^{j+(k-1)N} (1-\tau)\nu_{1}^{N-j+1}   \tau^{k-1}  |\tilde G_1|.
      \end{equation}

 Let  $\ell\ge 2$  be an integer to be  chosen later.  Define
\[
(\tilde\varphi,\tilde\psi,\tilde g ) =\sum_{k=1}^\ell ( \varphi_{k}, \psi_{k}, g_k) \quad\mbox{and}\quad V_j = \cup_{k=1}^\ell G'_{ kj}  \quad \forall\,1\le j\le N.
\]
Then $(\tilde\varphi,\tilde\psi,\tilde g)$ is smooth and compactly supported in $\tilde G_1$ and satisfies that
\[
\begin{cases}
\mbox{$\dv\tilde\psi=0,\; \dv\tilde g =\tilde\varphi$,  $\pi_1+(D\tilde \varphi,\partial_t\tilde \psi)\in [\T(\xi_1,\dots  ,\xi_N)]_{\epsilon}$ in $\tilde G_{1}$;}
\\
\|\tilde\varphi \|_{L^\infty(\tilde G_{1})}+\|\partial_t \tilde \varphi \|_{L^\infty(\tilde G_1)}+\|\partial_t \tilde g \|_{L^\infty(\tilde G_1)}<N \ell \epsilon;  \\
  \mbox{$\pi_1 +  (D \tilde  \varphi, \partial_t\tilde \psi )= \xi_{N-j+1}$ in $V_j$ for each $1\le j\le N,$}
    \end{cases}
  \]
where, by (\ref{start-01}) and (\ref{meas-G}), for each $ 1\le j\le N,$
\begin{equation}\label{scheme-3}
\begin{split}  |V_j|    &= \sum_{k=1}^\ell  | G'_{ kj} |   \ge \sum_{k=1}^\ell (1-\epsilon)^{j+(k-1)N} (1-\tau)\nu_{1}^{N-j+1}   \tau^{k-1}  |\tilde G_1|
 \\
 & \ge (1-\epsilon)^{ \ell N}  (1-\tau)\nu_{1}^{N-j+1}  \Big (\sum_{k=1}^\ell   \tau^{k-1} \Big)  |\tilde G_1|   \\
 &= (1-\epsilon)^{ \ell N} (1-\tau^{\ell})\nu_{1}^{N-j+1} |\tilde G_1| \\ &\ge  (1-\epsilon)^{1+ \ell N} (1-\tau^{\ell})\nu_{1}^{N-j+1} (1-\lambda) | G|.
  \end{split}
\end{equation}

We now select  a  sufficiently large $\ell\ge 2$ such that
\begin{equation}\label{lg-1}
1-\tau^{\ell }\ge (1-\delta)^{1/2},
\end{equation}
and  a sufficiently small $ \epsilon \in (0,\delta)$ such   that
\begin{equation}\label{sm-1}
(1+N\ell)\epsilon <\delta \quad\mbox{and} \quad  (1-\epsilon)^{1+ \ell N} \ge (1-\delta)^{1/2}.
\end{equation}

Finally, we define
\[
\begin{cases} (\varphi,\psi,g)=(\tilde \varphi,\tilde \psi,\tilde g)+(\varphi_0,\psi_0,g_0),\\
G_1=G'_0 \cup V_N, \;\;\; G_j=V_{N-j+1} \quad  (2\le j\le N).\end{cases}
\]
Then $(\varphi,\psi, g)$ is smooth and compactly supported in $G$ and
\[
\begin{cases}
 \dv\psi=0,\; \dv g=\varphi,\\
 \eta+(D \varphi,\partial_t \psi)\in [\T(\xi_1,\dots  ,\xi_N)]_{\epsilon} \subset [\T(\xi_1,\dots  ,\xi_N)]_{\delta}\;\mbox{ in $G;$}\\
\| \varphi \|_{L^\infty(G)}+\|\partial_t  \varphi \|_{L^\infty(G)}+\|\partial_t  g \|_{L^\infty(G)}<(1+N\ell)\epsilon<\delta;\\
 \mbox{$\eta+ (D \varphi ,\partial_t \psi )  =\xi_j$ in $G_j$ for each $1\le j\le N.$}
  \end{cases}
\]

Moreover, from   (\ref{start-01}), (\ref{scheme-3}), (\ref{lg-1}) and (\ref{sm-1}), it follows that
\[
\begin{split} |G_1| =|G_0'|+|V_N| & \ge (1-\epsilon) \lambda |G|+(1-\epsilon)^{1+\ell N} (1-\tau^{\ell})\nu_{1}^{1} (1-\lambda) |G|\\
&\ge (1-\delta) ( \lambda +(1-\lambda)\nu_1^1)|G| =(1-\delta) \nu_1|G|;
\\
 |G_j| = |V_{N-j+1}| &\ge  (1-\epsilon)^{1+\ell N} (1-\tau^{\ell})\nu_{1}^{j} (1-\lambda) |G|\\
&\ge (1-\delta)  (1-\lambda)\nu_1^j |G| =(1-\delta) \nu_j |G|  \quad \forall\,2\le j\le N.
\end{split}
\]

Therefore, assertions (a)--(c) of the theorem follows.  This completes the proof.
\end{proof}


The following result is interesting in its own right, as it shows that  $\T_N$-configurations are incompatible with the quasimonotonicity condition.

\begin{proposition}  Let  $\sigma\colon \R^{m\times n}\to \R^{m\times n}$ be  continuous with graph
$\K.$

If there exists a $\T_N$-configuration $(\xi_1,\dots  ,\xi_N)$  supported on  $\K$  such that
\[
\T(\xi_1,\dots,\xi_N) \setminus \K\ne \emptyset,
\]
   then $\sigma$ cannot satisfy  the quasimonotonicity condition \eqref{q-mono} for any $\nu>0.$
\end{proposition}

  \begin{proof} Suppose on the contrary that $\sigma$ satisfies condition (\ref{q-mono}) for some $\nu>0.$

  Let $\eta=(A,B)\in \T(\xi_1,\dots,\xi_N)\setminus \K.$
For  each $\delta\in (0,1)$, let $(\varphi_\delta,  \psi_\delta,g_\delta)$ be the function  determined in Theorem \ref{convex-block-thm} with $G=\Omega_T$.
Let $V=\cup_{i=1}^N G_i$ and write $y=(x,t).$
Since  $\sigma (A+D\varphi_\delta)=B+\partial_t \psi_\delta$ on $V$, condition (\ref{q-mono}) implies  that
  \[
\begin{split} \nu\int_{\Omega_T} |D\varphi_\delta|^2dy
&\le   \int_{\Omega_T} \langle \sigma (A+D\varphi_\delta ), D\varphi_\delta \rangle dy\\
&  =   \int_{V} \langle \sigma (A+D\varphi_\delta ), D\varphi_\delta \rangle dy    +   \int_{\Omega_T\setminus V} \langle \sigma (A+D\varphi_\delta ), D\varphi_\delta \rangle dy
\\
&=   \int_{V} \langle B+\partial_t \psi_\delta, D\varphi_\delta \rangle dy    +   \int_{\Omega_T\setminus V} \langle \sigma (A+D\varphi_\delta ), D\varphi_\delta \rangle dy\\
& =   \int_{\Omega_T} \langle B+\partial_t \psi_\delta, D\varphi_\delta \rangle dy    +   \int_{\Omega_T\setminus V} \langle \sigma (A+D\varphi_\delta )-(B+\partial_t\psi_\delta ), D\varphi_\delta \rangle dy\\
& = \int_{\Omega_T\setminus V} \langle \sigma (A+D\varphi_\delta)-(B+\partial_t\psi_\delta), D\varphi_\delta\rangle dy
   \le M |\Omega_T\setminus V|\le C \delta,\end{split}
  \]
 where we have used that $\dv \psi_\delta=0$ on $\Omega_T,$  so that
\[
 \int_{\Omega_T} \langle B+\partial_t \psi_\delta, D\varphi_\delta \rangle dy = \int_{\Omega_T} \langle  \partial_t \psi_\delta, D\varphi_\delta \rangle dy=  \int_{\Omega_T}  \dv \psi_\delta \cdot \partial_t \varphi_\delta\, dy=0,
 \]
 and also that the family $\{[\T(\xi_1,\dots,\xi_N)]_\delta\}_{0<\delta<1}$ is uniformly bounded. Hence,
\[
\lim_{\delta\to 0^+} \int_{\Omega_T} |D\varphi_\delta(y)|^2dy=0,
\]
 which implies
 \[
\lim_{\delta\to 0^+}  \int_{\Omega_T}  |\sigma (A)-(B+\partial_t\psi_\delta)|\, dy=\lim_{\delta\to 0^+}  \int_{\Omega_T}  |\sigma (A+D\varphi_\delta)-(B+\partial_t\psi_\delta)|\, dy=0.
 \]

  On the other hand, since $\int_{\Omega_T}  \partial_t\psi_\delta \,dy=0$,  Jensen's inequality gives
\[
|\sigma (A)- B |=\frac{1}{|\Omega_T|} \int_{\Omega_T}  |\sigma (A)- B |\, dy \le \frac{1}{|\Omega_T|}  \int_{\Omega_T}  |\sigma (A)-(B+\partial_t\psi_\delta)|\, dy,
\]
so that $|\sigma (A)- B |\to 0$ as $\delta\to 0^+.$ Thus  $\sigma(A)=B,$  which contradicts the choice $\eta=(A,B) \notin \mathcal K.$
\end{proof}


\subsection{Condition  $O_N$}

For each $\rho\in\R^{m\times n}\times \R^{m\times n}$ and $r>0$, we write $\rho=(\rho^1,\rho^2),$ where $\rho^1,\rho^2\in \R^{m\times n},$ and denote by  $\B_r(\rho)$ and  $\bar \B_r(\rho)$ the open and  closed  balls in  $\R^{m\times n}\times \R^{m\times n}$  {with center $\rho$ and radius  $r,$}   respectively. In particular, we write $\B_r=\B_r(0)$ and  $\bar\B_r=\bar \B_r(0)$ for all $r>0.$

\begin{definition}\label{O-N}
Let  $\sigma\colon \R^{m\times n}\to \R^{m\times n}$ be  a continuous  function with  graph  $\K.$

\smallskip

(I) We say that $\sigma$ satisfies  \emph{Condition $O_N$}  provided that there exist  numbers  $r_0>0$  and $0<\delta_1<\delta_2<1$  and continuous functions
\begin{equation}\label{fun-x-g}
  (\kappa_i,\gamma_i)    \colon \bar \B_{r_0}   \to [1/\delta_2, 1/\delta_1] \times \Gamma\quad  (1\le i\le N)
\end{equation}
with
\begin{equation}\label{satisfy-g}
\gamma_1(\rho) +\dots+\gamma_N(\rho) =0\quad \forall\, \rho\in \bar \B_{r_0}
\end{equation}
such that letting
\[
\begin{cases}
\pi_1(\rho)=\rho, \;\;
\pi_i(\rho) =\rho+\gamma_1(\rho)+\dots  +\gamma_{i-1}(\rho) \;\;  (2\le i\le N), \\
\xi_i (\rho) =   \pi_i(\rho) +\kappa_i(\rho) \gamma_i(\rho)  \;\;  (1\le i\le N)\quad\forall\, \rho\in \bar \B_{r_0},
\end{cases}
\]
the following properties hold:

\begin{itemize}
\item[(P1)]
\begin{enumerate}
\item[(i)] $\xi_i(\rho)\in \K \quad\forall\, \rho\in\bar \B_{r_0}\;\; (1\le i\le N);$
\item[(ii)] $\xi_i^1(\bar \B_{r_0})\cap \xi^1_j(\bar \B_{r_0})=\pi_i^1(\bar \B_{r_0})\cap \pi^1_j(\bar \B_{r_0})=\emptyset\;\; (1\le i\ne j\le N).$
\end{enumerate}

\vspace{1ex}

\item[(P2)]    For $1\le i\le N$, $0<r\le r_0$ and $0\le \lambda\le 1,$   define
\[
 S^r_i(\lambda) =\{\lambda \xi_i(\rho) +(1-\lambda)\pi_i(\rho): \rho\in \B_{r} \}.
 \]
 Then there exists  $\delta_0\in (  \delta_2,1)$ such that   for each $0<r\le r_0$ and  $\lambda\in  \{0\}\cup [\delta_0,1),$  the sets
$\{S^r_i(\lambda)\}_{i=1}^N$  are open and pairwise disjoint.
\end{itemize}

\vspace{1ex}

(II) Assuming that $\sigma$ satisfies Condition $O_N,$ we introduce several additional sets.

For each  $1\le i\le N,$ $0<r\le r_0$ and $0\le \lambda\le 1,$   define
\begin{equation}\label{new-set-L}
L^r_i(\lambda)=\bigcup_{\substack{(\alpha,\beta)\in S^r_i(\lambda)\times S^r_i(0) \\ \alpha-\beta\in \Gamma}} [\alpha,\beta],
\end{equation}
where $\Gamma$ is the closed wave cone defined by \eqref{set-G}. Also, for each $0<r\le r_0$ and $0\le \lambda\le 1,$ define
\[
 S^r(\lambda)=\bigcup_{i=1}^N S^r_i(\lambda) \quad\text{and} \quad L^r(\lambda)=\bigcup_{i=1}^N L^r_i(\lambda).
\]
Finally,  for each $0<r\le r_0$ and $\delta_0\le \lambda<1,$ define
\[
 \Sigma^r(\lambda)  = \bigcup_{\delta_0\le \lambda' \le \lambda} L^r(\lambda'),
 \]
and take
 \[
 \Sigma(1)=\bigcup_{\delta_0\le \lambda <1}  \Sigma^{r_0}(\lambda).
\]
\end{definition}

 \begin{remark}\label{remk33}    (i)
 Condition $O_N$  substantially simplifies the Condition $(OC)_N$ in \cite{Ya22,Ya23}. In particular, property   (P2)  now requires only  the openness of   $\{S^r_i(\lambda)\}_{i=1}^N$ for $0<r\le r_0$ and   $\lambda\in  \{0\}\cup [\delta_0,1).$ In addition, the sets  $\Sigma^r(\lambda)$   introduced here,  built on  the new sets $L_i^r(\lambda)$ in \eqref{new-set-L} (motivated by   \cite{MRS05}),
differ significantly  from the sets $\Sigma(\lambda)$ in \cite[Definition 2.2]{Ya23}. Moreover, the sets $\Sigma^r(\lambda)$   are automatically  open for all $0<r\le r_0$ and $\delta_0\le \lambda<1$ (see Proposition \ref{prop-cp} below).

 \smallskip

(ii) We record  below several useful observations based on Definition \ref{O-N}.  Let $\sigma\colon \R^{m\times n}\to \R^{m\times n}$ satisfy Condition $O_N$ as above.
\begin{enumerate}
 \item[(a)]  For each $1\le i\le N,$ $0\le\lambda\le 1$ and $\rho\in\bar\B_{r_0},$ let
\begin{equation}\label{chi-i}
 \chi_i(\rho)= {1}/{\kappa_i(\rho)}\in[\delta_1,\delta_2] \quad\text{and}\quad \zeta_i(\lambda,\rho) =\lambda \xi_i(\rho) +(1-\lambda)\pi_i(\rho).
\end{equation}
Then $\pi_{i+1}(\rho)=\chi_i(\rho)\xi_i(\rho)+(1-\chi_i(\rho))\pi_i(\rho).$
Moreover, if $0<r\le r_0$ and $\rho\in \B_r,$  then $\zeta_i(\lambda,\rho)\in S^r_i(\lambda),$ $\pi_i(\rho)\in S^r_i(0),$ and $\zeta_i(\lambda,\rho)-\pi_i(\rho)=\lambda\kappa_i(\rho)\gamma_i(\rho) \in \Gamma;$ hence,
\[
[\zeta_i(\lambda,\rho),\, \pi_i(\rho)]  \subset L^r_i(\lambda).
\]

\item[(b)]  Let $\delta_2<\lambda\le 1,$ $0<r\le r_0$ and $\rho\in \B_{r}.$  Then the points $X_j=\zeta_j(\lambda,\rho)\in S^r_j(\lambda)$ ($1\le j\le N$) satisfy that
\begin{equation}\label{eq-0}
 \pi_i(\rho) =\sum_{j=1}^N  \nu_i^j(\lambda,\rho) X_j \quad \forall \, 1\le i\le N,
\end{equation}
where  the coefficients   $\nu_i^j(\lambda,\rho)=\nu_i^j$ $(1\le i, j\le N)$ are given by formula (\ref{co-eff-0}) with $t_k= {\chi_k(\rho)}/{\lambda}.$  Since ${\delta_1}/{\lambda}\le t_k\le   {\delta_2}/{\lambda}<1$ for each $1\le k\le N,$   it follows that  for each $1\le i\le N,$
\begin{equation}\label{eq-1}
\sum_{j=1}^N  \nu_i^j(\lambda,\rho) =1, \quad\mbox{and}\quad  \nu_i^j(\lambda,\rho)  \ge  (\lambda-\delta_2)^{N-1}\delta_1 \quad \forall\, 1\le  j \le N.
\end{equation}
Furthermore,  $(X_1,\dots,X_N)$ is a $\T_N$ configuration supported on $S^r(\lambda)$ with
\[
\T(X_1,\dots,X_N)\subset L^r(\lambda).
\]
 \end{enumerate}

\end{remark}

The following result establishes additional significant consequences of Condition $O_N$.

\begin{proposition}\label{prop-cp} Let $\sigma$ satisfy Condition $O_N$ as in Definition \ref{O-N}.
\begin{itemize}
\item[(i)]    If $\delta_0\le \lambda<1$ and $0< r\le r_0,$  then the following sets are all open:
\[
 L^r_i(\lambda) \;    (1\le i\le N),\; \; L^r(\lambda),\;\;  \Sigma^r(\lambda) \;\; \text{and} \; \; \Sigma(1).
 \]

 \item[(ii)]  If  $1\le k\le N$ and $0< r<s \le r_0,$ then
\[
\overline{S_k^r(\lambda)}\subset S_k^s(\lambda) \quad\forall\, 0 \le\lambda\le 1, \quad\text{and}\quad \overline{\Sigma^r(\lambda)}\subset \Sigma^s(\lambda) \quad\forall\, \delta_0\le \lambda<1.
\]
\item[(iii)] If $\Delta\subset \Sigma(1)$ is a compact set, then $\Delta\subset \Sigma^{r}(\lambda)$ for some $r\in (0, r_0)$ and $\lambda\in [\delta_0,1).$
\end{itemize}
\end{proposition}

\begin{proof}  1.  Let $1\le i\le N$ and  $A\in L_i^r(\lambda).$ Then $A=q\alpha+(1-q)\beta$ for some $q\in [0,1]$, $\alpha\in S^r_i(\lambda)$ and $\beta\in S^r_i(0)$ with $\alpha-\beta\in \Gamma.$
By property (P2) of Condition $O_N$, both $S^r_i(\lambda)$ and $S^r_i(0)$ are open sets; hence, there exists $\epsilon>0$ such that  $\B_\epsilon(\alpha)\subset S^r_i(\lambda)$ and $\B_\epsilon(\beta)\subset S^r_i(0).$  For any $C\in \B_\epsilon(A),$ let $D=C-A.$  Then
\[
C=A+D=q(\alpha+D)+(1-q)(\beta+D)\in [\alpha+D,\,\beta+D].
\]
Since $\alpha+D\in \B_\epsilon(\alpha)\subset S^r_i(\lambda)$, $ \beta+D \in \B_\epsilon(\beta)\subset S^r_i(0)$ and $(\alpha+D)-(\beta+D)=\alpha-\beta\in \Gamma,$ it follows from \eqref{new-set-L} that
\[
[\alpha+D,\,\beta+D]\subset L^r_i(\lambda).
\]
 Therefore, $\B_\epsilon(A)\subset L^r_i(\lambda),$  proving that $L^r_i(\lambda)$ is open.
From this, the openness of the sets $L^r(\lambda)$, $\Sigma^r(\lambda)$ and $\Sigma(1)$  follows immediately.

2. The first inclusion in (ii) is immediate from the definitions. To prove the second one, let  $\delta_0\le \lambda<1,$  $0< r<s \le r_0,$ and $\bar X\in \overline{\Sigma^r(\lambda)}.$
Choose a sequence $\{X_k\}_{k=1}^\infty$ in $\Sigma^r(\lambda)$ such that $X_k\to \bar X.$ Then $X_k\in L^r_{j_k}(\lambda_k)$ for some $1\le j_k\le N$ and $\delta_0\le\lambda_k\le \lambda;$ hence,
\[
X_k=q_k \alpha_k +(1-q_k)\beta_k
\]
for some  $q_k\in [0,1],$ $\alpha_k\in S^r_{j_k}(\lambda_k)$ and $\beta_k\in S^r_{j_k}(0)$ with $\alpha_k-\beta_k\in \Gamma.$  Assume
\[
\alpha_k=\zeta_{j_k}(\lambda_k, \rho_k) \quad\text{and} \quad \beta_k=\pi_{j_k}(\rho'_k) \quad  (k=1,2,\dots),
\]
where $\rho_k, \rho'_k\in \B_r.$  Let  $k'=\{k_i\}\to\infty$   be a subsequence such that
\[
 j_{k'}\equiv  j\in \{1,\dots,N\}, \;\; \lambda_k\to \lambda'\in [\delta_0,\lambda],\;\; q_{k'}\to q\in [0,1],\;\; \rho_{k'}\to \rho\in \bar\B_r,\;\; \rho'_{k'}\to \rho'\in \bar\B_r .
 \]
Hence, $\alpha_{k'}\to \alpha=\zeta_{j}(\lambda', \rho) \in S^s_j(\lambda')$ and $\beta_{k'}\to \beta=\pi_{j}(\rho') \in S^s_j(0);$ moreover, since $\Gamma$ is closed, $\alpha-\beta\in \Gamma.$
Therefore,
 \[
 \bar X=q\alpha+(1-q)\beta \in [\alpha,\beta] \subset L^s_j(\lambda')\subset  \Sigma^s(\lambda),
 \]
which proves that $\overline{\Sigma^r(\lambda)}\subset \Sigma^s(\lambda).$

3. Let $\Delta\subset \Sigma(1)$ be compact. Since
\[
\Sigma(1)= \bigcup_{\substack{0<r<r_0 \\ \delta_0\le \lambda<1 }} \Sigma^r(\lambda),
\]
the collection  $\{\Sigma^r(\lambda): \, r\in (0, r_0),\; \lambda\in [\delta_0, 1)\}$ forms an open cover of  $\Delta.$ Hence, there exist  $k\ge 1$ and parameters $r_j\in (0, r_0)$,  $\lambda_j \in [\delta_0, 1)$ for $1\le j\le k$   such that
\[
\Delta\subset  \cup_{j=1}^k   \Sigma^{r_j}(\lambda_j) \subset  \Sigma^{r}(\lambda),
\]
where
\[
r=\max_{1\le j\le k}  r_j \in (0,r_0)\quad\mbox{and}\quad \lambda=\max_{1\le j\le k}  \lambda_j \in [\delta_0, 1),
\]
as desired in  (iii).
\end{proof}


   \begin{figure}[ht]
\begin{center}
\begin{tikzpicture}[scale =1]
\draw[thick,dashed] (-5,-2)--(1,-1);
\draw[thick] (-5,-2.2)--(1.1,-0.9);
   \draw[](-5.2,-2) node[below]{$X'_1$};
      \draw[](-4.6,-2.1) node[below]{$X_1$};
    \draw[](0.2,5.9) node[above]{$X'_3$};
        \draw[](0.5,6.6) node[left]{$X_3$};
              \draw[](1,-3.2) node[right]{$X'_5$};
       \draw[](0.4,-3.4) node[right]{$X_5$};
   \draw(3.9,1.6) node[right]{$X_4$};
     \draw[](3.8,2.1) node[right]{$X'_4$};
    \draw[](-6.3,2.25) node[above]{$X'_2$};
        \draw[](-6.8,2.3) node[below]{$X_2$};
             \draw[thick,dashed] (1,-1)--(1,3);
      \draw[thick] (1.1,-0.9)--(0.9,3.1);
  \draw[thick,dashed] (1,-1)--(1,-3);
  \draw[thick] (1.1,-0.9)--(1.205,-3);
  \draw[thick,dashed] (-2,-1.5)--(-5,1);
   \draw[thick,dashed] (-6.5,2.25)--(-5,1);
   \draw[thick,dashed] (-5,1)--(-2,4);
    \draw[ thick] (-5,1.2)--(-2.1,4.1);
     \draw[thick] (-1.95,-1.55)--(-6.3,1.2+2.75*1.3/3.05);
     \draw[thick,dashed] (0,6)--(-2,4);
     \draw[ thick] (-0.1,6.1)--(-2.1,4.1);
    \draw[thick,dashed] (-2,4)--(1,3);
      \draw[ thick] (-2.1,4.1)--(0.9,3.1);
       \draw[thick,dashed] (1,3)--(4,2);
         \draw[thick] (0.9,3.1)--(3.9,2.1);
       \draw(-1.1,5.4) node[left]{$\zeta_3(\lambda',\rho)$};
     \draw[thick] (-2.2,4.3)--(4,1.7);
       \draw[thick] (-2.2,4.3)--(-5.5,1.1);
         \draw[thick] (-2.2,4.3)--(0,4.3+6.4/3);
        \draw[thick ] (-2,-1.8)--(-5.5,1.1);
          \draw[thick ] (-6.5,-1.8+26.1/7)--(-5.5,1.1);
        \draw[thick ] (-2,-1.8)--(0.8,-1.4);
               \draw[thick ] (-2,-1.8)--(-4.8,-2.2);
                \draw[thick ] (0.9,3)--(0.8,-1.4);
                 \draw[thick ] (0.76,-3.16)--(0.8,-1.4);
   \draw[fill] (-5.5,1.1)  circle (0.05);
            \draw(-5.7,1.1) node[below]{$\pi_3(\rho)$};
                    \draw(-4.9,1.2) node[right]{$\pi_3(\rho')$};
            \draw[fill]  (-5,1.2) circle (0.05);
              \draw[ultra thick] (-1.1,5.35)--(-5,1.2);
                \draw[fill]  (-1.1, 5.35) circle (0.05);
         \draw(-3,3.4) node[left]{$Y$};
           \draw[fill]  (-3.05,3.28) circle (0.05);
\end{tikzpicture}
\end{center}
\caption{An illustration for Theorem \ref{lem1} (shown here for the case  $N=5$). \\
The point $Y$ lies on the (ultra-thick) solid line segment $[\zeta_3(\lambda',\rho), \, \pi_3(\rho')].$ \\
The  sets $\T(X_1,\dots,X_N)$ and $\T(X'_1,\dots,X'_N)$ are depicted by the solid  line segments, respectively.
The union of dashed line segments represents  the reference set  $\T(X^0_1,\dots,X^0_N),$ where $X^0_j=\zeta_j(\mu, 0)$ for $1\le j\le N$  (not shown). \\
 The points $X^0_j,$ $ X_j$ and $X'_j$ lie in the open set $S^r_j(\mu)$ for each $1\le j \le N,$  and all  line segments shown lie in the open set   $\Sigma^r(\mu).$ However, only the points $ X_j$ and $X'_j$ ($1\le j\le N$) and the solid line segments  are used in the proof.}
\label{fig2}
\end{figure}
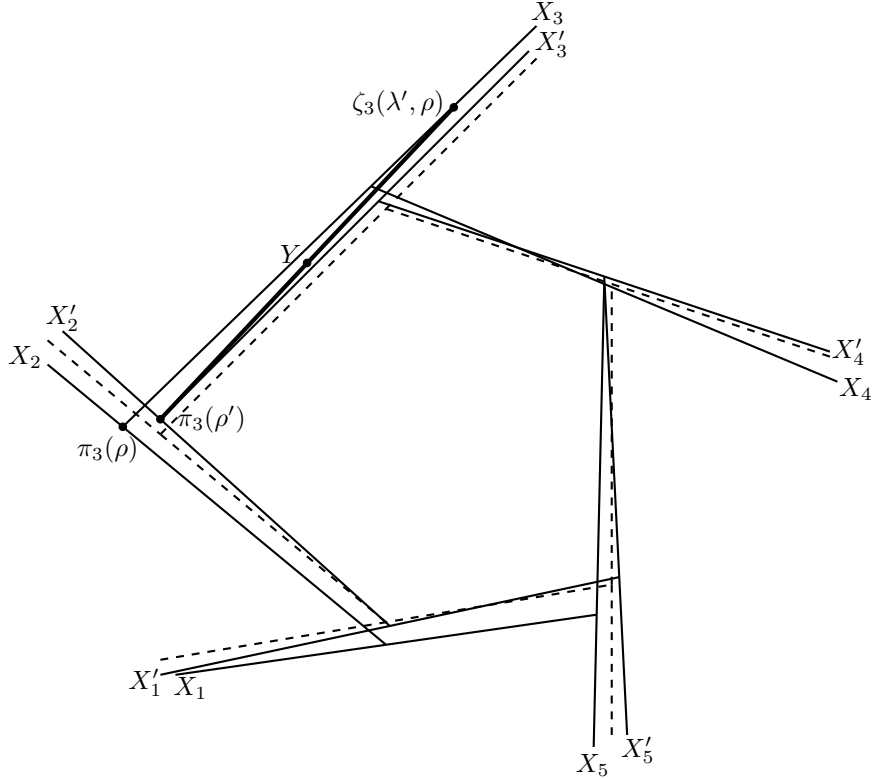


The following theorem provides a key step of construction for convex integration.

\begin{theorem}\label{lem1} Let $\sigma$ satisfy Condition $O_N.$ Let $\delta_0\le \lambda \le \mu<1,$ $0<r\le r_0,$ and  $Y\in \Sigma^r(\lambda).$

Assume that
\[
Y =q\, \zeta_i(\lambda', \rho)+(1-q)\, \pi_i(\rho'),
\]
where $q\in [0,1],$ $1\le i\le N,$ $\delta_0 \le \lambda' \le \lambda,$ and $\rho, \rho' \in \B_r$ with
$
\zeta_i(\lambda', \rho)- \pi_i(\rho')\in\Gamma.
$
Let
\[
 X_j=\zeta_j(\mu, \rho)\quad\mbox{and}\quad X'_j=\zeta_j(\mu, \rho') \quad\forall\, 1\le j\le N.
\]
Then  for any bounded open set $G\subset \R^{n+1}$  and  $0<\tau<1,$ there exists  a function $(\varphi, \psi,g)\in C_0^\infty (G;\R^m \times \R^{m\times n}\times \R^{m\times n})$   with  the following properties:
\begin{itemize}
\item[(a)] $\dv\psi=0,\; \dv g=\varphi$ and $Y+(D\varphi,  \partial_t \psi)\in \Sigma^r(\mu)$ in $G;$
\item[(b)]  $\|\varphi\|_{L^\infty(G)}+\|\partial_t \varphi\|_{L^\infty(G)}+\|\partial_t g\|_{L^\infty(G)}<\tau;$
\item[(c)]   there exist pairwise disjoint open  sets $G'_1,\dots  ,G'_N,  G''_1,\dots, G''_N\subset\subset G$ such that
\begin{equation}\label{eq-d}
\begin{cases}
 Y+(D\varphi,  \partial_t\psi)=\chi_{G'_j} X_j+\chi_{G''_j} X_j'\in S_j^r(\mu) \;\; \text{in $G_j=G'_j\cup G''_j$} & \forall\, 1\le j\le N,\\[2mm]
|G_i|\ge  (1-\tau) \big [(1-q) \nu_i^i(\mu,\rho') +q  \big (\frac{\lambda'}{\mu} +(1-\frac{\lambda'}{\mu})  \nu_i^i(\mu,\rho)\big )\big  ] |G|,\\[1mm]
|G_j| \ge (1-\tau)  \big [(1-q) \nu_i^j(\mu,\rho') +q  (1-\frac{\lambda'}{\mu})  \nu_i^j(\mu,\rho) \big ] |G| &\forall\, j\ne i;
\end{cases}
\end{equation}
in particular, $\sum_{j=1}^N  |G_j|\ge (1-\tau)|G|.$
\end{itemize}
 \end{theorem}

\begin{proof}  1. Since $[\zeta_i(\lambda', \rho),\, \pi_i(\rho')]\subset L^r_i(\lambda')\subset \Sigma^r(\mu)$ and $\Sigma^r(\mu)$ is open, there exists   $\epsilon_1>0$ such that
\[
[\zeta_i(\lambda', \rho),\, \pi_i(\rho')]_{\epsilon_1}\subset \Sigma^r(\mu).
\]
In what follows, let $ 0< \delta<1$ be a number to be determined later.

Apply Lemma \ref{lem-0} with  $\gamma=\zeta_i(\lambda', \rho)- \pi_i(\rho'),$ $\lambda=q$  and $\epsilon=\delta$ to obtain a function $(\varphi_1, \psi_1,g_1)\in C_0^\infty (G;\R^m \times \R^{m\times n}\times \R^{m\times n})$   with  the following properties:
\begin{equation}\label{eq-d-0}
\begin{cases}
\text{$\dv\psi_1=0,$ $\dv g_1=\varphi_1$  and $Y+(D\varphi_1,  \partial_t \psi_1)\in [\zeta_i(\lambda', \rho),\, \pi_i(\rho')]_\delta$ in $G;$}\\[1ex]
\text{$\|\varphi_1\|_{L^\infty(G)}+\|\partial_t \varphi_1\|_{L^\infty(G)}+\|\partial_t g_1\|_{L^\infty(G)}<\delta;$}\\[1ex]
\text{there exist disjoint open  sets $G', G''\subset\subset G$ such that}\\
$$
\begin{cases}
\mbox{$Y+(D\varphi_1,  \partial_t\psi_1)=\zeta_i(\lambda', \rho)$\, in $G'$,} &|G'|\ge  (1-\delta) q\, |G|,\\
\mbox{$Y+(D\varphi_1,  \partial_t\psi_1)=\pi_i(\rho')$\, in $G''$,} & |G''|\ge  (1-\delta)  (1-q) \,|G|.
\end{cases}
$$
\end{cases}
\end{equation}

2. If $G''=\emptyset$ (necessarily, $q=1$),  then we omit this step. Assume $G''\ne \emptyset.$ Note that by \eqref{eq-0},
\[
\pi_i(\rho') = \sum_{j=1}^N  \nu_i^j(\mu,\rho') X'_j.
\]
Since  $(X'_1,\dots ,X'_N)$ is  a $\T_N$-configuration with $\T(X'_1,\dots ,X'_N) \subset  \Sigma^r(\mu),$  by the openness of  $\Sigma^r(\mu),$  there exists $\epsilon_2>0$ such that
\[
  [\T(X'_1,\dots ,X'_N)]_{\epsilon_2} \subset \Sigma^r(\mu).
\]

Apply Theorem \ref{convex-block-thm}  with $(\xi_1,\ldots,\xi_N)=(X'_1,\dots ,X'_N)$ and  $\eta=\pi_i(\rho')$ to obtain  a function $(\varphi_2, \psi_2,g_2)\in C_0^\infty (G'';\R^m \times \R^{m\times n}\times \R^{m\times n})$   with  the following properties:
\begin{equation}\label{eq-d-1}
\begin{cases}
\text{$\dv\psi_2=0,\;  \dv g_2=\varphi_2$ and $\pi_i(\rho')+(D\varphi_2,  \partial_t \psi_2)\in [\T(X'_1,\dots ,X'_N)]_\delta$ in $G'';$}\\[1ex]
\text{$\|\varphi_2\|_{L^\infty(G'')}+\|\partial_t \varphi_2\|_{L^\infty(G'')}+\|\partial_t g_2\|_{L^\infty(G'')}<\delta;$}\\[1ex]
\text{there exist pairwise disjoint open  sets $G''_1,\dots  ,G''_N\subset\subset G''$ such that}\\
$$
\begin{cases}
\mbox{$\pi_i(\rho')+(D\varphi_2,  \partial_t\psi_2)=X'_j$ in $G''_j$}\quad \forall\, 1\le j\le N,\\
|G''_j|\ge  (1-\delta)   \nu^j_i(\mu,\rho') \,|G''|   \quad\forall\, 1\le j\le N.
\end{cases}
$$
\end{cases}
\end{equation}

3. If $G'=\emptyset$  (necessarily, $q=0$),  then we skip this step.  Assume $G'\ne \emptyset$.  Note that by \eqref{eq-0},
\[
\zeta_i(\lambda',\rho) =  \left (\frac{\lambda'}{\mu} +\Big (1-\frac{\lambda'}{\mu}\Big )  \nu^i_i(\mu,\rho)\right )X_i+\sum_{1\le j\le N,\, j\ne i} \Big ( 1-\frac{\lambda'}{\mu}\Big ) \nu_i^j(\mu,\rho) X_j.
\]
   Since  $(X_1,\dots ,X_N)$ is  a $\T_N$-configuration with $\T(X_1,\dots ,X_N) \subset  \Sigma^r(\mu),$  by the openness of  $\Sigma^r(\mu),$  there exists $\epsilon_3>0$ such that
   \[
    [\T(X_1,\dots ,X_N)]_{\epsilon_3} \subset \Sigma^r(\mu).
    \]

Now apply Theorem \ref{convex-block-thm} with $(\xi_1,\ldots,\xi_N)=(X_1,\dots ,X_N)$  and $\eta=\zeta_i(\lambda',\rho)$ to  obtain  a function $(\varphi_3, \psi_3,g_3)\in C_0^\infty (G';\R^m \times \R^{m\times n}\times \R^{m\times n})$   with  the following properties:
\begin{equation}\label{eq-d-2}
\begin{cases}
\text{$\dv\psi_3=0,\;  \dv g_3=\varphi_3$ and $\zeta_i(\lambda',\rho)+(D\varphi_3,  \partial_t \psi_3)\in [\T(X_1,\dots ,X_N)]_\delta$ in $G';$}\\[1ex]
\text{$\|\varphi_3\|_{L^\infty(G')}+\|\partial_t \varphi_3\|_{L^\infty(G')}+\|\partial_t g_3\|_{L^\infty(G')}<\delta;$}\\[1ex]
\text{there exist pairwise disjoint open  sets $G'_1,\dots  ,G'_N\subset\subset G'$ such that}\\
$$
\begin{cases}
\mbox{$\zeta_i(\lambda',\rho)+(D\varphi_3,  \partial_t\psi_3)=X_j$ in $G'_j$}\quad \forall\, 1\le j\le N,\\
|G'_i|\ge  (1-\delta) \big (\frac{\lambda'}{\mu} +(1-\frac{\lambda'}{\mu})  \nu_i^i(\mu,\rho)\big ) |G'|,\\
|G'_j| \ge  (1-\delta)   ( 1-\frac{\lambda'}{\mu})  \nu_i^j(\mu,\rho) \,|G'| \quad\forall\, j\ne i.
\end{cases}
$$
\end{cases}
\end{equation}

4. Finally, let
\[
 (\varphi,\psi,g)= (\varphi_1,\psi_1,g_1)\chi_{G} + (\varphi_2,\psi_2,g_2)\chi_{G''} +(\varphi_3,\psi_3,g_3)\chi_{G'} .
\]
Then  $(\varphi, \psi,g)\in C_0^\infty (G;\R^m \times \R^{m\times n}\times \R^{m\times n})$ and, by \eqref{eq-d-0}-\eqref{eq-d-2}, we have the following properties:
\[
  \begin{cases}  \dv\psi=0,\; \dv g=\varphi\;\;  \text{and}\\
Y+(D\varphi,  \partial_t \psi)  \in [\zeta_i(\lambda', \rho),\, \pi_i(\rho')]_\delta\cup [\T(X'_1,\dots ,X'_N)]_\delta\cup [\T(X_1,\dots ,X_N)]_\delta \quad\text{in $G$;}  \\[1ex]

\|\varphi\|_{L^\infty(G)}+\|\partial_t \varphi\|_{L^\infty(G)}+\|\partial_t g\|_{L^\infty(G)}<2\delta;\\[1ex]
$$
\begin{cases}
 Y+(D\varphi,  \partial_t\psi)=\chi_{G'_j} X_j+\chi_{G''_j} X_j' \in S_j^r(\mu) \;\; \text{in $G_j=G'_j\cup G''_j$} & \forall\, 1\le j\le N,\\
|G_i|\ge  (1-\delta)^2 \big [(1-q) \nu_i^i(\mu,\rho') +q  \big (\frac{\lambda'}{\mu} +(1-\frac{\lambda'}{\mu})  \nu_i^i(\mu,\rho)\big )\big ] |G|,\\
|G_j| \ge (1-\delta)^2 \big [(1-q) \nu_i^j(\mu,\rho') +q  (1-\frac{\lambda'}{\mu})  \nu_i^j(\mu,\rho) \big ] |G| & \forall\, j\ne i.
\end{cases}
$$
\end{cases}
\]

Therefore, all requirements (a)--(c) of the theorem will follow  if $0<\delta<1$ is chosen to satisfy
\[
0<\delta<\min\left \{\epsilon_1, \, \epsilon_2,\, \epsilon_3,\, \frac{\tau}{2},\, 1-\sqrt{1-\tau}\right \}.
\]
\end{proof}


\section{Proof of Theorem \ref{mainthm1}}\label{s-4}

 Let  $G\subset \R^d$ be  a bounded open set and let $f\colon G\to \R^q$ be a measurable function, where $d$ and $q$ are positive integers.
For $x_0\in G$,  the  {\em essential oscillation}  of $f$ at $x_0$ is defined by
\[
\omega_f(x_0)=\inf\big\{\|f(x)-f(y)\|_{L^\infty(U\times U)} : U\subset G, \;  \mbox{$U$  open,} \;  x_0\in  U \big\}.
\]
We say that $f$ is {\em essentially continuous} at $x_0$  if $\omega_f(x_0)=0.$
Clearly, if $f$ is continuous at  $x_0,$ then it is essentially continuous there; however, the converse need not hold.

A {\em cube}   in $\R^{d}$   always means an open cube whose  sides are parallel to the coordinate axes. For $y\in\R^d$ and $l>0$, we denote by $Q=Q_{y,l}$   the  cube centered at  $y$ with side length $2l$ and write  $l=\rad(Q).$

\subsection{Stage Theorem  for Convex Integration}

We now establish the main stage theorem needed for the proof of Theorem \ref{mainthm1} in the framework of convex integration.

 \begin{theorem}\label{thm1} Let  $m,n\ge 1$ and $N\ge 2.$ Assume that  $\sigma\colon\R^{m\times n}\to \R^{m\times n} $ satisfies  Condition $O_N.$   Let $G\subset \R^{n+1}$ be a bounded open set and $(u,v)\in C^1(\bar G;\R^m\times\R^{m\times n})$ be such that
\begin{equation}\label{strict-sub}
u=\dv v\quad\mbox{and}\quad (Du,\partial_t v)\in \Sigma^r(\lambda) \;\; \; \mbox{on $\bar G,$}
\end{equation}
where  $0<r<r_0$ and $\delta_0\le \lambda<1.$

Then for any  $\lambda<\mu<1,$  $r<s<r_0$  and  $0<\epsilon<1,$  there exist a function $(\tilde u,\tilde v)\in (u,v)+C^{\infty}_{0}(G;\R^m\times\R^{m\times n})$ and finitely many disjoint cubes $Q_1,\ldots,Q_M\subset G$ with $|G\setminus\cup_{j=1}^M Q_j|<\epsilon|G|$ satisfying the following:
\begin{itemize}
\item[(a)] $0<\mathrm{rad}(Q_j)<\epsilon\;\;\forall\,1\le j\le M;$
\item[(b)] $\tilde u=\dv \tilde v$ and  $(D\tilde u,  \partial_t \tilde v)\in \Sigma^{s}(\mu) $ on  $\bar G;$
\item[(c)]  $\|\tilde u-u\|_{L^\infty(G)}+ \|\partial_t \tilde u-\partial_t  u\|_{L^\infty(G)}<\epsilon;$
\item[(d)] $|Q_j\cap\{(D\tilde u,\partial_t \tilde v) \in S^{s}(\mu) \}| \ge (1-\epsilon)|Q_j|\;\; \forall\, 1\le j\le M;$
\item[(e)] $|\{(D\tilde u,\partial_t \tilde v) \in S^{s}(\mu) \}| \ge (1-\epsilon)|G|;$
\item[(f)]  $\begin{cases}
  |\{(D\tilde u,\partial_t \tilde v) \in S_k^{s}(\mu) \}| \ge \frac12 (\mu- \lambda) (\mu-\delta_2)^{N-1}\delta_1 |G|,\\
 |\{(D\tilde u,\partial_t \tilde v) \in S_k^{s}(\mu) \}| \ge \frac{\lambda}{\mu} |\{(D u,\partial_t  v) \in  S_k^r(\lambda) \}| \end{cases}\forall\, 1\le k\le N;$
 \item[(g)] $\|D\tilde u-Du\|_{L^1(G)}\le C_0\big[ |F_0|+   (\epsilon+(\mu- \lambda) ) |G|\big ],$ where
 \[
 F_0=\{(D u,\partial_t  v) \notin S^{r}(\lambda)\}
 \] and  $C_0>0$  is a constant depending only on  $\delta_0$ and the diameter of the set $\Sigma(1).$
\end{itemize}
  \end{theorem}

\begin{proof}
{\em Step 1.} By Proposition \ref{prop-cp}(ii), let $d'>0$ denote the minimum of
\[
\mbox{$\mathrm{dist}(\Sigma^{r}(\mu),\partial\Sigma^{s}(\mu))$ \, and \, $\min\limits_{1\le k\le N}\mathrm{dist}(S^{r}_k(\mu),\partial S_k^{s}(\mu))$.}
\]
Define
\[
\begin{cases} G_k= \{y=(x,t)\in G: (D u(y),\partial_t  v(y)) \in S_k^{r}(\lambda)\} \;\;  (1\le k\le N),\\
 I=\{1\le k\le N:  G_k\ne \emptyset\}. \end{cases}
\]
By (P2) of Condition $O_N$,  the sets $\{G_k\}_{k\in I}$  are nonempty, open  and pairwise disjoint.

For each $k\in I,$ let $G_k'\subset G_k$ be a nonempty open set such that
\begin{equation}\label{set-G'}
  |\partial G_k'|=0\quad \mbox{and} \quad |G_k\setminus G_k'|\le \epsilon'\,|G_k|,
\end{equation}
where  $0<\epsilon'<1$ is  to be determined  later, and define
\[
G'_0= G\setminus \overline{\cup_{k\in I}  G_k'}= G\setminus \cup_{k\in I}  \overline{G_k'}.
\]
Then   the sets $\{G'_k\}_{k\in I\cup \{0\}}$ are  open and pairwise disjoint,  and
\begin{equation}\label{null-G'}
 |G\setminus\cup_{k\in I\cup \{0\}} G_k'|=0.
\end{equation}

\smallskip

{\em Step 2.}  Let $\bar y\in G;$ then $Y=Y_{\bar y} =(Du (\bar y),\partial_t  v(\bar y)) \in \Sigma^{r}(\lambda)$ so that  $Y\in L_i^{r}(\lambda')$ for some (perhaps not unique)  $1\le i=i_{\bar y} \le N$ and $\delta_0\le\lambda'=\lambda'_{\bar y}\le \lambda.$ Thus,
\begin{equation}\label{Y-form}
Y=q\zeta_{i}(\lambda',\rho) +(1-q)\pi_{i}(\rho')
\end{equation}
for some  $0\le q=q_{\bar y} \le 1$ and $\rho=\rho_{\bar y}, \, \rho'=\rho'_{\bar y} \in \B_{r}$  with $\zeta_{i}(\lambda',\rho) -\pi_{i}(\rho')\in \Gamma.$

In particular, if $\bar y\in G_k'$ for some  $k\in I,$  then $Y \in S^{r}_{k}(\lambda)$; thus, $Y=\zeta_k(\lambda,\rho)$ for some $\rho\in\B_r.$  In this case, we set $q_{\bar y}=1,$  $i_{\bar y}=k,$ $\lambda'_{\bar y}=\lambda$ and $ \rho_{\bar y}= \rho'_{\bar y} =\rho.$

Let  $X_j=\zeta_j(\mu,\rho_{\bar y})$ and $X'_j=\zeta_j(\mu,\rho'_{\bar y})$ for all $1\le j\le N.$
We  apply Theorem \ref{lem1}  with $Y=Y_{\bar y}$ and  $\tau=\epsilon'$  on  the cube $Q=Q_{\bar y,l},$  where $0<l<\ell_{\bar y}:=\sup\{l>0 : Q_{\bar y,l}\subset G\},$  to obtain  a function $(\varphi, \psi,g)=(\varphi_{\bar{y},l}, \psi_{\bar{y},l},g_{\bar{y},l})\in C^\infty_0(Q;\R^m\times\R^{m\times n}\times\R^{m\times n})$  such that
\begin{itemize}
\item[(i)] $\dv\psi=0,$ $\dv g=\varphi$ and $Y+(D\varphi,  \partial_t \psi)\in \Sigma^{r}(\mu)$ on $\bar Q;$
\item[(ii)]    $\|\varphi\|_{L^\infty(Q)}+\|\partial_t \varphi\|_{L^\infty(Q)}+\|\partial_t g\|_{L^\infty(Q)} <\epsilon';$
\item[(iii)]  there exist pairwise disjoint open  sets $P'_1,\dots, P'_N,   P''_1,\dots, P''_N\subset\subset Q$ such that
\begin{equation}\label{eq-d2}
\begin{cases}
 Y+(D\varphi,  \partial_t\psi)=\chi_{P'_j} X_j+\chi_{P''_j} X_j' \quad \text{in $P_j=P'_j\cup P''_j$} & \forall\, 1\le j\le N,\\
|P_i|\ge  (1-\epsilon') \big [(1-q) \nu_i^i(\mu,\rho') +q  \big (\frac{\lambda'}{\mu} +(1-\frac{\lambda'}{\mu})  \nu_i^i(\mu,\rho)\big )\big  ] |Q|,\\
|P_j| \ge (1-\epsilon')  \big [(1-q) \nu_i^j(\mu,\rho') +q  (1-\frac{\lambda'}{\mu})  \nu_i^j(\mu,\rho) \big ] |Q| &\forall\, 1\le j\le N.
\end{cases}
\end{equation}
In particular,  $\sum_{j=1}^N |P_j|\ge  (1-\epsilon')|Q|.$
\end{itemize}

By (i) and (iii), we have
\begin{equation}\label{dist-1}
\begin{cases} \dist(Y+(D\varphi,  \partial_t \psi),\partial\Sigma^{s}(\mu))\ge d' &\mbox{on $\bar Q,$}  \\
\dist (Y+(D\varphi,  \partial_t \psi),\partial S_j^{s}(\mu))\ge d' &\mbox{in $P_j$} \quad \forall\, 1\le j\le N.
\end{cases}
\end{equation}
Moreover, since $\delta_0\le\lambda'\le\lambda$ and $1-\frac{q\lambda'}{\mu} \ge \mu-\lambda,$ it follows from (\ref{eq-1}) that
\begin{equation}\label{eq-e}
\begin{split}
|P_i| &\ge  (1-\epsilon') \Big [ \frac{q\lambda'}{\mu} +\Big (1-\frac{q\lambda'}{\mu}\Big )  (\mu-\delta_2)^{N-1}\delta_1 \Big  ] |Q|,   \\
|P_j|
& \ge (1-\epsilon')  (\mu-\lambda)  (\mu-\delta_2)^{N-1}\delta_1 |Q| \qquad \forall\, 1\le j  \le N.
\end{split}
\end{equation}

Define  $\tilde u  =u_{\bar y,l}=  u + \varphi$ and $\tilde  v  =v_{\bar y,l}=  v +  \psi +g $ on $\bar Q=\bar Q_{\bar y,l}.$  Then \begin{equation}\label{new-fun0}
\begin{cases}
(\tilde u,\tilde v) \in  (u,v) +C_0^\infty(Q;\R^m\times\R^{m\times n}),\\
  \mbox{$\tilde u= \dv \tilde  v$  \;on $\bar Q,$}\\
  \|\tilde u-u\|_{L^\infty(Q)}+\|\partial_t \tilde   u -\partial_t  u\|_{L^\infty(Q)} < \epsilon'.
  \end{cases}
\end{equation}

By the uniform continuity of $(Du,\partial_t v)$ on $\bar G,$  select a number $\ell'>0$ such that
\[
 |D u(y)-D u(y')|+|\partial_t  v(y)-\partial_t  v(y')|<\min\Big\{\frac{d'}{4},\epsilon'\Big\} \quad \forall\, y,y'\in \bar G,\;\;
 |y-y'|<\ell'.
 \]
Let $0<l<\min\{\ell_{\bar y},\ell'\}$.  Then for all $y\in \bar Q=\bar Q_{\bar y,l},$
\[
\begin{split} |(D\tilde  u(y), & \partial_t \tilde v(y)) - (Y+(D\varphi(y),  \partial_t\psi(y)) )  | \\
& \le |D u(y)-D u(\bar y)|+|\partial_t  v(y)-\partial_t  v(\bar y)|+|\partial_t  g(y)| < \min\Big\{\frac{d'}{4},\epsilon'\Big\} +\epsilon'.
\end{split}
\]

We now select $\epsilon'>0$ to satisfy
\begin{equation}\label{choice-e'-1}
\epsilon'<\min\Big\{\frac{d'}{4},\frac{\epsilon}{2}\Big\}.
\end{equation}
Then by  (i),  (iii) and (\ref{dist-1}), for all $0<l<\min\{\ell_{\bar y},\ell'\}$ with $Q=Q_{\bar{y},l}$,
\begin{equation}\label{cubes}
 \begin{cases}
 (D u, \partial_t v)\big |_{\bar Q} \in \B_\epsilon(Y), \\
  (D\tilde  u, \partial_t \tilde  v)\big |_{\bar Q} \in \Sigma^{s}(\mu),  \\
  (D\tilde  u, \partial_t \tilde  v)\big |_{P_j} \in  S_j^{s}(\mu)\cap [\B_\epsilon (X_j) \cup \B_\epsilon (X'_j)]   \quad  \forall\, 1 \le j \le N,\\
  | Q\cap \{(D\tilde u,\partial_t\tilde v)\in S^{s}(\mu)\}|\ge |\cup_{j=1}^N P_j|\ge (1-\epsilon')|Q|.
  \end{cases}
\end{equation}

\smallskip

{\em Step 3.}   Let $k\in I_0:=I\cup\{0\}$. Choose a sequence $\{Q^k_\nu\}_{\nu=1}^\infty$ of disjoint cubes in $G_k'$ with $0<\mathrm{rad}(Q^k_\nu)<\min\{\ell',\epsilon\}$ for all $\nu\ge1$ such that
\[
\big|G_k'\setminus\cup_{\nu=1}^\infty Q^k_\nu\big|=0.
\]
Pick a large integer $m_k\ge1$ so that
\begin{equation}\label{choice-Qs}
\big|G_k'\setminus\cup_{\nu=1}^{m_k} Q^k_\nu\big|\le\epsilon'|G_k'|.
\end{equation}
For each  $\nu\ge 1,$ let $\bar y^k_\nu$ denote the center of $Q^k_\nu$ and $l^k_\nu=\mathrm{rad}(Q^k_\nu)$.

Following the constructions in Step 2, we define
\begin{equation}\label{final-fun}
(\tilde  u,\tilde v)  =    \sum_{k\in I_0} \sum_{\nu=1}^{m_k} (u_{\bar{y}^k_\nu,l^k_\nu}, v_{\bar{y}^k_\nu,l^k_\nu}) \chi_{ Q^k_\nu}+ (u,v)\chi_{G\setminus G'}\quad\mbox{in $G$},
\end{equation}
where $G'=\bigcup_{k\in I_0}\bigcup_{\nu=1}^{m_k}  Q^k_\nu.$
Then $(\tilde u,\tilde v) \in  (u,v) +C_0^\infty(G;\R^m\times\R^{m\times n})$ and
\[
|G\setminus G'|=\sum_{k\in I_0}|G_k'\setminus\cup_{\nu=1}^{m_k} Q^k_\nu|\le\epsilon'|G|<\epsilon|G|.
\]

We relabel the cubes $Q^k_\nu$ $(k\in I_0,\,1\le\nu\le m_k)$ as $Q_1,\ldots,Q_M.$ Then requirement (a) is satisfied.
From (\ref{strict-sub}), (\ref{new-fun0}), (\ref{choice-e'-1}), (\ref{cubes}) and (\ref{final-fun}), requirements (b) and (c) also follow.

Moreover, using (\ref{null-G'}), (\ref{choice-e'-1}), (\ref{cubes}) and (\ref{choice-Qs}), we have
\[
\begin{split}
 |\{(D\tilde u,\partial_t\tilde v) &\in S^{s}(\mu)\}|  = \sum_{k\in I_0}|G_k' \cap\{(D\tilde u,\partial_t\tilde v)\in S^{s}(\mu)\}|\\
& = \sum_{k\in I_0}\sum_{\nu=1}^\infty|Q^k_\nu \cap\{(D\tilde u,\partial_t\tilde v)\in S^{s}(\mu)\}|  \ge \sum_{k\in I_0}\sum_{\nu=1}^{m_k}|Q^k_\nu \cap\{(D\tilde u,\partial_t\tilde v)\in S^{s}(\mu)\}|\\
& \ge(1-\epsilon')\sum_{k\in I_0}\sum_{\nu=1}^{m_k}| Q^k_\nu|  =(1-\epsilon')\sum_{k\in I_0}\big(|G_k'|- |G'_k\setminus \cup_{\nu=1}^{m_k} Q^k_\nu|\big)\\
& \ge (1-\epsilon')^2 \sum_{k\in I_0}|G_k'|=(1-\epsilon')^2|G|>(1-\epsilon)|G|.
\end{split}
\]
Hence, requirements (d) and (e) are satisfied.

\smallskip

{\em Step 4.}  We   now verify  requirement  (f).

First,  from the second  of  (\ref{eq-e}) and third of (\ref{cubes}), we have that  for each $1\le j \le N,$
\[
 \begin{split}
  |\{(D\tilde u,\partial_t \tilde v)   \in S_j^{s}(\mu)\}|  & \ge \sum_{k\in I_0}\sum_{\nu=1}^{m_k}|Q^k_\nu \cap\{(D\tilde u,\partial_t\tilde v)\in S_j^{s}(\mu)\}|  \\
  & \ge (1-\epsilon')  (\mu-\lambda)(\mu-\delta_2)^{N-1}\delta_1 \sum_{k\in I_0} \sum_{\nu=1}^{m_k} |Q_\nu^k| \\
  & \ge  (1-\epsilon')^2    (\mu-\lambda)(\mu-\delta_2)^{N-1}\delta_1  |G|  \\ &  \ge  \frac12 (\mu-\lambda)(\mu-\delta_2)^{N-1}\delta_1  |G|, \end{split}
\]
where $\epsilon'\in(0,1)$ is chosen to satisfy, in addition to (\ref{choice-e'-1}),
\begin{equation}\label{choice-e'-2}
 (1-\epsilon')^2\ge 1/2.
\end{equation}
This verifies  the first inequality of (f).

Next, let $1\le k\le N.$ If $k\notin I,$ then $G_k=\{y\in G: (Du(y),v_t(y))\in S^r_k(\lambda)\}=\emptyset;$  thus,  the second inequality of (f)  is automatically satisfied.
Now assume $k\in I;$ then
\[
\bar{y}^k_\nu\in Q^k_\nu\subset G_k'\subset G_k\ne\emptyset\qquad (1\le \nu\le m_k).
\]
Hence,  in  \eqref{Y-form} of  Step 2, we have
$q=1$, $i=k,$   $\lambda'=\lambda$ and $\rho=\rho'\in\B_r.$ From (\ref{set-G'}),  the first of \eqref{eq-e},  the third of (\ref{cubes}),  and (\ref{choice-Qs}), it follows that
 \[
 \begin{split} |\{(D\tilde u,\partial_t \tilde v)\in S_k^{s}(\mu)\}|
 &  \ge   \sum_{j\in I_0}\sum_{\nu=1}^{m_j}  | Q_\nu^j \cap \{(D  \tilde u,\partial_t  \tilde v)\in S_k^{s}(\mu)\}|   \\
 & \ge   \sum_{\nu=1}^{m_k}  | Q_\nu^k \cap \{(D  u_{\bar{y}^k_\nu,l^k_\nu},\partial_t  v_{\bar{y}^k_\nu,l^k_\nu})\in S_k^{s}(\mu)\}|  \\
 &  \ge  (1-\epsilon') \Big [\frac{\lambda}{\mu} +\Big (1-\frac{\lambda}{\mu}\Big )(\mu-\delta_2)^{N-1} \delta_1\Big ] \sum_{\nu=1}^{m_k} |Q_\nu^k|\\
 &   \ge (1-\epsilon')^2 \Big [ \frac{\lambda}{\mu} +\Big (1-\frac{\lambda}{\mu}\Big )(\mu-\delta_2)^{N-1}\delta_1 \Big ] |G_k'|\\
 &  \ge (1-\epsilon')^3 \Big  [ \frac{\lambda}{\mu} +\Big (1-\frac{\lambda}{\mu}\Big )(\mu-\delta_2)^{N-1}\delta_1 \Big ]  |G_k|. \end{split}
 \]
Therefore,  the second inequality of (f) is ensured  if   $\epsilon'\in (0,1)$ is chosen to further satisfy, along with (\ref{choice-e'-1}) and (\ref{choice-e'-2}), that
 \begin{equation}\label{small1}
 (1-\epsilon')^3 \Big [ \frac{\lambda}{\mu} +\Big (1-\frac{\lambda}{\mu}\Big )(\mu-\delta_2)^{N-1}\delta_1 \Big ]\ge \frac{\lambda}{\mu},
 \end{equation}
   which is possible since  $\big (1-\frac{\lambda}{\mu}\big )(\mu-\delta_2)^{N-1}\delta_1>0.$

\smallskip

{\em Step 5.}  Finally, we verify  requirement (g). In the following, we denote by $C$ any constant  depending  on  $\delta_0$ and the diameter of the set $\Sigma(1).$
Note from (\ref{set-G'}), (\ref{choice-e'-1}) and (\ref{choice-Qs})  that
 \begin{equation}\label{f0}
 \begin{split}
  \|D\tilde u -Du\|_{L^1(G)}  = &\,  \|D\tilde u-Du\|_{L^1(F_0)} + \sum_{k\in I}  \|D\tilde u-Du\|_{L^1(G_k\setminus G_k')}
\\
& \, + \sum_{k\in I}  \|D\tilde u-Du\|_{L^1(G_k'\setminus \cup_{\nu=1}^{m_k}Q^k_\nu)} + \sum_{k\in I}\sum_{\nu=1}^{m_k}  \|D\tilde u-Du\|_{L^1(Q^k_\nu)}\\
\le &\, C(|F_0| + \epsilon'|G|)  + \sum_{k\in I}\sum_{\nu=1}^{m_k} \int_{Q^k_\nu} |(D\tilde u,\partial_t \tilde v)-(D u,\partial_t  v)|
 \\
 \le &\, C(|F_0| + \epsilon |G|)  + \sum_{k\in I} \sum_{\nu=1}^{m_k}\int_{ P^k_{\nu}}|(D\tilde u,\partial_t \tilde v) -(D u,\partial_t  v)|
 \\
&  \, + \sum_{k\in I}\sum_{\nu=1}^{m_k} \int_{Q^k_\nu\setminus  P^k_{\nu}}|(D\tilde u,\partial_t \tilde v)-(D u,\partial_t  v)|, \end{split}
\end{equation}
where  $P^k_{\nu}\subset\subset Q_\nu^k=Q_{\bar y^k_\nu,l^k_\nu}$ ($k\in I$,  $1\le \nu\le m_k$) are the sets defined as in Step 2 with
\[
Y= Y^k_\nu= Y_{\bar{y}^k_\nu}= (Du(\bar{y}^k_\nu),\partial_t v(\bar{y}^k_\nu)) \in S_k^r(\lambda),
\]
 $Q=Q^k_\nu,$ and $P_i=P^k_\nu;$ hence,  in \eqref{Y-form}, we have   $q=1$, $i=k$,  $\lambda'=\lambda$ and $\rho=\rho'=\rho_{\bar{y}^k_\nu}\in \B_r.$
By (\ref{small1}), we have
\[
  (1-\epsilon')  \Big [ \frac{\lambda}{\mu} +\Big (1-\frac{\lambda}{\mu}\Big ) (\mu-\delta_2)^{N-1}\delta_1 \Big ]  >  \frac{\lambda}{\mu},
\]
which, by  the first of \eqref{eq-e},  implies that $| P^k_{\nu}|>\frac{\lambda}{\mu}| Q_\nu^k|.$  Hence,
\[
 |Q^k_\nu\setminus P_\nu^k| <\frac{\mu-\lambda}{\mu}|Q^k_\nu|< \frac{\mu-\lambda}{\delta_0}|Q^k_\nu|,
\]
so that
\begin{equation}\label{f2}
  \int_{Q^k_\nu\setminus  P^k_{\nu}}  |(D\tilde u, \partial_t \tilde v) -(D u,\partial_t  v)|  \le C|Q^k_\nu\setminus  P_\nu^k|\le C(\mu-\lambda)|Q^k_\nu|.
\end{equation}
Moreover, let  $X^k_{\nu}=\zeta_k(\mu,\rho) \in S_k^{r}(\mu)$ be defined as in Step 2 corresponding to  $Y=Y^k_\nu$  with $\rho=\rho_{\bar{y}^k_\nu}\in \B_r.$  Then
  $|X_\nu^k-Y^k_\nu|=(\mu-\lambda)|\xi_k(\rho_{\bar{y}^k_\nu})-\pi_k(\rho_{\bar{y}^k_\nu})|\le C(\mu-\lambda);$  thus, by (\ref{cubes}),  we have
\begin{equation}\label{f1}
 \begin{split}
 \int_{ P_\nu^k}  |(D\tilde u,&\partial_t \tilde v)  -(D u,\partial_t  v)| \\
 &  \le  \int_{ P^k_\nu} \big (|(D\tilde u,\partial_t \tilde v)-X_\nu^k|+|X_\nu^k-Y^k_\nu|+|(D u,\partial_t  v)-Y^k_\nu|\big)\\
 & \le C(\epsilon+(\mu-\lambda))|Q^k_\nu|. \end{split}
  \end{equation}
Inserting (\ref{f2}) and (\ref{f1}) into (\ref{f0}) proves  requirement (g).

The proof is now complete.
\end{proof}


\subsection{Proof of Theorem \ref{mainthm1}}

Assume that $\sigma\colon\R^{m\times n}\to \R^{m\times n}$ is  locally Lipschitz  and satisfies Condition $O_N$ for some  $N\ge 2.$
Also, assume that $(\bar u, \bar v)\in C^1(\bar \Omega_T;\R^m\times\R^{m\times n})$ satisfies
\[
\bar u=\dv \bar v \quad{\mbox{and}}\quad (D\bar u,\partial_t \bar v)\in \Sigma(1) \;\;\mbox{ on  $\, \bar \Omega_T.$}
\]
Since  the set $\Delta=(D\bar u,\partial_t \bar v)(\bar\Omega_T)\subset \Sigma(1)$ is compact, by  Proposition \ref{prop-cp}(iii),
there exist $0< r_1 <r_0$ and  $ \delta_0 \le  \lambda_1 <1$ such that
\begin{equation}\label{subs-2}
\Delta=(D u_0,\partial_t v_0)(\bar\Omega_T) \subset  \Sigma^{r_1}(\lambda_1),
\end{equation}
where $(u_0,v_0):=(\bar u,\bar v)$ in $\Omega_T.$

Let $0<\delta<1.$ Define  the sequences $\{\lambda_\nu\}_{\nu=2}^\infty$,  $\{r_\nu\}_{\nu=2}^\infty$  and  $\{\epsilon_\nu\}_{\nu=1}^\infty$   as follows:
\begin{equation}\label{select-1}
 \begin{cases}  \lambda_{\nu+1}= \frac12(1+\lambda_{\nu}),\\
 r_{\nu+1}=\frac12(r_0+r_{\nu}),\\
   \epsilon_{\nu}=\ \delta/3^\nu,    \end{cases}  \quad \forall\, \nu=1,2,\ldots.
\end{equation}
Then $\delta_0\le \lambda_1<  \lambda_{2}<\cdots<1,$ $0<r_1<r_{2}<\cdots < r_0,$
\[
\lim_{\nu\to\infty}\lambda_\nu=1,\quad \mbox{and}    \quad   \lim_{\nu\to\infty}r_\nu =r_0.
\]

\subsubsection{Construction Scheme for \eqref{pdr1}}  We are now in a position to carry out the construction scheme for the differential relation \eqref{pdr1}, as outlined in the Introduction. To this end, we choose the open sets
\[
\mathcal U_\nu =\Sigma^{r_{\nu+1}}(\lambda_{\nu+1}) \quad  (\nu=1,2,\dots)
\]
and will use them to construct  the corresponding functions $(u_\nu,   v_\nu)$  as required in   \eqref{pdr2}.

We begin with \eqref{subs-2}. Applying the  main stage theorem for convex integration, Theorem \ref{thm1}, to  $(u,v)=(\bar u,\bar v)$ on the set $G=\Omega_T$, with
\[
 \lambda=\lambda_1,\; \;  \mu=\lambda_2,\;\;  r=r_1,\;\; s=r_2 \;\;\mbox{and}\;\;   \epsilon=\epsilon_1,
\]
we obtain   a  function $(u_1,v_1)= (\tilde u,\tilde v)\in  (u_0,v_0)+C^{\infty}_0(\Omega_T;\R^m\times \R^{m\times n})$ and finitely many  disjoint cubes $Q^1_1,\ldots,Q^1_{M_1}\subset\Omega_T$ with $|\Omega_T\setminus\cup_{j=1}^{M_1}Q^1_{j}|<\epsilon_1|\Omega_T|$
such that
\begin{equation}\label{u1}
\begin{cases}
0<\rad(Q^1_j)<\epsilon_1\;\;\forall\, 1\le j\le M_1;  \\
\mbox{$u_1=\dv v_1$ and $(D u_1,  \partial_t v_1)\in \Sigma^{r_2}(\lambda_2)$  on $\bar \Omega_T$;}\\
 \|u_1-u_0\|_{L^\infty(\Omega_T)} +\|\partial_t u_1-\partial_t u_0\|_{L^\infty(\Omega_T)}<\epsilon_1;  \\
|Q^1_j\cap \{(D u_1,\partial_t v_1) \in S^{r_2}(\lambda_2) \}| \ge (1-\epsilon_1)|Q^1_j| \quad \forall\, 1\le j\le M_1;\\
|\{(D u_1,\partial_t v_1) \in S^{r_2}(\lambda_2)\}| \ge (1-\epsilon_1)|\Omega_T|; \\
\mbox{$
\begin{cases}
 |\{(D u_1,\partial_t v_1) \in S_k^{r_2}(\lambda_2)\}| \ge \frac12 (\lambda_2-\lambda_1)(\lambda_2-\delta_2)^{N-1}\delta_1|\Omega_T|, \\
|\{(D u_1,\partial_t v_1) \in S_k^{r_2}(\lambda_2) \}| \ge \frac{\lambda_1}{\lambda_2} |\{(D u_0,\partial_t v_0) \in S_k^{r_1}(\lambda_1) \}|\;\; \forall\, 1\le k\le N;
\end{cases} $} \\
\|Du_1 - Du_0\|_{L^1(\Omega_T)}\le C_0\big[|F_0|+(\epsilon_1+(\lambda_2-\lambda_1))|\Omega_T|\big],\\
\mbox{where $F_0=\{(D u_0,\partial_t v_0) \not\in S^{r_1}(\lambda_1) \}$ and $C_0>0$ is the constant in Theorem \ref{thm1}(g).}
\end{cases}
\end{equation}
Here, we assume further that
\begin{equation}\label{u1-0}
C_0\ge\mathrm{diam}(\Sigma(1)).
\end{equation}
Let $Q^1_0=\Omega_T\setminus\cup_{j=1}^{M_1}\bar{Q}^1_j.$ Then $Q^1_0$ is an open subset of $\Omega_T$ with
\begin{equation}\label{u1-1}
\mbox{$|Q^1_0|<\epsilon_1|\Omega_T|$ \; and \; $|\Omega_T\setminus\cup_{j=0}^{M_1}Q^1_j|=0$.}
\end{equation}

Fix any  $0\le i \le M_1.$ We  apply Theorem \ref{thm1} to  $(u,v)=(u_1,v_1)$ on the set $G=Q^1_i$, with
\[
 \lambda=\lambda_2,\; \;  \mu=\lambda_3,\;\; r=r_2,\;\; s=r_3 \;\;\mbox{and} \;\; \epsilon=\epsilon_2,
 \]
to obtain a function $(u_i^2,v_i^2)= (\tilde u,\tilde v)\in (u_1,v_1)+C^{\infty}_0(Q^1_i;\R^m\times \R^{m\times n})$ and finitely many disjoint cubes $Q^2_{i,1},\ldots,Q^2_{i,M^2_i}\subset Q^1_i$ with $|Q^1_i\setminus\cup_{j=1}^{M^2_i}Q^2_{i,j}|<\epsilon_2|Q^1_i|$
such that
\begin{equation}\label{u2}
\begin{cases}
0<\rad(Q^2_{i,j})<\epsilon_2\;\;\forall\,1\le j\le M^2_i;  \\
\mbox{$u^2_i=\dv v^2_i$ and $(D u^2_i,  \partial_t v^2_i)\in \Sigma^{r_3}(\lambda_3)$  on $\bar Q^1_i$;}\\
\|u^2_i-u_1\|_{L^\infty(Q^1_i)} +\|\partial_t u^2_i-\partial_t u_1\|_{L^\infty(Q^1_i)}<\epsilon_2; \\
|Q^2_{i,j}\cap \{(D u^2_i,\partial_t v^2_i) \in S^{r_3}(\lambda_3) \}| \ge (1-\epsilon_2)|Q^2_{i,j}| \quad \forall\, 1\le j\le M^2_i;\\
|Q^1_i\cap\{(D u^2_i,\partial_t v^2_i) \in S^{r_3}(\lambda_3) \}| \ge (1-\epsilon_2)|Q^1_i|; \\
\begin{cases}
 |Q^1_i\cap\{(D u^2_i,\partial_t v^2_i) \in S_k^{r_3}(\lambda_3)\}| \ge \frac{1}{2}(\lambda_3-\lambda_2)(\lambda_3-\delta_2)^{N-1}\delta_1|Q^1_i|, \\
|Q^1_i\cap\{(D u^2_i,\partial_t v^2_i) \in S_k^{r_3}(\lambda_3) \}| \ge \frac{\lambda_2}{\lambda_3} |Q^1_i\cap\{(D u_1,\partial_t v_1) \in S_k^{r_2}(\lambda_2) \}|
\end{cases} (1\le k\le N);\\
\|Du^2_i - D u_1\|_{L^1(Q^1_i)}\le C_0\big[|F^1_i|+(\epsilon_2+(\lambda_3-\lambda_2))|Q^1_i|\big], \\
 \mbox{where $F^1_i=Q^1_i\cap\{(D u_1,\partial_t v_1) \not\in S^{r_2}(\lambda_2) \}.$ }\\
\end{cases}
\end{equation}
Here, by the fourth  of (\ref{u1}), we have  $|F^1_i|<\epsilon_1|Q^1_i|$ if $1\le i\le M_1;$ while if $i=0,$ then,  by (\ref{u1-1}),
\begin{equation}\label{u2-1}
|F^1_0|\le |Q^1_0|<\epsilon_1|\Omega_T|.
\end{equation}
Let $Q^{2}_{i,0}=Q^1_i\setminus\cup_{j=1}^{M^2_i}\bar{Q}^2_{i,j}$; then $Q^2_{i,0}$ is an open subset of $Q^1_i$ with
\[
\mbox{$|Q^{2}_{i,0}|<\epsilon_2|Q^1_i|$ \; and \;  $|Q^1_i\setminus\cup_{j=0}^{M^2_i} Q^2_{i,j}|=0$.}
\]

We  relabel the  open sets  $Q^2_{i,j}$ $(0\le i\le M_1,\,0\le j\le M^2_i)$ as  $Q^2_0,Q^2_1,\ldots,Q^2_{M_2}$ by setting
\[
M_2=\sum_{i=0}^{M_1}(M^2_i+1)-1 \quad\text{and}\quad   Q^2_{\mu}=Q^2_{i,j} \;\;
\Big(\mu=i+j+\sum_{k=1}^i M^2_{k-1}\Big),
\]
for all $0\le i\le M_1$ and $0\le j\le M^2_i;$ here we adopt the convention that $\sum_{k=1}^0 M^2_{k-1}=0.$
Note that all   sets $Q^2_\mu$, except for those with  $\mu =i+\sum_{k=1}^i M^2_{k-1}$ for $0\le i\le M_1$, are  cubes  in $\R^{n+1}$.  Moreover,
\[
\bigg|\Omega_T\setminus\bigcup_{j=0}^{M_2}Q^2_j\bigg|=0, \quad \bigg|Q^1_i\setminus\bigcup_{j=0}^{M^2_i}Q^2_{i+j+\sum_{k=1}^i M^2_{k-1}}\bigg|=0,\quad \big|Q^2_{i+\sum_{k=1}^i M^2_{k-1}}\big|<\epsilon_2|Q^1_i|\;\;(0\le i\le M_1).
\]

Define
\[
(u_2,v_2)=\sum_{i=0}^{M_1}(u^2_i,v^2_i)\chi_{Q^1_i}\in (u_0,v_0)+C^{\infty}_0(\Omega_T;\R^m\times \R^{m\times n}).
\]
Then
\[
\begin{cases}
0<\rad(Q^2_{j})<\epsilon_2\;\;\mbox{$\forall\, 0\le j\le M_2$ with $j\ne i+\sum_{k=1}^i M^2_{k-1}$ $(0\le i\le M_1)$;}  \\
\mbox{$u_2=\dv v_2$ and $(D u_2,  \partial_t v_2)\in \Sigma^{r_3}(\lambda_3)$  on $\bar \Omega_T$;}\\
\|u_2-u_1\|_{L^\infty(\Omega_T)} +\|\partial_t u_2-\partial_t u_1\|_{L^\infty(\Omega_T)}<\epsilon_2; \\
|Q^2_{j}\cap \{(D u_2,\partial_t v_2) \in S^{r_3}(\lambda_3) \}| \ge (1-\epsilon_2)|Q^2_{j}| \\
\quad \forall\, 0\le j\le M_2 \;\;  \mbox{with $j\ne i+\sum_{k=1}^i M^2_{k-1}$ $(0\le i\le M_1)$;}\\
|\{(D u_2,\partial_t v_2) \in S^{r_3}(\lambda_3) \}| \ge (1-\epsilon_2)|\Omega_T|; \\
\begin{cases}
 |Q^1_i\cap\{(D u_2,\partial_t v_2) \in S_k^{r_3}(\lambda_3) \}| \ge \frac{1}{2}(\lambda_3-\lambda_2)(\lambda_3-\delta_2)^{N-1}\delta_1|Q^1_i|, \\
|Q^1_i\cap\{(D u_2,\partial_t v_2) \in S_k^{r_3}(\lambda_3) \}| \ge \frac{\lambda_2}{\lambda_3} |Q^1_i\cap\{(D u_1,\partial_t v_1) \in S_k^{r_2}(\lambda_2) \}|,\\
\mbox{for each $1\le k\le N$ and $0\le i\le M_1$};
\end{cases} \\
\|Du_2 - D u_1\|_{L^1(\Omega_T)}\le  C_0\big[2\epsilon_1+\epsilon_2+(\lambda_3-\lambda_2)\big]|\Omega_T|.
\end{cases}
\]
Here, the last inequality follows from  (\ref{u1-0}), (\ref{u2}) and (\ref{u2-1}) as
\[
\begin{split}
\|Du_2 - D u_1 & \|_{L^1(\Omega_T)}=\sum_{j=0}^{M_1} \|Du_2 - D u_1\|_{L^1(Q^1_j)}= \|Du_2 - D u_1\|_{L^1(Q^1_0)} +\sum_{j=1}^{M_1} \|Du_2 - D u_1\|_{L^1(Q^1_j)}\\
& \le C_0 |Q^1_0|+\sum_{j=1}^{M_1} C_0 \big[|F^1_j|+(\epsilon_2+(\lambda_3-\lambda_2))|Q^1_j|\big] \\
& \le C_0\epsilon_1|\Omega_T|+C_0[ \epsilon_1+\epsilon_2+(\lambda_3-\lambda_2)]  \sum_{j=1}^{M_1}  |Q^1_j|  \le  C_0 \big[2\epsilon_1+\epsilon_2+(\lambda_3-\lambda_2)\big]|\Omega_T|.
\end{split}
\]

Repeating this process indefinitely, we obtain a sequence $\{(u_\nu,v_\nu)\}_{\nu=1}^\infty$ in the space
\[
(u_0,v_0)+C^{\infty}_0(\Omega_T;\R^m\times \R^{m\times n})
\]
and finitely many disjoint open sets $Q^\nu_0,Q^\nu_1,\ldots,Q^\nu_{M_\nu}\subset\Omega_T$ with $|\Omega_T\setminus\cup_{j=0}^{M_\nu}Q^\nu_j|=0$ $(\nu=1,2,\ldots)$ such that
for each $\nu\ge2$,
\[
M_\nu=\sum_{i=0}^{M_{\nu-1}}(M^\nu_i +1)-1
\]
and all  sets $Q^\nu_j$, except for those with  $j =i+\sum_{k=1}^i M^\nu_{k-1}$ for $0\le i\le M_{\nu-1}$, are  cubes in $\R^{n+1}$ and that for each $\nu\ge2$ and $0\le i\le M_{\nu-1},$
\begin{equation}\label{u-nu-0}
\Big |Q^{\nu-1}_i \setminus\bigcup_{j=0}^{M^\nu_i} Q^{\nu}_{i+j+\sum_{k=1}^i M^\nu_{k-1}} \Big |=0 \quad\mbox{and}\quad \Big |Q^{\nu}_{i+\sum_{k=1}^i M^\nu_{k-1}}\Big |<\epsilon_\nu|Q^{\nu-1}_i|.
\end{equation}
Moreover, for each $\nu\ge2$,
\begin{equation}\label{u-nu}
\begin{cases}
 0<\rad(Q^{\nu}_{j})<\epsilon_{\nu}\;\;\mbox{$\forall\, 0\le j\le M_{\nu}$ with $j\ne i+\sum_{k=1}^i M^\nu_{k-1}$ $(0\le i\le M_{\nu-1})$;}  \\
 \mbox{$u_{\nu}=\dv v_{\nu}$ and $(D u_{\nu},  \partial_t v_{\nu})\in \Sigma^{r_{\nu+1}}(\lambda_{\nu+1})$  on $\bar \Omega_T$;}\\
 \|u_{\nu}-u_{\nu-1}\|_{L^\infty(\Omega_T)} +\|\partial_t u_{\nu}-\partial_t u_{\nu-1}\|_{L^\infty(\Omega_T)}<\epsilon_{\nu}; \\
 |Q^{\nu}_{j}\cap \{(D u_{\nu},\partial_t v_{\nu}) \in S^{r_{\nu+1}}(\lambda_{\nu+1}) \}| \ge (1-\epsilon_{\nu})|Q^{\nu}_{j}| \\
 \mbox{\;\;$\forall\, j=0,\ldots,M_{\nu}$ with $j\ne i+\sum_{k=1}^i M^\nu_{k-1}$ $(0\le i\le M_{\nu-1})$;}\\
 |\{(D u_{\nu},\partial_t v_{\nu}) \in S^{r_{\nu+1}}(\lambda_{\nu+1}) \}| \ge (1-\epsilon_{\nu})|\Omega_T|; \\
 \begin{cases}
 |Q^{\nu-1}_i\cap\{(D u_{\nu},\partial_t v_{\nu}) \in S_k^{r_{\nu+1}}(\lambda_{\nu+1}) \}| \ge \frac{1}{2}(\lambda_{\nu+1}-\lambda_{\nu})(\lambda_{\nu+1}-\delta_2)^{N-1}\delta_1|Q^{\nu-1}_i|, \\
|Q^{\nu-1}_i\cap\{(D u_{\nu},\partial_t v_{\nu}) \in S_k^{r_{\nu+1}}(\lambda_{\nu+1}) \}| \ge \frac{\lambda_{\nu}}{\lambda_{\nu+1}} |Q^{\nu-1}_i\cap\{(D u_{\nu-1},\partial_t v_{\nu-1}) \in S_k^{r_{\nu}}(\lambda_{\nu}) \}|,\\
\mbox{for each $1\le k\le N$ and $0\le i\le M_{\nu-1};$}
\end{cases} \\
 \|Du_{\nu} - D u_{\nu-1}\|_{L^1(\Omega_T)}\le  C_0 \big[2\epsilon_{\nu-1}+\epsilon_{\nu}+(\lambda_{\nu+1}-\lambda_{\nu})\big]|\Omega_T|.
\end{cases}
\end{equation}

\begin{lemma}\label{lem512} For  all $p>q\ge 1$, $0\le j\le M_q$ and  $1\le k\le N,$  we have
\[
 |Q^q_j  \cap\{( Du_{p} ,\partial_t v_{p})  \in S_k^{r_{p+1}}(\lambda_{p+1}) \}|  \ge  \frac12\, \frac{\lambda_{q+2}}{\lambda_{p+1}}  (\lambda_{q+2}- \lambda_{q+1})(\lambda_{q+2}-\delta_2)^{N-1}\delta_1\,|Q^q_j|.
\]
\end{lemma}
\begin{proof}
From the construction of  $\{(u_{p},v_{p})\}_{p=1}^\infty$ and  $\{Q^{p}_j\}_{p\ge1,\,0\le j\le M_p}$ above, we  see that  if $p>q\ge 1$ and $0\le j\le M_q$, then
\[
Q^q_j=(\cup_{k\in I} Q^p_k)\cup E
\]
 for some index set $I\subset\{0,1,\ldots, M_p\}$ and  null-set $E.$
Thus,  for all $p>q\ge 1$, $0\le j\le M_q$ and $1\le k\le N,$   by  (\ref{u-nu}), we have
\[
\begin{split} |Q^q_j  \cap &\{( Du_{p} ,\partial_t v_{p})  \in S_k^{r_{p+1}}(\lambda_{p+1}) \}| \\
&  \ge \frac{\lambda_p}{\lambda_{p+1}}\, |Q^q_j\cap\{(Du_{p-1},\partial_t  v_{p-1}) \in S_k^{r_{p}}(\lambda_p) \}|\\
& \ge \frac{\lambda_p}{\lambda_{p+1}} \frac{\lambda_{p-1}}{\lambda_{p}}\,|Q^q_j\cap\{(Du_{p-2},\partial_t  v_{p-2}) \in S_k^{r_{p-1}}(\lambda_{p-1}) \}|\\
&\hspace{30ex}  \vdots \\
& \ge \frac{\lambda_p}{\lambda_{p+1}}  \frac{\lambda_{p-1}}{\lambda_{p}}   \cdots  \frac{\lambda_{q+2}}{\lambda_{q+3}}\, |Q^q_j\cap\{(Du_{q+1},\partial_t  v_{q+1}) \in S_k^{r_{q+2}}(\lambda_{q+2}) \}| \\
& =\frac{\lambda_{q+2}}{\lambda_{p+1}}\, |Q^q_j\cap\{(Du_{q+1},\partial_t  v_{q+1}) \in S_k^{r_{q+2}}(\lambda_{q+2}) \}| \\
& \ge    \frac12 \, \frac{\lambda_{q+2}}{\lambda_{p+1}}  (\lambda_{q+2}- \lambda_{q+1})(\lambda_{q+2}-\delta_2)^{N-1}\delta_1\,|Q^q_j|.
\end{split}
\]
\end{proof}

\begin{lemma} \label{lem52} Let $f(\xi)= \sigma(\xi^1)-\xi^2 $ for $\xi=(\xi^1,\xi^2)\in \R^{m\times n}\times \R^{m\times n}.$
Then
\[
 |f(\xi)|\le C\,(1-\lambda) \quad \forall\, 0\le \lambda\le 1,\;\; \xi \in S^{r_0}(\lambda),
 \]
 where $C>0$ is a constant.
\end{lemma}
\begin{proof}
Since  $f$ is locally Lipschitz, we have $|f(\xi)-f(\eta)|\le  L|\xi-\eta|$ for all $\xi, \eta \in \overline{\Sigma(1)}$, where  $L>0$ is a constant.

Let $0\le\lambda\le1$ and $\xi\in S^{r_0}(\lambda)\subset\overline{\Sigma(1)}$. Then $\xi=\lambda\xi_i(\rho)+(1-\lambda)\pi_i(\rho)$ for some $1\le i\le N$ and $\rho\in \B_{r_0}.$ Since  $f(\xi_i(\rho))=0,$ we have that
\[
 |f(\xi)|=|f(\xi)-f(\xi_i(\rho))| \le L|\xi-\xi_i(\rho)|=L(1-\lambda)|\xi_i(\rho)-\pi_i(\rho)|\le C(1-\lambda).
\]
\end{proof}

\subsubsection{Completion of  the Proof}

1.  By  the third  of  both (\ref{u1}) and  (\ref{u-nu}),   for each $j\ge 1$,
 \begin{equation}\label{norm-0}
 \begin{split}
& \|u_j  -u_0\|_{L^\infty(\Omega_T)}   + \| \partial_t u_j -\partial_t  u_0\|_{L^\infty(\Omega_T)}\\
 & \le \sum_{\nu=0}^{j-1} \left [\|u_{\nu+1}-u_\nu\|_{L^\infty(\Omega_T)}   + \|\partial_t u_{\nu+1} -\partial_t  u_\nu\|_{L^\infty(\Omega_T)} \right ]   \le \sum_{\nu=0}^\infty {\epsilon_{\nu+1}}  =\frac12 \delta. \end{split}
\end{equation}
Hence, $\{u_\nu\}$ and $\{\partial _t u_\nu\}$ are uniformly bounded on $\Omega_T$. By the second  of \eqref{u-nu},  $\{Du_\nu\}$ is also uniformly bounded on $\Omega_T$. Consequently, the sequence $\{u_\nu\}$ is bounded in $\bar u+W_0^{1,\infty}(\Omega_T;\R^m).$

Moreover, by  the third and last of  (\ref{u-nu}), the sequence $\{u_\nu\}$  is also Cauchy  in $W^{1,1}(\Omega_T;\R^m).$ Hence, there exists
 $u\in W^{1,1}(\Omega_T;\R^m)$ such that
\[
\mbox{$u_\nu\to u$ in $W^{1,1}(\Omega_T;\R^m)$ as $\nu\to\infty.$}
\]
Since $\{u_\nu\}$ is bounded in $\bar u+W_0^{1,\infty}(\Omega_T;\R^m),$  it follows  that $u\in \bar u+ W^{1,\infty}_{0}(\Omega_T;\R^m).$
Furthermore, passing to the limit in \eqref{norm-0} yields
 \[
 \|u  -\bar u\|_{L^\infty(\Omega_T)}   + \| \partial_t u -\partial_t  \bar u\|_{L^\infty(\Omega_T)}  \le \sum_{\nu=0}^\infty {\epsilon_{\nu+1}}  =\frac12 \delta <\delta.
\]

 2. By   (\ref{u-nu})  and Lemma \ref{lem52},  we have
\[
\begin{split}
 \int_{\Omega_T} |f(Du_{\nu}, \partial_t v_{\nu})|  \,dxdt  = & \, \int_{(D u_{\nu},\partial_t v_{\nu}) \in S^{r_{\nu+1}}(\lambda_{\nu+1}) } |f(Du_{\nu}, \partial_t v_{\nu})| \,dxdt  \\
&\, + \int_{(D u_{\nu},\partial_t v_{\nu}) \notin S^{r_{\nu+1}}(\lambda_{\nu+1})} |f(Du_{\nu}, \partial_t v_{\nu})| \,dxdt \\
\le & \,  C[(1-\lambda_{\nu+1})   +\epsilon_{\nu} ]|\Omega_T| \to 0\quad\text{ as $\nu\to \infty.$}
\end{split}
 \]
Observe  from  $u_{\nu}=\dv v_{\nu}$ in $\Omega_T$ that
\begin{equation}\label{pre-weak}
 \int_{\Omega_T}   \big (u_{\nu} \cdot \partial_t \varphi   - \langle \sigma(Du_{\nu}) , D\varphi \rangle \big ) \,dxdt   = -\int_{\Omega_T} \langle f(Du_{\nu}, \partial_t v_{\nu}), D\varphi \rangle\,dxdt
\end{equation}
for all $\varphi\in  C^\infty_0(\Omega_T;\R^m).$
Since
\[
\|u_{\nu}-u\|_{W^{1,1}(\Omega_T)}\to 0\quad\mbox{and}\quad \|f(Du_{\nu}, \partial_t v_{\nu})\|_{L^1(\Omega_T)} \to 0\quad\mbox{as $\nu\to\infty$,}
\]
and  $\{\|Du_{\nu}\|_{L^\infty(\Omega_T)}\}$ is bounded, by taking $\nu\to\infty$ in (\ref{pre-weak}), we have that
\[
  \int_{\Omega_T} \big (u  \cdot \partial_t \varphi  - \langle \sigma(Du) , D\varphi \rangle \big )\, dx dt=0  \quad \forall\, \varphi\in  C^\infty_0(\Omega_T;\R^m).
\]
This proves that $u$ is a weak solution to (\ref{ibvp-5}).

3. We show that $Du$ is nowhere essentially continuous in $\Omega_T.$

Let $y_0\in \Omega_T$ and $U$ be any open subset of $\Omega_T$ containing $y_0.$  Let $l_0>0$ be chosen so that $Q_{y_0,2l_0}\subset U.$
Since $|Q_{y_0,l_0}|=|Q_{y_0,l_0}\cap(\cup_{j=0}^{M_{\nu}}Q^\nu_j)|$ for all $\nu\ge1,$ it follows from (\ref{u-nu-0}) that for some $q_0\ge2$, if $q\ge q_0,$ then
\[
Q_{y_0,l_0}\cap Q^q_j\ne\emptyset
\]
for some $0\le j\le M_q$ with $j\ne\sum_{k=1}^i M^q_{k-1}+i$ $(0\le i\le M_{q-1})$; thus, if $q$ is sufficiently large, then, by  (\ref{u-nu}), we have
\begin{equation}\label{p-meas-0}
\bar Q^q_j\subset Q_{y_0,2l_0}\subset U.
\end{equation}

By Lemma \ref{lem512},   we have that, for all $p>q$ and $1\le k\le N$,
\[
\begin{split}
 |Q^q_j\cap \{Du_{p}  \in \mathbb P(S_k^{r_{p+1}}(\lambda_{p+1}))\}|
 & \ge  |Q^q_j  \cap\{( Du_{p} ,\partial_t v_{p})  \in S_k^{r_{p+1}}(\lambda_{p+1}) \}| \\
&\ge  \frac12\,  \frac{\lambda_{q+2}}{\lambda_{p+1}}  (\lambda_{q+2}- \lambda_{q+1}) (\lambda_{q+2}-\delta_2)^{N-1}\delta_1\,|Q^q_j|,
\end{split}
 \]
 where  $\mathbb P(\rho)=\rho^1$  for all $\rho=(\rho^1,\rho^2)\in\R^{m\times n}\times\R^{m\times n}.$  Letting $p\to \infty$, as $u_{p}\to u$ in $W^{1,1}(\Omega_T;\R^m)$, $\lambda_p\to 1$  and $r_p\to r_0,$ we obtain that
\begin{equation}\label{p-meas}
 |Q^q_j\cap \{Du \in \overline{\mathbb P(S^{r_0}_k(1))}\}|  \ge  \frac12\, \lambda_{q+2}(\lambda_{q+2}-\lambda_{q+1})(\lambda_{q+2}-\delta_2)^{N-1}\delta_1\, |Q^q_j|>0
\end{equation}
for all $1\le k\le N.$

 Let
\[
d_0=\min_{k\ne l}\dist\left(\overline{\mathbb P(S^{r_0}_k(1))},\, \overline{\mathbb P(S^{r_0}_l(1))}\right).
\]
 By  (P1)(ii) of Condition $O_N$,  we have $d_0>0.$ Hence, by (\ref{p-meas-0}) and (\ref{p-meas}), it follows that
\[
\|Du(y)-Du(z)\|_{L^\infty(U\times U)}\ge \|Du(y)-Du(z)\|_{L^\infty(Q^q_j\times Q^q_j)} \ge d_0,
\]
which proves that $\omega_{Du}(y_0) \ge d_0$ for all $y_0\in\Omega_T.$  Consequently,  $Du$ is not essentially continuous at any point of  $\Omega_T.$

The proof of Theorem \ref{mainthm1} is now complete.


\section{Specific $\T_N$-configurations and  Nondegeneracy  Conditions}\label{s-5}

In this  section, we analyze certain specific $\mathcal{T}_N$-configurations in $\mathbb{R}^{m\times n} \times \mathbb{R}^{m\times n}$ by selecting suitable parameters  and  establish the nondegeneracy conditions  essential  for verifying Condition $O_N$.

Assume that $m\ge 1$ and  $n\ge 2.$
 For  $\x\in\R^{n-1},$  write $\x=(x_1,\bar  \x)$, $\bar  \x=(x_2,\dots,x_{n-1})\in\R^{n-2}$ when $n\ge 3,$ and omit $\bar\x$  when $n=2.$
For  $r\in\{1,\dots, n\}$, $\x,\, \y\in\R^{n-1},$ and $a\in\R^n,$ define
\[
\begin{cases}  \tilde\alpha_r(\x) =(x_1,\ldots,x_{r-1},0,x_r,\ldots,x_{n-1})\in\R^n,\\
 \bm{\alpha}^*_r(a)  = (a_1,\ldots,a_{r-1},a_{r+1},\ldots,a_n)\in\R^{n-1},  \\
 \alpha_r(\x)  =\tilde\alpha_r(\x)+e_r, \quad  b_r(\x,\y)=\tilde\alpha_r(\y)-(\x\cdot\y)e_r,
\end{cases}
\]
where  $\{e_1,\dots,e_n\}$  is  the standard basis of $\R^n.$
 Moreover, if $\mathbf Y=(\y^k)_{k=1}^m \in (\R^{ n-1})^m$,  we define
\[
\beta_r(\x,\mathbf Y)=\begin{bmatrix} b_{r}(\x,\y^1 ) \\\vdots\\ b_{r}(\x,\y^m )\end{bmatrix}\in \R^{m\times n}.
\]

The  following result is straightforward to verify, so  the proof is omitted.

\begin{lemma} \label{G-class}
Let $\Gamma$ be  defined by \eqref{set-G}.  Then, for each  $r\in \{1,\ldots,n\},$ the  map
\[
(p,\x,\mathbf Y)\longmapsto   (p\otimes \alpha_r(\x), \,  \beta_{r}(\x,\mathbf Y))
\]
is one-to-one   from  $ (\R^m\setminus\{0\})\times\R^{n-1}\times (\R^{ n-1})^m$ into $\Gamma.$ Its image will be denoted  by
 $\Gamma_r.$
\end{lemma}

\subsection{Selection of  Parameters} Assume that $N\ge  2m$ is an integer, and let
 $ r_1,\dots,r_N \in  \{1,\dots,n\}$ be given. Suppose
 \[
\gamma_i=\big (p_i\otimes \alpha_{r_i}(\x_i),\, \beta_{r_i}(\x_i,\mathbf Y_i) \big )\in \Gamma_{r_i} \quad\forall\, 1\le i\le N,
\]
where $p_i\in\R^m\setminus\{0\},$  $\x_i\in \R^{n-1},$ and $\mathbf Y_i=(\y_i^k)_{k=1}^m  \in  (\R^{n-1})^m$ are parameters.

We aim to select the parameters so that $\{\gamma_i\}_{i=1}^N$  satisfy  condition  \eqref{sum-0}  in Definition \ref{def-tau-N}, which in this setting reduces to
\begin{equation}\label{sum-1}
\sum_{i=1}^N p_i\otimes  \alpha_{r_i}(\x_i)  =0 \quad\mbox{and}\quad   \sum_{i=1}^N b_{r_i}(\x_i,\y^k_i) =0\;\; (1\le k\le m).
\end{equation}

First,  we regard the first  condition  in (\ref{sum-1}) as a linear system for the variables $p_{m+1},\dots,p_N$ and $\x_1, \dots, \x_m,$ keeping all other variables fixed. The following result gives a class of explicit solutions; the proof is straightforward and is therefore omitted.

\begin{lemma}\label{xmpn} Let $p_1,\ldots,p_{m}\in \R^m$ and $\x_{m+1},\ldots,\x_N\in \R^{n-1}$  be such that
\[
\Delta_1(\x_N):=\alpha_{r_N}^{r_1}(\x_N)\cdots\alpha_{r_N}^{r_m}(\x_N) \ne 0,
\]
where $\alpha^i_{r}(\x) =\alpha_{r}(\x)\cdot e_i$ for all $1\le i,r\le n.$
Then the first  condition  in  \eqref{sum-1} is satisfied if
\begin{equation}\label{pi-aj}
\begin{split}
  p_i   &=\sum_{j=1}^m q_i^j p_j \;\;   (m+1\le i\le N-1),\\
  p_N   &=  -\sum_{j= 1}^{m}  \Big (1+  \sum_{i=m+1}^{N-1} q_i^j  \alpha^{r_j}_{r_i}(\x_i) \Big  ) \frac{p_j}{\alpha^{r_j}_{r_N}(\x_N)},
\\[1ex]
 \x_j & =\bm\alpha_{r_j}^*(a_j) \;\;(1\le j\le m), \;\; \mbox{with $a_j\in\R^n$ given by}   \\
 a_j &=   \Big ( 1+ \sum_{i=m+1}^{N-1} q_i^j  \alpha^{r_j}_{r_i}(\x_i) \Big )  \frac{\alpha_{r_N}(\x_N) }{\alpha^{r_j}_{r_N}(\x_N)}   - \sum_{i=m+1}^{N-1} q_i^j \alpha_{r_i}(\x_i),
\end{split}
\end{equation}
where $q_{i}=(q_i^1,\dots,q_i^m)\in\R^m$ $(m+1\le i\le N-1)$   are arbitrary parameters.
\end{lemma}

Next, we rewrite the second  condition in  (\ref{sum-1}) as
\begin{equation}\label{eq-yz}
b_{r_{N-1}}(\x_{N-1},\y^k_{N-1}) +b_{r_N}(\x_{N},\y^k_N)  =-\sum_{j=1}^{N-2} b_{r_j}(\x_j,\y^k_j) \quad (1\le k\le m).
\end{equation}
For each $1\le k\le m,$  write  $\y^k_N=(y^k_{N,1},\bar\y^k_N)$ and  regard (\ref{eq-yz}) as a linear system of $n$ equations for the  $n$  unknowns  $(\y^k_{N-1}, y^{k}_{N,1}),$  treating all other variables as fixed. The following general result provides a solution to  such an  $n\times n$ system.

\begin{lemma}\label{sol-beta} Let   $r,\, s\in\{1,\dots  ,n\}$ and $\x,\, \u\in\R^{n-1}$ satisfy
\[
\Delta_2(\x,u_1):=   \alpha^s_r(\x) u_1 -\bm\alpha_s^*( \alpha_r(\x))\cdot\e_1\ne 0,
\]
where $\{\e_1,\dots,\e_{n-1}\}$ denotes the standard basis of $\R^{n-1}.$
Then for any given  $c\in\R^n$ and $\bar\z\in\R^{n-2},$  equation
\[
b_r(\x,\y)+b_s(\u,\z)=c, \;\;\mbox{where $\z=(z_1,\bar  \z),$}
\]
 has a unique solution $(\y,z_1),$  given by
\begin{equation}\label{z1-y}
\begin{split} z_1 & =\frac{ - \alpha_r(\x) \cdot c + \overline{\bm\alpha_s^*( \alpha_r(\x))   - \alpha^s_r(\x)   \u} \cdot \bar  \z }{\Delta_2(\x,u_1)};  \\
\y & =\bm\alpha_r^*[c -\tilde\alpha_s(\z) +( \u \cdot \z )e_s], \;\; \mbox{with}\;\;  \z=(z_1,\bar  \z).
\end{split}
  \end{equation}
\end{lemma}

\begin{proof} The given equation  is $\tilde\alpha_r(\y)- (\x \cdot \y) e_r+\tilde\alpha_s(\z)-( \u \cdot \z )e_s=  c,$ which reduces  to \begin{equation}\label{eq-beta-1}
\begin{cases} \tilde\alpha^i_r(\y)+\tilde\alpha^i_s(\z)-  ( \u \cdot \z )\delta_{is}= c^i \quad \forall\, i\ne r,\\
 \tilde\alpha^r_s(\z)- (\x \cdot \y) -(\u \cdot \z ) \delta_{rs}=c^r.
\end{cases}
\end{equation}
From  this, we have
\[
\begin{split} &\tilde\alpha^r_s(\z)   -(\u \cdot \z ) \delta_{rs}-c^r =\x\cdot\y =\tilde\alpha_r(\x)\cdot\tilde\alpha_r(\y)  =\sum_{i\ne r}\tilde\alpha^i_r(\x) \tilde\alpha^i_r(\y) \\
&= \sum_{i\ne r }\tilde\alpha^i_r(\x) [c^i-\tilde\alpha^i_s(\z)+( \u \cdot \z )\delta_{is}] = \tilde\alpha_r(\x) \cdot  c-  \tilde\alpha_r(\x) \cdot \tilde\alpha_s(\z) +( \u \cdot \z )\tilde\alpha^s_r(\x),\end{split}
\]
 which, as $\tilde\alpha^s_r(\x) +\delta_{rs}= \alpha^s_r(\x),$ yields that
\[
\begin{split} \alpha^s_r(\x) ( \u \cdot \z ) & =
  \tilde\alpha^r_s(\z)  -c^r- \tilde\alpha_r(\x) \cdot  c+ \tilde\alpha_r(\x) \cdot \tilde\alpha_s(\z) \\
  &=  \tilde\alpha^r_s(\z)  - \alpha_r(\x) \cdot  c+ \bm\alpha_s^*(\tilde\alpha_r(\x)) \cdot  \z
\\
& = \tilde\alpha^r_s(\z)  - \alpha_r(\x) \cdot  c+ [\bm\alpha_s^*(\tilde\alpha_r(\x))\cdot\e_1]z_1+ \overline{\bm\alpha _s^*(\tilde\alpha_r(\x))} \cdot  \bar  \z.\end{split}
  \]
Hence,
\[
\begin{split}
 [ \alpha^s_r(\x)   u_1&-\bm\alpha_s^*(\tilde\alpha_r(\x))\cdot\e_1]z_1
  = \tilde\alpha^r_s(\z)  - \alpha_r(\x) \cdot c+ [\overline{\bm\alpha_s^*(\tilde\alpha_r(\x))}  - \alpha^s_r(\x)  \bar  \u] \cdot \bar  \z  \\
 &= \tilde\alpha_s(\z)\cdot e_r - \alpha_r(\x) \cdot c+[\overline{\bm\alpha_s^*( \alpha_r(\x))}  - \alpha^s_r(\x)  \bar  \u] \cdot \bar  \z
 -   \overline{\bm\alpha_s^*(e_r)}  \cdot \bar  \z  \\
 &=[\bm\alpha_s^*(e_r)\cdot\e_1] z_1 - \alpha_r(\x) \cdot c+ [\overline{\bm\alpha_s^*( \alpha_r(\x))}  - \alpha^s_r(\x)  \bar  \u] \cdot \bar  \z.
 \end{split}
\]
Since  $\alpha^s_r(\x)   u_1-\bm\alpha_s^*(\tilde\alpha_r(\x))\cdot\e_1- \bm\alpha_s^*(e_r)\cdot\e_1= \Delta_2(\x,u_1)\ne 0,$ this yields   the formula for $z_1$ in \eqref{z1-y}.

Finally, from $\tilde\alpha_r(\y)= c- \tilde\alpha_s(\z)+( \u \cdot \z )e_s+ (\x \cdot \y)e_r,$ we have
\[
\y =\bm\alpha_r^*[c- \tilde\alpha_s(\z)+( \u \cdot \z )e_s+ (\x \cdot \y) e_r]
  =\bm\alpha_r^*[c- \tilde\alpha_s(\z)+( \u \cdot \z )e_s],
\]
which   proves the formula for $\y$ in \eqref{z1-y}.
\end{proof}

Assume that $\x_1,\ldots,\x_N\in\R^{n-1}$  are given such that
\[
\Delta_2(\x_{N-1},x_{N,1})  =\alpha_{r_{N-1}}^{r_N}(\x_{N-1})x_{N,1}-\alpha_{r_{N-1}}(\x_{N-1})\cdot\tilde\alpha_{r_N}(\e_1)\ne 0,
\]
and let $a_j=\alpha_{r_{j}}(\x_{j})$ ($1\le j\le N$).
Given $\y^k_1,\ldots,\y^k_{N-2}\in\R^{n-1}$ and $\bar\y^k_N\in\R^{n-2}$ for  $1\le k\le m,$  we  apply Lemma \ref{sol-beta} to condition (\ref{eq-yz}), with
\[
 r=r_{N-1},\; s=r_N,\; \x=\x_{N-1},\; \u=\x_N,\; \bar\z=\bar\y^k_N, \;\;\text{and} \;\; c  = -\sum_{j=1}^{N-2} b_{r_j}(\x_j,\y^k_j),
\]
 to determine $ z_1  =y^k_{N,1}$, $\z=\y_{N}^k   =( y^k_{N,1} , \bar  \y_{N}^k )$ and $ \y = \y^k_{N-1}$ as follows:
\begin{equation}\label{al-y}
\begin{split}
y^k_{N,1}  &  =  \frac{  \sum_{j=1}^{N-2} \bm\alpha_{r_j}^*(a_{N-1} -a_{N-1}^{r_j} a_j) \cdot \y_j^k+\overline{\bm\alpha_{r_N}^*( a_{N-1}    - a^{r_N}_{N-1}  a_N)} \cdot \bar  \y^k_N} {\Delta_2(\x_{N-1},x_{N,1})},\\
  \y_{N}^k &   =( y^k_{N,1} , \bar  \y_{N}^k ),  \\
 \y_{N-1}^k & =\bm\alpha_{r_{N-1}}^*\Big [ ( \x_{N} \cdot \y_{N}^k )e_{r_{N}}-\tilde\alpha_{r_{N}}(\y_{N}^k) + \sum_{j=1}^{N-2} \Big ((\x_j\cdot\y_j^k)e_{r_j}-\tilde\alpha_{r_j}(\y^k_j) \Big )  \Big ].
\end{split}
\end{equation}

Now let  {$(p_{1}, \dots, p_{m})$,} $(q_{m+1}, \dots, q_{N-1})$, $(\x_{m+1},\dots  ,\x_N)$ and $\{(\y^k_1,  \dots  , \y^k_{N-2},\bar  \y_{N}^k)\}_{k=1}^m$ be given such that
\begin{equation}\label{nonzero}
\Delta_1(\x_N) \Delta_2(\x_{N-1}, x_{N,1})\ne 0.
\end{equation}
We first use formula (\ref{pi-aj}) to determine $p_i$ ($m+1\le i\le N$) and  $\x_j$ ($1\le j\le m$), and then insert these  quantities into  formula (\ref{al-y})  to  obtain
\begin{equation}\label{def-y-1}
(\y^k_{N-1}, {y_{N,1}^k})   =\frac{L(q_{m+1}, \dots, q_{N-1}; \x_{m+1},\dots  ,\x_N; \y^k_1,  \dots  , \y^k_{N-2},  \bar  \y_{N}^k)}{ \Delta_1(\x_{N})\, \Delta_2(\x_{N-1},x_{N,1})} {\in\R^n},
\end{equation}
 for  $1\le k\le m,$   where  $L$ is a polynomial in  its arguments.

Let $\V=\V_{(m,n,N)}$ be the set of all  parameters $V=(P,Q,X,Y)$    given by
\begin{equation}\label{para-PQXY}
\begin{cases}
 P =(p_1,\dots,p_{m})\in (\R^m)^{m},\\
 Q =(q_{m+1},\dots, q_{N-1})\in (\R^m)^{N-m-1},\\
 X=(\x_{m+1},\dots,\x_N)\in (\R^{n-1})^{N-m} \;\; \mbox{satisfying (\ref{nonzero})},\\
 Y  =\{(\y^k_1,  \ldots, \y^k_{N-2}, \bar  \y^k_{N})\}_{k=1}^m  \in [(\R^{n-1})^{N-2}\times \R^{n-2}]^m.
\end{cases}
\end{equation}
Then  $\V$   is an  \emph{open dense} set in the Euclidean space $\R^{\dim \V}$, where
\[
\begin{split} \dim \V &= m^2+m(N-m-1) +(N-m)(n-1) +m[ (N-2)(n-1)+n-2]\\
&= (mn+n-1)N-2mn.\end{split}
\]

We define the following functions of  $V=(P,Q,X,Y)$ on $\V$:
\begin{equation}\label{fn-xapyBg}
\begin{split}
&p_i(V) =p_i \;\; ( 1\le i\le m), \;  \mbox{and $(p_{m+1}(V),\dots,p_N(V))$ by (\ref{pi-aj});}  \\
&(\x_1(V),\ldots, \x_m(V)) \;\; \mbox{by (\ref{pi-aj}), and } \,   \x_i(V)=\x_i \;\; (  m+1\le i\le N);  \\
&a_i(V)=\alpha_{r_i}(\x_i(V)) \;\; (1\le i \le N);  \\
&\y^k_i(V)=\y^k_i \;\;  (1\le i\le N-2), \;   \mbox{$ (\y^k_{N-1}(V),y^k_{N,1}(V))$ by (\ref{def-y-1}), and}\\
& {\mbox{$\y^k_{N}(V)=(y^k_{N,1}(V),\bar{\y}^k_N)$ for $1\le k\le m;$}} \\
& B_i(V)= \beta_{r_i}(\x_i(V),\mathbf Y_i(V))  \;\mbox{with  $\mathbf Y_i(V)=(\y_i^k(V))_{k=1}^m$}   \; \;  (1\le i\le N); \\
& g_i(V)   = (p_i(V)\otimes a_i(V),   B_i(V)) \;\; (1\le i\le N);\\
&  \sigma_1(V)=0, \;\mbox{and}\; \sigma_i(V)=g_1(V)+\dots+g_{i-1}(V) \;\; (2\le i\le N).
\end{split}
\end{equation}

Let $\U =\V\times \R^N\subset \R^D,$  where \[
D =N+ \dim\V = n[(m+1)N-2m].
\]
 For each $U=(V,\bm\kappa)\in\U,$ where  $\bm\kappa=(\kappa_1,\dots,\kappa_N)\in\R^N,$  let
\begin{equation}\label{eta}
\eta_i(U)    = \sigma_i(V) +\kappa_i g_i(V)\quad (1\le i\le N).
\end{equation}

Finally, define
\[
\begin{split} \V_0  &=\{V \in \V :  p_i(V)\ne 0\;\;\forall\,1\le i\le N  \},\\
\U_0 & =\{U=(V,\bm\kappa) \in \U:   V\in\V_0,\,    \kappa_i >1\;\;\forall\,1\le i\le N \}.
\end{split}
\]
Then $\V_0$  is an open dense  subset of $\V$, and    $\U_0$ is a nonempty open subset of   $\U.$

\begin{lemma} For  each $\rho\in\R^{m\times n}\times\R^{m\times n}$ and $U\in\U_0,$  the $N$-tuple
\[
(\rho+\eta_1(U), \dots, \rho+\eta_N(U))
\]
is a $\T_N$-configuration  in $\R^{m\times n}\times \R^{m\times n}.$
\end{lemma}
\begin{proof}
The result follows easily since $g_i(V)\in \Gamma_{r_i}\subset \Gamma\; (1\le i\le N)$ and $g_1(V)+\cdots+g_N(V)=0$   for all $V\in\V_0.$
\end{proof}

\subsection{Nondegeneracy Conditions}  Let $\sigma \colon \R^{m\times  n}\to   \R^{m\times  n}$ be a $C^1$ function with graph $\K\subset  \R^{m\times  n}\times  \R^{m\times  n}.$ Note  that $\rho\in\K$ if and only if $\Psi(\rho)=0,$ where
\[
\Psi(\rho)=\sigma(\rho^1)-\rho^2\in \R^{m\times  n} \quad\forall\, \rho=(\rho^1,\rho^2)\in \R^{m\times  n}\times  \R^{m\times  n}.
\]
Assume that $(\eta_1(U_0),\dots,\eta_N(U_0))$ is supported on  $\K$    for some $U_0\in \U_0\subset \R^{D},$
and let $\varepsilon\colon \U\to\R^{m\times n}\times\R^{m\times n}$ be such that $\varepsilon(U_0)=0.$

For each $\rho\in\R^{m\times n}\times \R^{m\times n}$ and $U\in\U_0,$  the $N$-tuple
\begin{equation}\label{family-T5}
(\rho+\varepsilon(U)+\eta_1(U),\dots,\rho+\varepsilon(U)+\eta_N(U))
\end{equation}
is a $\T_N$-configuration, with $\pi_1=\rho+\varepsilon(U),$  which reduces to $(\eta_1(U_0),\dots,\eta_N(U_0))$ when $(\rho,U)=(0,U_0).$

 We  attempt  to embed the $\T_N$-configurations in  \eqref{family-T5} into  $\K$ for certain $(\rho,U)$ near $(0,U_0).$  This amounts to solving the system of $mnN$ equations:
\begin{equation}\label{embed-T5}
\Psi(\rho+\varepsilon(U)+\eta_j(U))=0 \quad (1\le j\le N),
\end{equation}
 for $U$  near $ U_0$ in terms of $\rho$ near $0.$ Hence, the embedding problem falls within the framework of the Implicit Function Theorem.

 Since $U\in \R^D,$ where $D =n[(m+1)N-2m]\ge mnN$ for $N\ge 2m,$ applying the Implicit Function Theorem requires a reduction of the parameters in $U$.

\subsubsection{Reduction of Parameters}\label{ss-521}
Let $U_0 =(Z_0',\tilde U_0,\kappa_0),$ where
\[
Z_0'\in\R^{mnN-N},\;\; \tilde U_0\in\R^{n(N-2m)}, \;\; \mbox{and}\;\; \kappa_0\in\R^N.
\]
For $ Z=(Z',\kappa)\in  \R^{mnN-N}\times \R^N=\R^{mnN},$ define
\[
V(Z')=(Z', \tilde U_0),\quad U(Z)=(Z',\tilde U_0,\kappa),
\]
and
\begin{equation}\label{def-hzo}
\begin{cases}
 h_i(Z') =g_i(V(Z')), \\
 \omega_i(Z')=\sigma_i(V(Z')), \\
 \zeta_i(Z)=\eta_i(U(Z))\end{cases} \;\; (1\le i\le N).
\end{equation}

Finally, we define a new set of the reduced parameters  by
\[
 \Z_0=\{Z\in\R^{mnN}:\,  U(Z)\in \U_0\}.
 \]
  Since $U_0=U(Z_0)\in\U_0$ for $Z_0=(Z_0',\kappa_0),$ it follows that  $Z_0 \in\Z_0.$ Thus $\Z_0$ is a nonempty open set in $\R^{mnN}.$

 \subsubsection{Nondegeneracy for Embedding}

For each $1\le i\le N,$ we solve equation  \eqref{embed-T5},  where $\varepsilon(U)=\omega_i(Z'_0)-\omega_i(Z'),$  for $U=U(Z);$ namely,
 we  solve  the system
 \begin{equation}\label{spin-0}
 \Phi_i^j(\rho,Z)=  \Psi(\rho+\omega_i(Z'_0)-\omega_i(Z')+\zeta_{j}(Z))=0  \quad\forall\, 1\le j\le N.
 \end{equation}
 for $Z$ near $Z_0$, in terms of  $\rho$ near $0.$   This system can be compactly written as
 \[
 \Phi_i(\rho,Z)=0,
 \]
  where
\begin{equation}\label{def-PHI}
 \Phi_i(\rho,Z)  =(\Phi^1_i(\rho,Z),\dots,\Phi^N_i(\rho,Z)) \quad\forall\, (\rho,Z)\in (\R^{m\times n}\times \R^{m\times n})\times \R^{mnN}.
\end{equation}

The  following property  is straightforward  to verify; thus we omit the proof.

\begin{lemma} For each $1\le i,k\le N$ and $(\rho,Z)\in (\R^{m\times n}\times \R^{m\times n})\times \R^{mnN},$ it follows that
\begin{equation}\label{Phi_ik}
  \Phi_i(\rho+\omega_i(Z')-\omega_k(Z')+\omega_k(Z'_0)-\omega_i(Z'_0), Z)  = \Phi_k(\rho,Z).
\end{equation}
\end{lemma}

Observe that, for $1\le i\le N,$  the partial  Jacobi matrix
\[
 \frac{\partial\Phi_i}{\partial Z}(0,Z_0) =
 \begin{bmatrix} D\Psi(\zeta_1(Z_0)) D(\zeta_{1}  -\omega_i)(Z_0) \\\vdots\\
 D\Psi(\zeta_{N}(Z_0))D(\zeta_{N}  -\omega_i)(Z_0) \end{bmatrix}\in \R^{mnN\times mnN},
 \]
where $D\Psi=[D\sigma\,| -I],$ depends  on $Z_0$ and  the quantities
\[
(D\sigma(\zeta^1_1(Z_0)),\dots,D\sigma(\zeta^1_N(Z_0)))\in (\mathbb R^{mn\times mn})^N,
\]
but not on the specific form of the function $\sigma$  itself.
This   motivates the following definition.

\begin{definition}\label{def-Ji}
For each $1\le i\le N,$ define
\begin{equation}\label{fun-Ji}
\mathcal J_i(H_1,\dots,H_5,Z)=
\det \begin{bmatrix} H_1 D(\zeta^1_{1}  -\omega^1_i)(Z) -D(\zeta^2_{1}  -\omega^2_i)(Z)\\\vdots\\
 H_{N} D(\zeta^1_{N}  -\omega^1_i)(Z) -D(\zeta^2_{N}  -\omega^2_i)(Z) \end{bmatrix}
 \end{equation}
for all  $(H_1,\dots,H_N,Z)\in\P=(\mathbb R^{mn\times mn})^N\times  \R^{mnN}.$

Note that each $\mathcal J_i$ is a polynomial of $(H_1,\dots,H_N)$ and a rational function of $Z.$

 \end{definition}

\begin{lemma}\label{ImFT-lem} Let $H_k^0=D\sigma(\zeta^1_k(Z_0))$  for $1\le k\le N,$   and assume the {\em nondegeneracy condition}
 \begin{equation}\label{NC-1}
 \mathcal J_i(H_1^0,\dots,H_N^0,Z_0)\ne 0 \quad \forall\, 1\le i\le N.
  \end{equation}
Then  there exist  open balls $\B_{\tau_0}(0)\subset \R^{m\times n}\times \R^{m\times n}$ and $\mathcal B_{k_0}(Z_0)\subset \Z_0$  such that,  for each $1\le i\le N,$   there exists a function
$Z_i\colon \bar\B_{\tau_0}(0) \to  \mathcal B_{k_0}(Z_0)$ satisfying the following properties:
\begin{equation}\label{ImFT-0}
\begin{cases}
Z_i(0)=Z_0,\\
\forall (\rho,Z)\in \bar\B_{\tau_0}(0)\times \bar{\mathcal B}_{k_0}(Z_0),\; \;   \Phi_i(\rho,Z)=0 \iff  Z =Z_i(\rho),\\
\mbox{$Z_i\colon \bar\B_{\tau_0}(0) \to  \mathcal B_{k_0}(Z_0)$
is  algebraic and thus analytic.}
\end{cases}
\end{equation}
\end{lemma}

\begin{proof} Since $\Phi_i(0,Z_0)=0$ and
\[
\det  \frac{\partial\Phi_i}{\partial Z}(0,Z_0)   =\mathcal J_i(H_1^0,\dots,H_N^0,Z_0)\ne 0
\]
for each $1\le i\le N,$    the existence of function $Z_i(\rho)$  follows directly from the Implicit Function Theorem. Details are omitted.
\end{proof}

Assume that condition \eqref{NC-1} holds, and let $Z_i(\rho)$ ($1\le i\le N$) be the functions defined in Lemma \ref{ImFT-lem}.
For each $1\le i,j \le N$ and $\rho\in \bar \B_{\tau_0}(0),$  define
\begin{equation}\label{xi-pi-ij}
\begin{cases} \xi_{ij}(\rho)=\rho+\omega_i(Z'_0)-\omega_i(Z_i'(\rho))+\zeta_{j}(Z_i(\rho)),\\
\pi_{ij}(\rho)=\rho+\omega_i(Z'_0)-\omega_i(Z_i'(\rho))+\omega_{j}(Z_i'(\rho)). \end{cases}
\end{equation}
Since $\Phi_i(\rho,Z_i(\rho))=0$, it follows that
\begin{equation}\label{xi-ij-K}
 \xi_{ij}(\rho)\in \K \quad\forall\, 1\le i, j\le N,\;\; \rho\in \bar \B_{\tau_0}(0).
\end{equation}
Moreover, as $\pi_{ij}(0)=\omega_{j}(Z_0')$ for each $1\le i,j\le N,$ we  select a number $\tau_1\in (0,\tau_0)$  such that
\begin{equation}\label{tau-1}
 \pi_{ij}[\bar\B_{\tau_1}(0)] \subseteq  \B_{\tau_0}(\omega_{j}(Z'_0))\quad \forall\, 1\le i,j\le N.
\end{equation}

\begin{lemma}\label{lem65}   Let $1\le i,j\le N.$  Then
\begin{equation}\label{Z-ij}
Z_i(\rho)=Z_{j}\big ( \pi_{ij}(\rho)-\omega_{j}(Z'_0)\big ) \quad \forall\, \rho\in \bar\B_{\tau_1}(0);
\end{equation}
moreover, the map $\pi_{ij}$ is one-to-one on $\bar\B_{\tau_1}(0),$ with
 \begin{equation}\label{pi-ij-1}
 (\pi_{ij})^{-1} (y) =\pi_{ji} \big (y- \omega_{j}(Z_0') \big ) -\omega_i(Z_0') \quad\forall\, y\in \pi_{ij}[\bar\B_{\tau_1}(0)].  \end{equation}
\end{lemma}

\begin{proof}  1. Let $\rho\in \bar\B_{\tau_1}(0).$ By (\ref{Phi_ik}) and (\ref{ImFT-0}),  we have
\[
\Phi_{j}(\pi_{ij}(\rho)-\omega_{j}(Z'_0), Z_i(\rho))  = \Phi_i(\rho,Z_i(\rho))=0;
\]
 hence, again by (\ref{ImFT-0}), we have $Z_i(\rho)=Z_{j}\big ( \pi_{ij}(\rho)-\omega_{j}(Z'_0)\big ),$  proving (\ref{Z-ij}).

2. Let $\pi_{ij}(\rho)=\pi_{ij}(\rho')$ for some $\rho,\rho'\in\bar\B_{\tau_1}(0).$ By  (\ref{Z-ij}),  we have $Z_i(\rho')=Z_i(\rho),$ which in turn,  by the definition of $\pi_{ij}$ in (\ref{xi-pi-ij}), yields $\rho'=\rho;$ hence, $\pi_{ij}$ is one-to-one on $\bar\B_{\tau_1}(0).$

3. Let $y=\pi_{ij}(\rho),$ where $\rho=\pi_{ij}^{-1}(y) \in\bar\B_{\tau_1}(0)$. By  (\ref{Z-ij}),  $Z_i(\rho)=Z_{j}\big ( y-\omega_{j}(Z'_0)\big );$  thus, by (\ref{xi-pi-ij}),  we have
\[
y=\rho+\omega_i(Z'_0)-\omega_i \big (Z'_{j}( y-\omega_{j}(Z'_0))\big )+ \omega_{j}\big (Z'_{j}( y-\omega_{j}(Z'_0))\big ),
\]
 which proves (\ref{pi-ij-1}), since
\[
\begin{split} \rho  &=y- \omega_{j}\big (Z'_{j}( y-\omega_{j}(Z'_0))\big ) +\omega_i \big (Z'_{j}( y-\omega_{j}(Z'_0))\big )-\omega_i(Z'_0) \\
&=(y-\omega_{j}(Z'_0)) + \omega_{j}(Z'_0) - \omega_{j}\big (Z'_{j}( y-\omega_{j}(Z'_0))\big ) +\omega_i \big (Z'_{j}( y-\omega_{j}(Z'_0))\big )-\omega_i(Z'_0) \\
& =\pi_{ji} \big (y- \omega_{j}(Z_0') \big ) -\omega_i(Z_0'). \end{split}
\]
\end{proof}

\subsubsection{Nondegeneracy for Openness}  For $1\le i\le N$ and $\rho\in\bar\B_{\tau_0}(0),$ define
 \begin{equation}\label{def-xp-i}
\begin{cases} \xi_i(\rho)=\xi_{1i}(\rho)=\rho+\zeta_i(Z_1(\rho)), \\
\pi_i(\rho)=\pi_{1i}(\rho)=\rho+\omega_i(Z_1'(\rho)),\\
\mu_i(\rho)=  f_i (Z_i(\rho)),
\end{cases}
 \end{equation}
 where $f_i(Z)=\kappa_i h_i(Z')$ for all $Z=(Z',\bm\kappa)\in\R^{mnN}.$  Moreover, for $\lambda\in\R$ and $r\in (0,\tau_0]$, define
 \begin{equation}\label{set-S}
 S^r_i(\lambda)=\{\lambda \xi_i(\rho)+(1-\lambda)\pi_i(\rho): \rho\in\B_r(0)\}.
 \end{equation}

 \begin{lemma}\label{lem58} For  each $1\le i\le N$ and $\lambda\in\R$, it holds that
 \[
 S^r_i(\lambda)=\{y+\lambda \mu_i(y)+\omega_i(Z_0') : y\in \pi_{i}(\B_r(0))-\omega_i(Z_0')\} \quad \forall\, 0<r\le \tau_1.
 \]
\end{lemma}

 \begin{proof}
Let $1\le i\le N$ and $r\in (0,\tau_1].$
For each $\rho\in \B_{r}(0)$, let $y=\pi_{i}(\rho)-\omega_i(Z_0').$ Then, by (\ref{tau-1}) and (\ref{Z-ij}), we have $y\in   \bar\B_{\tau_1}(0)$ and
\[
Z_1 (\rho)=Z_i(\pi_{i}(\rho)-\omega_i(Z_0'))= Z_i(y).\]
Since  $\xi_{i}(\rho) =\pi_{i}(\rho)+f_{i}(Z_1(\rho)),$  it follows  that, for all $\lambda\in\R,$
 \begin{equation}\label{rho-y-0}
\lambda \xi_i(\rho)+(1-\lambda)\pi_i(\rho) =y+\lambda \mu_i(y)+\omega_i(Z_0').
\end{equation}
Finally, since $\pi_i$ is one-to-one on $ \B_{r}(0)$, we conclude
 \[
 S^r_i(\lambda)=\{y+\lambda \mu_i(y)+\omega_i(Z_0') : y\in \pi_{i}(\B_r(0))-\omega_i(Z_0')\}.
  \]
 \end{proof}

\begin{lemma} For each $1\le i\le N$ and $y\in \bar\B_{\tau_0}(0),$ let
 \begin{equation}\label{def-Mi}
\begin{cases} H_i(y)=D\sigma(y^1+\mu^1_i( y) +\omega^1_i(Z_0')),\\
M_i(y)=D'\mu^1_i(y)+D''\mu^1_i(y)H_i(y),\end{cases}
\end{equation}
where  $D'=D_{y^1}$ and $D''=D_{y^2}.$
Then for each $y\in \pi_{i}(\bar\B_{\tau_1}(0))-\omega_i(Z_0')\subseteq \bar\B_{\tau_0}(0),$ we have
\begin{equation}\label{eigen}
\begin{cases} \rank D \mu_i(y) = mn+\rank M_i(y),\\
 \det[ I_{2mn}+\lambda D \mu_i( y)] = (\lambda-1)^{mn}\det[I_{mn} +\lambda M_i(y)] \quad  \forall\, \lambda\in\R.
 \end{cases}
  \end{equation}
\end{lemma}

\begin{proof}
Let $y\in \pi_{i}(\bar\B_{\tau_1}(0))-\omega_i(Z_0').$  Then there exists a unique  $\rho\in \bar\B_{\tau_1}(0)$ such that  $y=\pi_{i}(\rho)-\omega_i(Z_0').$
By (\ref{rho-y-0}) with $\lambda=0$, we have   $y+\mu_i( y) +\omega_i(Z_0') = \xi_{i}(\rho)\in\K,$ which implies
\[
 y^2+\mu_i^2( y) +\omega^2_i(Z_0')=\sigma(y^1+\mu^1_i( y) +\omega^1_i(Z_0'))\quad \forall\, y\in \pi_{i}(\bar\B_{\tau_1}(0))-\omega_i(Z_0').
 \]
Differentiating   with respect to $D'=D_{y^1}$ and $D''=D_{y^2}$  gives
\[
D'\mu^2_i( y)  =H_i(y) (I +D'\mu^1_i( y)) \;\;\; \mbox{and}\;\;\;  I   + D''\mu^2_i( y)  =H_i(y) D''\mu^1_i( y)),
\]
where $I=I_{mn}.$ Thus,
\[
 D\mu_i( y)=\begin{bmatrix} D'\mu^1_i( y) & D''\mu^1_i( y)\\[0.5ex] D'\mu^2_i( y) &  D''\mu^2_i( y)\end{bmatrix}=\begin{bmatrix}
 D'\mu^1_i( y) & D''\mu^1_i( y)\\ H_i(y) (I +D'\mu^1_i( y)) & H_i(y)D''\mu^1_i( y)) -I\end{bmatrix}.
\]
Let $P =\begin{bmatrix} I   & O\\  H_i(y) & I \end{bmatrix},$ and thus  $P^{-1} =\begin{bmatrix} I   & O\\  -H_i(y) & I \end{bmatrix}.$ Then a direct computation shows
\[
 P^{-1}  \, D\mu_i( y) \, P       =\begin{bmatrix} M_i(y) & D''\mu^1_i( y)\\
O &-I \end{bmatrix},
\]
from which  \eqref{eigen} follows.
\end{proof}

To compute $M_i(y)$, we  differentiate  $\Phi_i(y,Z_i(y))=0$ with $D'$ and $D''$  to obtain
\[
\frac{\partial \Phi_i}{\partial y^1} +\frac{\partial\Phi_i}{\partial Z} D'Z_i=0, \quad \frac{\partial \Phi_i}{\partial y^2} +\frac{\partial\Phi_i}{\partial Z} D''Z_i=0,
\]
which implies
\[
D'Z_i(y) +D''Z_i(y)H_i(y)=- \left [\frac{\partial\Phi_i}{\partial Z} \right ]^{-1} \left [\frac{\partial \Phi_i}{\partial y^1} +\frac{\partial \Phi_i}{\partial y^2} H_i(y)\right ].
\]
Hence,
\[
\begin{split}M_i(y) &=D'\mu^1_i( y)+D''\mu^1_i( y)H_i(y) =  Df_i^1(Z_i(y))[D'Z_i(y) +D''Z_i(y)H_i(y)] \\
&  =-Df_i^1(Z_i(y))\left [\frac{\partial\Phi_i}{\partial Z} \right ]^{-1} \left [\frac{\partial \Phi_i}{\partial y^1} +\frac{\partial \Phi_i}{\partial y^2} H_i(y)\right ].\end{split}
\]
In particular, at $y=0$,
\[
M_i(0)  = Df_i^1(Z_0)\left [\frac{\partial\Phi_i}{\partial Z}(0,Z_0) \right ]^{-1} B_i,
\]
where
\begin{equation}\label{def-Bi}
 B_i=B_i(H_1^0,\dots,H_N^0)= -\left [\frac{\partial \Phi_i}{\partial y^1}(0,Z_0) +\frac{\partial \Phi_i}{\partial y^2}(0,Z_0) H_i^0\right ]       = \begin{bmatrix}  H_i^0-H^0_{1}\\\vdots \\O\\\vdots\\  H^0_i-H^0_{N}\end{bmatrix}.
\end{equation}
Since $\det(I_p+AB)=\det (I_q+BA)$  for all matrices $A\in \R^{p\times q}$ and $B\in \R^{q\times p},$ we obtain
\[
\det (I_{mn}+M_i(0))=\det \left [ I_{mnN} +B_i Df_i^1(Z_0)\left [\frac{\partial\Phi_i}{\partial Z}(0,Z_0) \right ]^{-1}\right]
=\frac{\mathcal N_i(H_1^0,\dots,H_N^0,Z_0)}{\mathcal J_i(H_1^0,\dots,H_N^0,Z_0)},
\]
where $\mathcal J_i$ is defined by (\ref{fun-Ji}) and
\begin{equation}\label{def-Ni-0}
\mathcal N_i(H_1^0,\dots,H_N^0,Z_0)=
\det \left [ \frac{\partial\Phi_i}{\partial Z}(0,Z_0) +B_i Df_i^1(Z_0)\right].
 \end{equation}
Here $B_i=B_i(H_1^0,\dots,H_N^0)$ is given in \eqref{def-Bi}.

Note that,  for each $1\le i\le N$, the function $\mathcal N_i(H_1^0,\dots,H_N^0,Z_0)$ is a polynomial of $(H_1^0,\dots,H_N^0)$ and a rational function of $Z_0.$

\begin{theorem}\label{thm-OCN}  In addition to  condition \eqref{NC-1}, assume the following {\bf nondegeneracy condition} also holds:
\begin{equation}\label{NC-2}
\mathcal N_i(H_1^0,\dots,H_N^0,Z_0) \ne 0 \quad \forall\, 1\le i\le N.
 \end{equation}
 Then there exist numbers $\tau_2\in (0,\tau_1)$ and $\ell_0\in (\frac12,1)$ such that,  for each $1\le i\le N$, $r\in (0,\tau_2],$ and $\lambda\in [0,1-\ell_0]\cup [\ell_0,1),$  the set
 \[
 S^r_i(\lambda)=\{\lambda \xi_i(\rho)+(1-\lambda)\pi_i(\rho): \rho\in\B_r(0)\}
 \]
 is open.
 \end{theorem}

\begin{proof} By (\ref{NC-1}) and (\ref{NC-2}),  we have
\[
\det (I_{mn}+M_i(0))\ne 0\quad\forall\, 1\le i\le N.
\]
Hence,  there exist  numbers $\tau_2\in (0,\tau_1)$ and $\ell_0\in (\frac12,1)$ such that
\[
\det(I_{mn}+\lambda M_i(y))\ne 0 \quad \forall\, 1\le i\le N,\;\; \lambda\in [0,1-\ell_0]\cup [\ell_0,1],\;\; y\in \bar\B_{\tau_2}(0).
\]
Moreover, by (\ref{eigen}),
\[
\det(I_{2mn}+\lambda D\mu_i(y)) \ne 0 \quad \forall\, 1\le i\le N,\;\; \lambda\in [0,1-\ell_0]\cup [\ell_0,1),\;\; y\in \bar\B_{\tau_2}(0).
\]
It then follows from the Inverse Function Theorem that the set
\[
\{y+\lambda \mu_i(y)+\omega_i(Z_0') : y\in \pi_{i}(\B_r(0))-\omega_i(Z_0')\}
\]
is open for each $1\le i\le N, \, r\in (0,\tau_2],$ and $\lambda\in [0,1-\ell_0]\cup [\ell_0,1).$

Finally, by Lemma \ref{lem58}, it follows that the set $S_i^r(\lambda)$ is open  for each $1\le i\le N, \, r\in (0,\tau_2],$ and $\lambda\in [0,1-\ell_0]\cup [\ell_0,1).$
\end{proof}

  \subsection{Compatibility with Polyconvexity} In the rest of this section, we assume $m=2.$

    For any $C^2$ function $F\colon \R^{2\times  n}\to  \R,$  denote by $\K_F$  the graph of its gradient $\sigma=DF.$ We often  identify  $\R^{2\times n}$ with $\R^{2n},$ so that  $DF$ is viewed  as a function from $\R^{2n}$ into $\R^{2n},$ and the Hessian   $D^2 F(A)$ is  regarded as a symmetric matrix $H=(H^{ij}_{k\ell})$ in $ \R^{2n\times 2n},$ with entries
  \[
H^{ij}_{k\ell}= \frac{\partial^2 F(A) }{\partial a_{ij}\partial a_{k\ell}}, \quad i,k=1,2, \;  1\le j,\ell\le n,
\]
 satisfying the symmetry condition $H^{ij}_{k\ell}=H_{ij}^{k\ell}$.
For  convenience, let $\mathbb S^{2n\times 2n}$ denote the space of  symmetric matrices in $\R^{2n\times 2n}$.

We focus on  {\em polyconvex}  functions  $F\colon\R^{2\times n}\to \R$   of  the special form
\begin{equation}\label{def-initial-F}
F(A)=G(A,\delta(A)) \quad\forall\, A=(a_{k\ell}) \in \R^{2\times n},
\end{equation}
where $G\colon\R^{2\times n}\times \R \to\R$ is  smooth and {\em convex}, and
\[
\delta(A)= a_{11}a_{22}- a_{12} a_{21}
\]
 denotes the $2\times 2$ subdeterminant of the   first  two columns of  $A$.
 For such  $F$,   the gradient is
\[
 DF(A)   =G_A (\tilde A)+ G_\delta (\tilde A)D\delta(A),
\]
 where  $\tilde A=(A,\delta(A)),$ and
 \[
  D\delta (A)  =\begin{bmatrix}a_{22}&-a_{21}&\bar\0  \\-a_{12} &a_{11}&\bar\0  \end{bmatrix}\in \R^{2\times n}\approx \R^{2n},
\]
with $\bar\0 =(0,\dots,0)\in\R^{n-2}.$
Thus,  the Hessian $D^2\delta(A)=H_0\in \mathbb S^{2n\times 2n}$ is  constant,  and the gradient  satisfies $D\delta(A)=H_0A\in\R^{2n}.$

Let $N\ge 4,$ and for $1\le i\le N,$ let
\[
\eta_i(U)=(\eta_i^1(U),\eta_i^2(U))\in\R^{2\times n}\times\R^{2\times n}
\]
be defined by formula (\ref{eta}) with $m=2,$  for all $U=(V,\bm\kappa)\in \U  \subset \R^{(3N-4)n}.$

Suppose  that $F$ is  of the form \eqref{def-initial-F} and that,  for some $U_0\in \U_0,$
\begin{equation}\label{on-gr}
 \eta_i(U_0)\in  \K_F \quad \forall\, 1\le i\le N.
\end{equation}
Then the convexity of $G$  implies that
\[
c_i-c_j+\langle Q_i, A_j-A_i\rangle + d_i  (\delta(A_j)-\delta(A_i)) \le 0 \quad\forall\, 1\le i, j\le N,
\]
 where
 \[
 A_i=\eta_i^1(U_0),\quad c_i=G(\tilde A_i), \quad Q_i=G_A(\tilde A_i), \quad d_i=G_\delta(\tilde A_i) \qquad  (1\le i\le N),
 \]
and, by (\ref{on-gr}), we also have $Q_i= \eta_i^2(U_0) - d_iD\delta(A_i)$ for each $1\le i\le N$.

Conversely, suppose that there exist   $U_0\in\U_0$ and  numbers $c_i,  d_i\in\R$ ($1\le i\le N$) such that the {\em strict  inequalities}
\begin{equation}\label{strict-ineq}
c_i-c_j+\langle Q_i, A_j-A_i\rangle + d_i  (\delta(A_j)-\delta(A_i))  <0 \quad (1\le i\ne j\le N)
\end{equation}
hold   with
\begin{equation}\label{AQD}
A_i=\eta_i^1(U_0)\;\;\mbox{and}\;\;  Q_i= \eta_i^2(U_0) - d_iD {\delta(A_i)}\quad (1\le i\le N).
\end{equation}
Then,  by  Lemma \ref{construct-conv} below, there exists a smooth convex function $G(A,\delta)$ on $\R^{2\times n}\times \R$ such that
\[
c_i=G(\tilde A_i),\;\; Q_i=G_A(\tilde A_i),\;\;\mbox{and}\;\;  d_i=G_\delta(\tilde A_i) \quad (1\le i\le N),
\]
and therefore
\[
\eta_i^2(U_0) =Q_i+ d_iD\delta({A_i}) =G_A(\tilde \eta_i^1(U_0))+ G_\delta(\tilde \eta_i^1(U_0))D\delta(\eta_i^1(U_0))=DF(\eta_i^1(U_0)),
\]
so $\eta_i(U_0)\in \K_F $ for all $1\le i\le N$,  where $F\colon \R^{2\times n}\to\R$ is defined  by (\ref{def-initial-F}) with this convex  $G.$

Therefore, the strict inequalities  (\ref{strict-ineq}) together with (\ref{AQD}) are essential for embedding a $\T_N$-configuration  on the gradient  graph of  a polyconvex function.

The following result shows that such compatibility holds when $N=5$.

\begin{lemma}\label{lem-N5} Let
\begin{equation}\label{sel-r}
 n\ge 2,\;\;  N=5,\;\;  r_1=r_3=r_4=1,\;\;  r_2=r_5=2.
\end{equation}
Then   there exist
\[
c_1,\dots,c_5, \; d_1,\dots,d_5\in\R, \; \mbox{and}\;\;   U_0 =(V_0,\bm\kappa_0)  \in \U_0\subset \R^{11n},
\]
  with $c_i\ne c_j$ for all $1\le i\ne j\le 5,$ such that
  \begin{equation}\label{ineq-cdY-5}
 \begin{cases}
  \eta^1_i(U_0) \ne \eta^1_j(U_0),\;\;  \sigma^1_i(V_0) \ne \sigma^1_j(V_0), \\
  c_i-c_j  + d_i [\delta(\eta_j^1(U_0))-\delta(\eta_i^1(U_0))]\\
   \;\; \; +\langle \eta_i^2(U_0)  - d_iD\delta(\eta_i^1(U_0)), \eta_j^1(U_0)-\eta_i^1(U_0)\rangle  <0  \end{cases}
 \forall\, 1\le i\ne j\le 5.
\end{equation}
\end{lemma}

\begin{proof}  Let
\[
U_0=(V_0,\bm\kappa_0)=(P_0,Q_0,X_0,Y_0,\bm\kappa_0)\in \U_0\subset \R^{11n}
\]
 be  the parameter defined by
\[
\begin{split}
&p^0_1=\begin{pmatrix} -2\\-3\end{pmatrix},\quad  p^0_2=\begin{pmatrix} -5\\-5\end{pmatrix},\quad P_0=(p_1^0,p_2^0);\\
& q^0_3=\begin{pmatrix} 1\\-\frac35\end{pmatrix},\quad  q^0_4=\begin{pmatrix} 0\\-\frac15\end{pmatrix},\quad Q_0=(q_3^0,q_4^0);
\\ &
\x^0_3=(1,\bar\0),\quad \x^0_4=(4,\bar\0),\quad  \x^0_5=(2,\bar\0),\quad  X_0=(\x^0_3,\x^0_4,\x^0_5); \\
& \y^{1,0}_{1}=(69,\bar\0),\quad  \y^{2,0}_{1}=(164,\bar\0),\quad  \y^{1,0}_{2}=(-165,\bar\0),\quad \y^{2,0}_{2}=(82,\bar\0),\\
& \y^{1,0}_{3}=(404,\bar\0),\quad  \y^{2,0}_{3}=(328,\bar\0),\quad \bar \y^{1,0}_{5}=\bar \y^{2,0}_{5}=\bar\0, \\
& Y_0=\{(\y^{k,0}_{1},\y^{k,0}_{2},\y^{k,0}_{3},\bar \y^{k,0}_{5})\}_{k=1}^2  \quad  \mbox{(the terms $\bar \y^{1,0}_{5}$,  $\bar\y^{2,0}_{5}$ are absent  when $n=2$)};
\\&\kappa^0_1=2,\quad  \kappa^0_2=3,\quad \kappa^0_3=4,\quad  \kappa^0_4=3,\quad \kappa^0_5=2,\quad \bm\kappa_0=(\kappa^0_1,\ldots,\kappa^0_5).
\end{split}
\]
It has been verified in \cite[\S 6.1]{Ya20} that,  for $n=2,$ with the constants
\[
 \begin{cases}
 (c_1,\dots,c_5 )=(0,-3650,-3318,5044,580 );\\
  (d_1,\dots,d_5)=(58,-7.5,772,57,376),
 \end{cases}
\]
condition  (\ref{ineq-cdY-5}) is satisfied.
Based on this computation, with the same parameter and constants,  condition (\ref{ineq-cdY-5}) can be verified directly   for all $n\ge 3.$
\end{proof}

We now specify the reduction of parameters discussed in \S \ref{ss-521}.
Let
\begin{equation}\label{sp-r}
\begin{cases} c_1,\ldots,c_5, d_1,\ldots,d_5\in\R,\\
U_0^\star=(p^0_1,p^0_2,q^0_3, q^0_{4}, \x^0_3, \x^0_4,\x^0_5,Y_0, \kappa^0_1,\dots,\kappa^0_5)
\end{cases}\end{equation}
be the specific numbers and parameter in the proof of Lemma \ref{lem-N5} that validate condition (\ref{ineq-cdY-5}).
In particular,  we have that
\begin{equation}\label{Delta12}
\begin{cases}
\x_{4}^0=(x^0_{4,1},\bar\x^0_4)=4\e_1,\;\; \x_5^0= 2\e_1,\\
 \Delta_1(\x^0_5)=2,\;\; \Delta_2((x^0_{4,1},\bar\x_4),x^0_{5,1})=7\;\;  \forall\, \bar\x_4\in\R^{n-2}.
\end{cases}
\end{equation}

Let $Z =(p_1,p_2,q_3,  q_{4},\x_3,\bar\x_4,  Y,\kappa_1,\dots,\kappa_5)\in\R^{10n}$,  where $\bar\x_4\in\R^{n-2}$ and all other variables  are given as  before.  Define
\begin{equation}\label{para2N}
\begin{cases}
 Z'=(p_1,p_2,q_3,  q_{4},\x_3,\bar\x_4,  Y)\in \R^{10n-5},\\
V(Z') =(p_1,p_2,q_3,  q_{4},\x_3,(x^0_{4,1},\bar\x_4), \x^0_5,Y)\in \V,\\
 U(Z)= (p_1,p_2,q_3,  q_{4},\x_3,(x^0_{4,1},\bar\x_4), \x^0_5,Y,\kappa_1,\dots,\kappa_5)\in\U,
  \end{cases}
\end{equation}
and
\begin{equation}\label{def-hzo-5}
\begin{cases}
 h_i(Z') =g_i(V(Z')), \\
 \omega_i(Z')=\sigma_i(V(Z')), \\
 \zeta_i(Z)=\eta_i(U(Z))\end{cases} \;\; (1\le i\le 5).
\end{equation}
Then,  by (\ref{ineq-cdY-5}),   the condition
\begin{equation}\label{ineq-cdY-51}
 \begin{cases}
 \zeta^1_i(Z) \ne \zeta^1_j(Z),\;\;  {\omega^1_i(Z') \ne \omega^1_j(Z'),} \\
  c_i-c_j  + d_i [\delta(\zeta_j^1(Z))-\delta(\zeta_i^1(Z))]\\
   \;\; \; +\langle \zeta_i^2(Z)  - d_iD\delta(\zeta_i^1(Z)), \zeta_j^1(Z)-\zeta_i^1(Z)\rangle  <0  \end{cases}
  \forall\, 1\le i\ne j\le 5
\end{equation}
holds for $Z =Z_0^\star=(p^0_1,p^0_2,q^0_3, q^0_{4}, \x^0_3,\bar\x_4^0,{Y_0}, \kappa^0_1,\dots,\kappa^0_5).$

Finally, we define
\begin{equation}\label{def-setZ1}
 \Z_1=\{Z\in\R^{10n}:\, \mbox{$U(Z)\in \U_0$\; and $Z$ validates condition (\ref{ineq-cdY-51})}\}
 \end{equation}
 to be the  new set of   reduced parameters.
 Clearly, $\Z_1\ne\emptyset$ is open in  $\R^{10n},$ as  $Z_0^\star\in\Z_1.$

\begin{remark}\label{rk-52}    Given
\[
Z'=(p_1,p_2,q_3,  q_{4},\x_3,\bar\x_4,  \{(\y_1^k,\y_2^k,\y_3^k,\bar\y_5^k)\}_{k=1}^2)\in\R^{10n-5},
\]
we can  explicitly write $ h_i(Z') =( h^1_i(Z') , h^2_i(Z') )$ $(1\le i \le 5)$ as follows:
\[
h_i^1(Z')=p_i\otimes \alpha_{r_i}(\x_i)\;\;\mbox{and}\;\;  h_i^2(Z') =\begin{bmatrix} \tilde\alpha_{r_i}(\y_i^1)-(\x_i\cdot\y_i^1 )e_{r_i}\\[1mm]
\tilde\alpha_{r_i}(\y_i^2) -(\x_i\cdot\y_i^2) e_{r_i}\end{bmatrix},
\]
where $r_i$'s are given in (\ref{sel-r}), $\x_4=(4,\bar\x_4)$, $\x_5 =2\e_1$,   and
\begin{equation}\label{sp-form-1}
 \begin{cases}
 p_3=p_3(Z')=q_3^1p_1+q_3^2 p_2,\\
 p_4= p_4(Z')=q_4^1p_1+q_4^2p_2,\\
 p_5=p_5(Z')=-\frac12 (1+q_3^1+q_4^1)p_1-(1+x_{3,1}q_3^2+4q_4^2)p_2;\\[1ex]
  \x_1=\x_1(Z')= \big (\frac12 +(\frac12- x_{3,1})q_3^1-\frac72 q_4^1, \, -q_3^1\bar\x_3-q_4^1 \bar\x_4\big ), \\
  \x_2=\x_2(Z')=\big (2 +(2x_{3,1}-1)q_3^2+7q_4^2,  \, -q_3^2\bar\x_3 -q_4^2 \bar\x_4\big );\\[1ex]
  \y_{4}^k  =  \y_{4}^k (Z')=  -\y_1^k-\y_3^k +\big ( \x_2(Z')\cdot\y_2^k,\, -\bar\y_2^k\big ) + \big ( 2y_{5,1}^k(Z'),\, -\bar\y_5^k\big ),\\
   \y_5^k =  \y_{5}^k (Z')= (y^k_{5,1}(Z'),\, \bar\y_5^k) \; \mbox{ with}\\
    y^k_{5,1} (Z')= \frac17 \Big [  \big((4, \bar\x_4)- {\x_1(Z')}\big )\cdot \y_1^k +\big ((1,\bar\x_4)-4 {\x_2(Z')}\big )\cdot\y_2^k \\
 \qquad\qquad \quad\;\;\;  +\big ((4,\bar\x_4)-\x_3\big )\cdot\y_3^k+ \bar\x_4\cdot\bar\y_5^k\Big ]  \qquad (k=1,2).
 \end{cases}
\end{equation}
With such explicit representations of $h_i(Z')$ $(1\le i \le 5)$, we also have
\begin{equation}\label{sp-form-2}
\begin{cases}
\omega_1(Z')=0,\;  \omega_i(Z')=h_1(Z')+\cdots+h_{i-1}(Z')\;\;(2\le i\le 5);\\
\zeta_i(Z)=\omega_i(Z')+\kappa_i h_i(Z')\;\;(1\le i\le 5),
\end{cases}
\end{equation}
where $Z=(Z',\bm\kappa)$ and $\bm\kappa=(\kappa_1,\ldots,\kappa_5)\in\R^5.$

Moreover, by (\ref{Delta12}),  functions in (\ref{def-hzo-5})  are in fact all {\em polynomials} of   their own variables.
\end{remark}

\subsection{Genericity of Nondegeneracy Conditions} Let $\mathcal J_i$ and $\mathcal N_i$ ($1\le i\le 5$) be the functions of $(H_1,\dots,H_5,Z)$ on $\P=(\R^{2n\times 2n})^5 \times  \R^{10n}$ defined in \eqref{fun-Ji} and \eqref{def-Ni-0} above with $m=2$.
By Remark \ref{rk-52}, all these functions are  polynomials of $(H_1,\dots,H_5,Z)$  on $\P.$

\begin{theorem} \label{nondeg-thm}
Let $\mathcal Q =\mathcal J_1\mathcal J_2\mathcal J_3\mathcal J_4\mathcal J_5\mathcal N_1\mathcal N_2\mathcal N_3\mathcal N_4\mathcal N_5$ be defined on $\P_1=(\mathbb S^{2n\times 2n})^5 \times  \R^{10n}.$
 Then $\mathcal Q\not\equiv 0$  on $\P_1.$ Consequently, the set
\[
\P_0=\{(H_1,\dots,H_5,Z)\in\P_1 :\mathcal Q(H_1,\dots,H_5,Z)\ne 0 \}
\]
 is open and dense in $\P_1.$
\end{theorem}

\begin{proof} It suffices to show that none of  the polynomials $\mathcal J_i$ or the polynomials $\mathcal N_i$ is identically zero on $\P_1.$ Although the polynomials within each family differ in form, these differences have little effect on the relevant computations. Consequently, it is enough to prove that  $\mathcal J_1 \not\equiv 0$ and $\mathcal N_1 \not\equiv 0$.

\medskip

(I) {\bf Proof of $\mathcal J_1\not\equiv 0.$}

\smallskip

For $\mathbf H=(H_1,\dots,H_5)\in (\mathbb S^{2n\times 2n})^5$ and $Z=(Z',\bm\kappa)\in \R^{10n-5}\times \R^5,$ define
\[
\theta_j= \theta_j(H_j,Z)=H_j\zeta_j^1(Z)-\zeta_j^2(Z)\quad (1\le j\le 5)
 \]
 and
 \[
  J(\mathbf H,Z)=\begin{bmatrix} D'\theta_1  & \rho_1  &0&0&0&0\\
D'\theta_2 & 0&\rho_2 &0&0&0\\
D'\theta_3 & 0&0&\rho_3 &0&0\\
D'\theta_4   &0&0&0& \rho_4 &0 \\
D'\theta_5 &0&0&0&0& \rho_5 \end{bmatrix},
\]
where $D'\theta_j =\frac{\partial\theta_j}{\partial Z'}= H_j D'\zeta_j^1(Z) - D'\zeta_j^2(Z)$, $
\rho_j =\rho_j(H_j,Z')=\frac{\partial\theta_j}{\partial \kappa_j}=H_j h_j^1(Z') -h_j^2 (Z');$
 thus
 \[
 \mathcal J_1(\mathbf H,Z)=\det J(\mathbf H,Z).
 \]
Since $h_1+\dots+h_5=0,$ we can write
\begin{equation}\label{theta-0}
\begin{cases}
   {\theta_1}  =   {\kappa_1}H_1  h_1^1-{\kappa_1} h_1^2,  \\
 {\theta_2}     =   H_2(h_1^1+ \kappa_2 h_2^1)-h_1^2 -\kappa_2 h_2^2, \\
  \theta_3  = H_3(h_1^1+ h_2^1 +\kappa_3 h_3^1)- h_1^2- h_2^2-\kappa_3 h_3^2,\\
  \theta_4     =H_4 [(\kappa_4-1) h^1_4 - h_5^1] -(\kappa_4-1)  h^2_4 + h_5^2 ,\\
\theta_5 = (\kappa_5-1)H_5 h^1_5 -(\kappa_5-1) h^2_5.
\end{cases}
\end{equation}
Observe that $\theta_1$ and $\theta_5$ each involve two $h_i^j$'s, while $\theta_2$ and $\theta_4$ each  involve four $h_i^j$'s, and  $\theta_3$ involves six $h_i^j$'s.
We aim to find some linear combinations of $\theta_i$'s (with coefficients independent of $Z'$)  that involve fewer terms of $h_i^j$'s. To this end,   let
\[
\mathrm{K}=\big\{\bm\kappa=(\kappa_1,\ldots,\kappa_5)\in\R^5 : \kappa_1\kappa_2\kappa_3 (\kappa_4-1)(\kappa_5-1)\ne 0\big\}
\]
and, for $\bm\kappa\in\mathrm K,$
\[
 M_1=\frac{H_2-H_1}{\kappa_2}, \;\;  M_2=\frac{H_3-H_1}{\kappa_3}-\frac{H_2-H_1}{\kappa_2\kappa_3},\;\; M_3=\frac{H_3-H_2}{\kappa_3},\;\; M_4=\frac{H_5-H_4}{\kappa_4-1}.
\]
From (\ref{theta-0}),   we derive that
 \begin{equation}\label{def-psi-0}
\begin{cases}
 \psi_1: =\frac{\theta_1}{\kappa_1}= \rho_1, \\
 \psi_2:  =\frac{1}{\kappa_2}\big( \theta_2-\frac{\theta_1}{\kappa_1}\big) =\rho_2+M_1 h_1^1,
\\
 \psi_3:=
 \frac{1}{\kappa_3}\big (\theta_3  + \frac{1-\kappa_2}{\kappa_1\kappa_2} \theta_1 -\frac{1}{\kappa_2} \theta_2\big )=    \rho_3+ M_2 h_1^1 + M_3 h_2^1,   \\
  \psi_4    :=\frac{1}{\kappa_4-1} \big ( \theta_4+\frac{\theta_5}{\kappa_5-1} \big )= \rho_4+M_4 h_5^1, \\
 \psi_5:  =\frac{\theta_5}{\kappa_5-1}= \rho_5. \end{cases}
\end{equation}
In the following, we write $A\approx B$ if matrices $A$ and $B$ have the same rank.
Then, by  elementary  operations on matrix $J(\mathbf H,Z)$ according to the linear relations in (\ref{def-psi-0}), it follows that
\begin{equation}\label{matrix-A}
J(\mathbf H,Z)\approx A(\mathbf H, Z)=\begin{bmatrix} D'\psi_1  &  \rho_1 &0&0&0&0\\
D'\psi_2 & - \frac{\rho_1}{\kappa_2}&\rho_2 &0&0&0\\
D'\psi_3 &  \frac{(1-\kappa_2)\rho_1}{\kappa_2\kappa_3} &-\frac{\rho_2}{\kappa_3}&\rho_3&0&0\\
D'\psi_4   &0&0&0&  \rho_4 &\frac{\rho_5}{\kappa_4-1 } \\
D'\psi_5 &0&0&0&0& \rho_5  \end{bmatrix}.
\end{equation}
To select  certain special $H_j$'s in $\mathbb S^{2n\times 2n},$   for  $E =(\epsilon_{k\ell})\in\R^{2\times n}$,  we define  $H(E)\in\mathbb S^{2n\times 2n}$    by
\[
H(E)^{rs}_{k\ell}= \epsilon_{k\ell} \delta_{(k\ell)}^{(rs)} \quad\text{ for all $r,k=1,2$ and $1\le s,\ell\le n,$}
\]
  where $\delta_{(k\ell)}^{(rs)}=\begin{cases}
1 & \mbox{if $(r,s)=(k,\ell)$}\\
0 & \mbox{otherwise.}
\end{cases}
$
For each $1\le j\le 5$ and  $\bm\epsilon^j=(\epsilon^j_{11},\epsilon^j_{12},\epsilon^j_{21},\epsilon^j_{22},\epsilon_j)\in\R^{5},$ let
\begin{equation}\label{sp-Hj}
\bm\varepsilon= (\bm\epsilon^1,\dots,\bm\epsilon^5),\; E^j=\begin{bmatrix} \epsilon^j_{11}&\epsilon^j_{12}&\epsilon_j &\cdots&\epsilon_j\\
\epsilon^j_{21}&\epsilon^j_{22}&\epsilon_j &\cdots&\epsilon_j\end{bmatrix}\in \R^{2\times n}\;\;\mbox{and}\;\; H_j=H(E^j).
\end{equation}
 Let $
 \mathcal A(\bm\varepsilon,Z)=A(\mathbf H,Z),
$ where $\mathbf H=(H_1, \dots,H_5)$  is given by $H_j=H(E^j)$ defined above.
For such $H_j$'s,
  the quantities $M_i$'s  above are given by
\begin{equation}\label{def-mi}
(M_i)_{k\ell}^{rs}=\begin{cases}  \mu^i_{k\ell}\, \delta_{(k\ell)}^{(rs)} & (r,k=1,2;\, 1\le s,\ell\le 2),\\
 \mu_i \,\delta_{(k\ell)}^{(rs)} & (r,k=1,2;\, 3\le s,\ell\le n),\end{cases}
\end{equation}
where  $\mu_i\in\R$ and $(\mu^i_{k\ell})\in\R^{2\times 2}$ are the corresponding quantities.
In particular,
\begin{equation}\label{def-mui}
\mu_1=\frac{\epsilon_2-\epsilon_1}{\kappa_2},\;\; \mu_2=\frac{\epsilon_3-\epsilon_1}{\kappa_3}-\frac{\epsilon_2-\epsilon_1}{\kappa_2\kappa_3}, \;\; \mu_3=\frac{\epsilon_3-\epsilon_2}{\kappa_3} \;\; \mbox{and}\;\; \mu_4=\frac{\epsilon_5-\epsilon_4}{\kappa_4-1}.
\end{equation}
 Let $\mathcal S$ be the set of  $Z'= (p_1, p_2, q_3, q_{4}, \x_3,\bar\x_4,  \y_1^1,\y_1^2,\y_2^1,\y_2^2,\y_3^1,\y_3^2, \bar\y_5^1,\bar\y_5^2)\in\R^{10n-5}$ with
\[
\bar\x_3=\bar\x_4=\bar\y_1^1=\bar\y_2^1=\bar\y_3^1 =\bar\y_5^1=\bar\y_1^2=\bar\y_2^2=\bar\y_3^2 =\bar\y_5^2=\bar\0.
\]
Then $\mathcal S\approx \R^{15}.$
For all $Z'\in\R^{10n-5}$, $1\le j\le 5$, $k=1,2$ and $1\le \ell\le n,$
\begin{equation}\label{kl-entry}
 (h_j^2)_{k\ell}(Z')  =\tilde\alpha_{r_j}^\ell(\y_j^k)-(\x_j\cdot\y_j^k)\delta_{r_j \ell} \quad\text{and}\quad
 (H_jh_j^1)_{k\ell}(Z')   =
\begin{cases}
\epsilon^j_{k\ell} p_j^k \alpha^\ell_{r_j}(\x_j) & \mbox{if $1\le \ell\le 2$},\\
\epsilon_j p_j^k \alpha^\ell_{r_j}(\x_j) & \mbox{if $3\le \ell\le n,$}
\end{cases}
\end{equation}
where $r_1,\ldots,r_5\in\{1,2\}$ are given in  (\ref{sel-r}).  Hence, by (\ref{sp-form-1}),
\begin{equation}\label{S-sp}
\begin{cases}
\bar\x_k(Z')=\bar\y_4^k( Z')=\bar\0\;\; (k=1,2),\\
(\rho_j)_{k\ell}( Z')=0 \;\; (1\le j\le 5; \, k=1,2; \, 3\le\ell\le n) \end{cases} \forall\,  Z'\in \mathcal S.
\end{equation}
 We now consider the system
\begin{equation}\label{null-sp-0}
 \mathcal A(\bm\varepsilon, Z)C =0,
  \end{equation}
 where $C =(C',\bm\tau)=(c_1,\ldots,c_{10n-5},\tau_1,\ldots,\tau_5)\in\R^{10n}.$
 Below, for  convenience, we write $ Z =(Z',\bm\kappa)=(z_1,\ldots,z_{10n-5},\kappa_1,\ldots,\kappa_5)\in\R^{10n-5}\times\mathrm{K}.
 $
Then, system (\ref{null-sp-0}) can be written as
\begin{equation}\label{null-sp-1}
\begin{cases} \sum\limits_{\nu=1}^{10n-5}   \frac{\partial (\psi_1)_{k\ell}}{\partial z_\nu} c_\nu  + (\rho_1)_{k\ell} \tau_1  =0,\\[2ex]
\sum\limits_{\nu=1}^{10n-5}   \frac{\partial (\psi_2)_{k\ell}}{\partial z_\nu} c_\nu -\frac{1}{\kappa_2}(\rho_1)_{k\ell}   \tau_1 +  (\rho_2)_{k\ell} \tau_2=0,\\[2ex]
\sum\limits_{\nu=1}^{10n-5}   \frac{\partial (\psi_3)_{k\ell}}{\partial z_\nu} c_\nu + \frac{1-\kappa_2}{\kappa_2\kappa_3} (\rho_1)_{k\ell}   \tau_1 -\frac{1}{\kappa_3}(\rho_2)_{k\ell} \tau_2+  (\rho_3)_{k\ell} \tau_3=0,\\[2ex]
\sum\limits_{\nu=1}^{10n-5}   \frac{\partial (\psi_4)_{k\ell}}{\partial z_\nu} c_\nu + (\rho_4)_{k\ell}   \tau_4 +\frac{1}{\kappa_4-1}(\rho_5)_{k\ell} \tau_5 =0,\\[2ex]
\sum\limits_{\nu=1}^{10n-5}   \frac{\partial (\psi_5)_{k\ell}}{\partial z_\nu} c_\nu   + (\rho_5)_{k\ell} \tau_5 =0 \qquad
\forall\, k=1,2;\; 1\le \ell\le n. \end{cases}
\end{equation}
We show that there exist $\bm\varepsilon\in (\R^5)^5$ and  $Z=(Z',{\bm\kappa})\in \mathcal{S}\times\mathrm{K}$ such that (\ref{null-sp-1}) implies $C=0;$ hence, $\mathcal A(\bm\varepsilon, Z)$ is nonsingular and $\mathcal J_1 =\det J(\mathbf H,Z)$ is not identically zero in $\P.$

First assume $n\ge 3.$ By  (\ref{def-mi}) and (\ref{def-psi-0}), for each $k=1,2$ and $3\le\ell\le n,$ since $\x_5=2\e_1$ (thus, $x_{5,\ell-1}=0$),   we have, in $\R^{10n-5}\times \mathrm{K},$
\begin{equation}\label{psi-kl}
\begin{split}
 (\psi_1)_{k\ell} &= \epsilon_1p_1^kx_{1,\ell-1} - y^k_{1,\ell-1},\\
 (\psi_2)_{k\ell}  &= \epsilon_2p_2^kx_{2,\ell-1}-y^k_{2,\ell-1} +\mu_1 \, p_1^kx_{1,\ell-1} ,\\
 (\psi_3)_{k\ell}  &= \epsilon_3p_3^kx_{3,\ell-1}-y^k_{3,\ell-1} +\mu_2\, p_1^kx_{1,\ell-1} +\mu_3\, p_2^kx_{2,\ell-1} , \\
 (\psi_4)_{k\ell} &= \epsilon_4 p_4^kx_{4,\ell-1}-y^k_{4,\ell-1}  ,\\
   (\psi_5)_{k\ell}  &=  - y^k_{5,\ell-1}.
\end{split}
\end{equation}
For $1\le d\le 10n-5$, $1\le s_1<\cdots<s_d\le 10n-5$, $\mathbf{s}=(s_1,\ldots,s_d)$, and $Z'_\mathbf{s}=(z_{s_1},\ldots,z_{s_d})\in\R^{d},$  we write $C'_{Z'_\mathbf{s}}=(c_{s_1},\ldots,c_{s_d})\in\R^d.$
By (\ref{S-sp}),  system (\ref{null-sp-1}) for all $3\le\ell\le n$ becomes, in $\mathcal{S}\times \mathrm{K}$,
\begin{equation}\label{n-sp-3}
\begin{cases}
 \sum\limits_{\nu=1}^{10n-5}  c_\nu   \epsilon_1  p_1^k \frac{\partial \bar\x_{1}}{\partial z_\nu}(Z')  -C'_{\bar\y^k_1} =\bar\0, \\[2ex]
\sum\limits_{\nu=1}^{10n-5}  c_\nu  \Big ( \epsilon_2  p_2^k \frac{\partial \bar\x_{2}}{\partial z_\nu}+\mu_1 p_1^k\frac{\partial \bar\x_1}{\partial z_\nu} \Big )(Z') -C'_{\bar\y^k_2}  =\bar\0,\\[2ex]
\sum\limits_{\nu=1}^{10n-5}  c_\nu  \Big (\mu_2\,p_1^k\frac{\partial \bar\x_1}{\partial z_\nu} +\mu_3\, p_2^k\frac{\partial \bar\x_2}{\partial z_\nu} \Big )(Z')  + \epsilon_3  p_3^k C'_{\bar\x_3} -C'_{\bar\y^k_3}   =\bar\0,\\[2ex]
 \sum\limits_{\nu=1}^{10n-5}  c_\nu   \frac{\partial \bar\y^k_4}{\partial z_\nu} (Z')  -\epsilon_4  p_4^k C'_{\bar\x_{4}}  =\bar\0,\\[2ex]
 C'_{\bar\y^k_{5}}  =\bar\0  \end{cases}    \forall\, k=1,2.
\end{equation}
Moreover, by (\ref{sp-form-1}),  we have that  for each $i,k=1,2,$
\begin{equation}\label{p3p4}
\begin{cases} p^k_3(Z')   =q_3^1p^k_1+q_3^2 p^k_2,\\
  p^k_4(Z')=q_4^1p^k_1+q_4^2p^k_2,\\
   \bar\x_i(Z')  =-q_3^i \bar\x_3 -q_4^i \bar\x_4,\\
   \bar\y_4^k(Z')=-  \bar\y_1^k-  \bar\y_2^k-  \bar\y_3^k-  \bar\y_5^k
   \end{cases}  \forall \,  Z'\in\R^{10n-5},
\end{equation}
and so,  by the definition of $\mathcal{S}$,  we have
\begin{equation}\label{d-x1x2}
 \frac{\partial \bar\x_i}{\partial z_\nu}(Z')  =-q_3^i\frac{\partial \bar\x_3}{\partial z_\nu} -q_4^i\frac{\partial \bar\x_4}{\partial z_\nu}\quad
 \forall\, Z'\in \mathcal S\quad (i=1,2;\,\nu=1,\ldots,10n-5).
\end{equation}
  Thus, (\ref{n-sp-3}) becomes, in $\mathcal{S}\times \mathrm{K}$, for each $1\le k\le 2,$
\begin{equation}\label{n-sp-30}
\begin{cases}
 \epsilon_1p_1^k q_3^1 C'_{\bar\x_3} +\epsilon_1p_1^k q_4^1 C'_{\bar\x_4} +C'_{\bar\y_1^k}  =\bar\0, \\
 ( \epsilon_2p_2^k  q_3^2 +\mu_1 p_1^kq_3^1)C'_{\bar\x_3} +(\epsilon_2p_2^k   q_4^2+\mu_1 p_1^kq_4^1) C'_{\bar\x_4} +C'_{\bar\y_2^k}  =\bar\0,  \\
  [(\mu_2-\epsilon_3) p_1^kq_3^1 +(\mu_3-\epsilon_3)p_2^kq_3^2] C'_{\bar\x_3} +(\mu_2p_1^kq_4^1 +\mu_3p_2^kq_4^2) C'_{\bar\x_4} +C'_{\bar\y_3^k}  =\bar\0, \\
  \epsilon_4(q_4^1p^k_1+q_4^2p^k_2) C'_{\bar\x_4} +C'_{\bar\y_1^k}+C'_{\bar\y_2^k}+C'_{\bar\y_3^k} + C'_{\bar\y_5^k}  =\bar\0,\\
 C'_{\bar\y_5^k} =\bar\0.
\end{cases}
\end{equation}
 After simple eliminations, (\ref{n-sp-30}) yields the system:
\begin{equation}\label{eq150}
\begin{cases} (ap_1^1q_3^1 + bp_2^1q_3^2) C'_{\bar\x_3} +(cp_1^1q_4^1+dp_2^1q_4^2) C'_{\bar\x_4}=\bar\0,\\
(ap_1^2q_3^1 + bp_2^2q_3^2) C'_{\bar\x_3} +(cp_1^2q_4^1+dp_2^2q_4^2) C'_{\bar\x_4}=\bar\0,\end{cases}
\end{equation}
where $a,$ $b,$ $c$ and $d$ are the numbers given by
\begin{equation}\label{def-abcd}
 \begin{cases} a=\epsilon_1 +\mu_1 + \mu_2- \epsilon_3,\quad
 b=\epsilon_2  +\mu_3 -\epsilon_3,\\
 c=\epsilon_1 +\mu_1  + \mu_2  -  \epsilon_4,\quad
 d=\epsilon_2 +\mu_3  -\epsilon_4,\end{cases}
\end{equation}
 with $\mu_i$'s given by (\ref{def-mui}).
It is clear that system (\ref{eq150}) admits only the trivial solution in $(C'_{\bar\x_3},C'_{\bar\x_4})$ provided that
  \[
   \det \begin{bmatrix} ap_1^1 q_3^1 + b p_2^1 q_3^2  & cp_1^1q_4^1+dp_2^1q_4^2 \\
 ap_1^2 q_3^1 + bp_2^2 q_3^2     & cp_1^2q_4^1+dp_2^2q_4^2
\end{bmatrix} =\det(p_1p_2) (adq_3^1q_4^2-bcq_3^2q_4^1) \ne 0;
\]
in this case, from (\ref{n-sp-30}), we also have $C'_{\bar\y^k_i}=\bar\0$ for all $k=1,2$ and $i=1,2,3,5.$
Therefore, if $Z=(Z',\bm\kappa)\in\mathcal{S}\times\mathrm{K}$ and $\epsilon_j$ $(1\le j\le 4)$ satisfy
\begin{equation}\label{det-not0}
 Q(p_1,p_2,q_3,q_4,\kappa_2,\kappa_3,\epsilon_1,\epsilon_2,\epsilon_3,\epsilon_4) =\det(p_1p_2) (adq_3^1q_4^2-bcq_3^2q_4^1) \ne 0,
 \end{equation}
then   $ C'_{\bar\x_3}= C'_{\bar\x_4}=C'_{\bar\y^k_i}=\bar\0$ for all $k=1,2$ and $i=1,2,3,5.$
Clearly, $Q$ is a nonzero rational function, with denominator only a polynomial of $(\kappa_2,\kappa_3)$; thus  the set $\mathcal{R}$ of $(\bm\varepsilon,Z',\bm\kappa)\in\R^{25}\times\mathcal{S}\times\mathrm{K}$ satisfying (\ref{det-not0})   is an open and dense subset of the ($45$-dimensional)  open set $\R^{25}\times\mathcal{S}\times\mathrm{K}$.

Now assume $n\ge 2.$ Relabeling
\[
(z_1,\dots,z_{15})=(p_1,p_2,q_3,q_4,x_{3,1}, y^1_{1,1}, y^2_{1,1},y^1_{2,1},y^2_{2,1},y^1_{3,1},y^2_{3,1})=:Z_{\mathcal{S}}\in \R^{15}
\]
and the corresponding $(c_1,\ldots,c_{15})=C'_{Z_{\mathcal{S}}}$, the system (\ref{null-sp-1}) with $(\bm\varepsilon,Z',\bm\kappa)\in \mathcal{R}$ reduces to
\begin{equation}\label{null-sp-22}
\begin{cases}
 \sum\limits_{\nu=1}^{15}   \frac{\partial (\psi_1)_{k\ell}}{\partial z_\nu} c_\nu  + (\rho_1)_{k\ell} \tau_1  =0,\\
 \sum\limits_{\nu=1}^{15}   \frac{\partial (\psi_2)_{k\ell}}{\partial z_\nu} c_\nu -\frac{1}{\kappa_2}(\rho_1)_{k\ell}   \tau_1 +  (\rho_2)_{k\ell} \tau_2=0,\\
 \sum\limits_{\nu=1}^{15}   \frac{\partial (\psi_3)_{k\ell}}{\partial z_\nu} c_\nu + \frac{1-\kappa_2}{\kappa_2\kappa_3} (\rho_1)_{k\ell}   \tau_1 -\frac{1}{\kappa_3}(\rho_2)_{k\ell} \tau_2+  (\rho_3)_{k\ell} \tau_3=0,\\
 \sum\limits_{\nu=1}^{15}   \frac{\partial (\psi_4)_{k\ell}}{\partial z_\nu} c_\nu + (\rho_4)_{k\ell}   \tau_4 +\frac{1}{\kappa_4-1}(\rho_5)_{k\ell} \tau_5 =0,\\
 \sum\limits_{\nu=1}^{15}   \frac{\partial (\psi_5)_{k\ell}}{\partial z_\nu} c_\nu   + (\rho_5)_{k\ell} \tau_5 =0
\end{cases}   (k,\ell=1,2).
\end{equation}
Note that condition (\ref{null-sp-22}) itself is a $20\times 20$ linear system in $(c_1,\dots,c_{15},\tau_1,\dots,\tau_5).$  It has been confirmed  in \cite{Ya20}  (using the MATLAB) that even with $Z$ given by the specific $Z_0^\star$ above in (\ref{ineq-cdY-51}) (which can satisfy (\ref{det-not0}) for some $\epsilon_1,\epsilon_2,\epsilon_3,\epsilon_4$), there exist  $(\epsilon^j_{k\ell})\in\R^{2\times 2}$ $(1\le j\le 5)$ such that system (\ref{null-sp-22}) admits only the trivial solution.

This finally proves that $\mathcal J_1\not\equiv 0.$

\medskip

(II) {\bf Proof of $\mathcal N_1\not\equiv 0.$}

\smallskip

As in Part (I), for $\mathbf H=(H_1,\dots,H_5)$ and $Z=(Z',\bm\kappa)$,  we use the same notations for
$\theta_j(H_j,Z),\, \rho_j(H_j,Z')$, $J(\mathbf H,Z),$ and let
 $B_1(\mathbf H) $  be defined by \eqref{def-Bi} with $N=5.$
 Since $f_1^1(Z)=\kappa_1h_1^1(Z'),$ we have
\[
B_1(\mathbf H)Df_1^1(Z)=  \begin{bmatrix} O & 0  &0&0&0&0\\
\kappa_1 (H_1-H_2) D'h_1^1 & (H_1-H_2)h_1^1&0 &0&0&0\\
\kappa_1 (H_1-H_3) D'h_1^1 & (H_1-H_3)h_1^1&0 &0&0&0\\
\kappa_1 (H_1-H_4) D'h_1^1 & (H_1-H_4)h_1^1&0 &0&0&0 \\
\kappa_1 (H_1-H_5) D'h_1^1 & (H_1-H_5)h_1^1&0 &0&0&0 \end{bmatrix}.
\]
 Thus, with the same functions $\psi_i$ as defined above,
 {\small
 \[
 J(\mathbf H,Z)+B_1(\mathbf H)Df_1^1(Z)  =  \begin{bmatrix} D'\theta_1 & \rho_1  &0&0&0&0\\
D'\theta_2+\kappa_1 (H_1-H_2) D'h_1^1 & (H_1-H_2)h_1^1&\rho_2 &0&0&0\\
D'\theta_3+\kappa_1 (H_1-H_3) D'h_1^1 & (H_1-H_3)h_1^1&0 &\rho_3&0&0\\
D'\theta_4+\kappa_1 (H_1-H_4) D'h_1^1 & (H_1-H_4)h_1^1&0 &0&\rho_4&0 \\
D'\theta_5+\kappa_1 (H_1-H_5) D'h_1^1 & (H_1-H_5)h_1^1&0 &0&0&\rho_5 \end{bmatrix}
\]
}
\[
\approx B(\mathbf H,Z):=\begin{bmatrix} D'\psi_1  &  \rho_1 &0&0&0&0\\
D'\psi_2 +N_2 D' h_1^1 & N_2 h^1_1 - \frac{\rho_1}{\kappa_2} &\rho_2  &0&0&0\\
D'\psi_3+N_3 D'h_1^1  & N_3 h_1^1 +\frac{(1-\kappa_2)\rho_1}{ \kappa_2\kappa_3}   &-\frac{\rho_2}{ \kappa_3}&\rho_3 &0&0\\
D'\psi_4 +N_4 D'h_1^1  &N_4 h_1^1&0&0&  \rho_4 &\frac{\rho_5}{ \kappa_4-1} \\
D'\psi_5 +N_5 D'h_1^1 &N_5 h_1^1 &0&0&0&  \rho_5   \end{bmatrix},
\]
 where
 \[
 N_2=\frac{\kappa_1}{\kappa_2} (H_1-H_2), \;\;N_3=\frac{\kappa_1}{\kappa_3} [ (H_1-H_3)  -\frac{H_1-H_2}{\kappa_2} ],\]
 \[
 N_4=\frac{\kappa_1}{\kappa_4-1} [ (H_1-H_4)  +\frac{H_1-H_5}{\kappa_5-1}],\;\; N_5=\frac{\kappa_1}{\kappa_5-1} (H_1-H_5).
 \]
 For simplicity, we write
 \[
 B(\mathbf H,Z) =\begin{bmatrix} D'\psi'_1  &  \phi_1 &0&0&0&0\\
D'\psi'_2   & \phi_2 &\rho_2  &0&0&0\\
D'\psi'_3 & \phi_3  &-\frac{\rho_2}{ \kappa_3}&\rho_3 &0&0\\
D'\psi'_4   &\phi_4 &0&0&  \rho_4 &\frac{\rho_5}{ \kappa_4-1} \\
D'\psi'_5  &\phi_5 &0&0&0&  \rho_5   \end{bmatrix},
\]
where
\[
\begin{cases} \psi'_1=\psi_1,  &\phi_1=\rho_1,\\
\psi'_2=\psi_2+N_2h^1_1, &\phi_2=N_2 h^1_1 - \frac{\rho_1}{\kappa_2},\\
\psi'_3=\psi_3+N_3h^1_1, &\phi_3=N_3 h_1^1 +\frac{(1-\kappa_2)\rho_1}{ \kappa_2\kappa_3},\\
\psi'_4=\psi_4+N_4h^1_1, &\phi_4=N_4 h_1^1,\\
\psi'_5=\psi_5+N_5h^1_1, &\phi_5= N_5h^1_1.
\end{cases}
\]
For $\bm\varepsilon \in(\R^5)^5$ and $Z =(Z',\bm\kappa)=(z_1,\ldots,z_{10n-5},\kappa_1,\ldots,\kappa_5)\in\mathcal S\times\mathrm{K},$  let $
 \mathcal B(\bm\varepsilon,Z)=B(\mathbf H,Z),$ where $\mathbf H=(H_1, \dots,H_5)$  with $H_j=H(E^j)$ defined by  (\ref{sp-Hj}). Then
  the quantities $N_i$'s above are given by
\begin{equation}\label{def-Ni}
(N_i)_{k\ell}^{rs}=\begin{cases}  \lambda^i_{k\ell}\, \delta_{(k\ell)}^{(rs)} & (r,k=1,2;\, 1\le s,\ell\le 2),\\[1ex]
 \lambda_i \,\delta_{(k\ell)}^{(rs)} & (r,k=1,2;\, 3\le s,\ell\le n)\end{cases} \quad\forall\, 2\le i\le 5,
\end{equation}
where  $\lambda_i\in\R$ and $(\lambda^i_{k\ell})\in\R^{2\times 2}$ are the corresponding quantities.
In particular,
\begin{equation}\label{def-Li}
 \begin{split}  & \lambda_2=\frac{\kappa_1(\epsilon_1-\epsilon_2)}{\kappa_2} , \quad  \lambda_3=\frac{\kappa_1}{\kappa_3} [ (\epsilon_1-\epsilon_3)  -\frac{\epsilon_1-\epsilon_2}{\kappa_2} ],\\
  &\lambda_4=\frac{\kappa_1}{\kappa_4-1} [ (\epsilon_1-\epsilon_4)  +\frac{\epsilon_1-\epsilon_5}{\kappa_5-1}],\quad   \lambda_5=\frac{\kappa_1 (\epsilon_1-\epsilon_5)}{\kappa_5-1}. \end{split}
\end{equation}
As in part (I), we  show that there exist $\bm\varepsilon\in (\R^5)^5$ and  $Z=(Z',{\bm\kappa})\in \mathcal{S}\times\mathrm{K}$ such that if $C=(c_1,\ldots,c_{10n-5},\tau_1,\ldots,\tau_5)\in\R^{10n}$ satisfies
\begin{equation}\label{null-sp-b}
 \mathcal B(\bm\varepsilon, Z)C =0,
  \end{equation}
 then $C=0$; hence, $\mathcal B(\bm\varepsilon, Z)$ is nonsingular and $\mathcal N_1 \not\equiv 0$ on $\P_1.$
Note  that  (\ref{null-sp-b}) can be written as
\begin{equation}\label{null-sp-b1}
\begin{cases} \sum\limits_{\nu=1}^{10n-5}   \frac{\partial (\psi'_1)_{k\ell}}{\partial z_\nu} c_\nu  + (\phi_1)_{k\ell} \tau_1  =0,\\[2ex]
\sum\limits_{\nu=1}^{10n-5}   \frac{\partial (\psi'_2)_{k\ell}}{\partial z_\nu} c_\nu + (\phi_2)_{k\ell} \tau_1 +  (\rho_2)_{k\ell} \tau_2=0,\\[2ex]
\sum\limits_{\nu=1}^{10n-5}   \frac{\partial (\psi'_3)_{k\ell}}{\partial z_\nu} c_\nu  + (\phi_3)_{k\ell} \tau_1 -\frac{1}{\kappa_3}(\rho_2)_{k\ell} \tau_2+  (\rho_3)_{k\ell} \tau_3=0,\\[2ex]
\sum\limits_{\nu=1}^{10n-5}   \frac{\partial (\psi'_4)_{k\ell}}{\partial z_\nu} c_\nu + (\phi_4)_{k\ell} \tau_1 + (\rho_4)_{k\ell}   \tau_4 +\frac{1}{\kappa_4-1}(\rho_5)_{k\ell} \tau_5 =0,\\[2ex]
\sum\limits_{\nu=1}^{10n-5}   \frac{\partial (\psi'_5)_{k\ell}}{\partial z_\nu} c_\nu   + (\phi_5)_{k\ell} \tau_1 + (\rho_5)_{k\ell} \tau_5 =0 \qquad
\forall\, k=1,2;\; 1\le \ell\le n. \end{cases}
\end{equation}

First assume $n\ge 3.$ By (\ref{psi-kl}) and (\ref{def-Ni}), for each $k=1,2$ and $3\le\ell\le n,$  we have, in $\R^{10n-5}\times \mathrm{K},$
\[
\begin{split}
 (\psi'_1)_{k\ell} &= \epsilon_1p_1^kx_{1,\ell-1} - y^k_{1,\ell-1},\\
 (\psi'_2)_{k\ell}  &= \epsilon_2p_2^kx_{2,\ell-1}-y^k_{2,\ell-1} +(\mu_1+\lambda_2) \, p_1^kx_{1,\ell-1} ,\\
 (\psi'_3)_{k\ell}  &= \epsilon_3p_3^kx_{3,\ell-1}-y^k_{3,\ell-1} +(\mu_2+\lambda_3) \, p_1^kx_{1,\ell-1} +\mu_3\, p_2^kx_{2,\ell-1} , \\
 (\psi'_4)_{k\ell} &= \epsilon_4 p_4^kx_{4,\ell-1}-y^k_{4,\ell-1} +\lambda_4 p_1^kx_{1,\ell-1},\\
   (\psi'_5)_{k\ell}  &=  - y^k_{5,\ell-1}+\lambda_5 p_1^kx_{1,\ell-1}.
\end{split}
\]
Note that $(N_jh_1^1(Z'))_{k\ell}=(\rho_i)_{k\ell}=0$ for all $Z'\in\mathcal S$ and all $2\le j\le 5$, $1\le i\le 5$, $k=1,2$ and $3\le\ell\le n.$  Thus, the system (\ref{null-sp-b1}) for all $3\le\ell\le n$ becomes, in $\mathcal{S}\times \mathrm{K}$, for each $1\le k\le 2,$
\begin{equation}\label{n-sp-b3}
\begin{cases}
 \sum\limits_{\nu=1}^{10n-5}  c_\nu   \epsilon_1  p_1^k \frac{\partial \bar\x_{1}}{\partial z_\nu}(Z')  -C'_{\bar\y^k_1} =\bar\0, \\[2ex]
\sum\limits_{\nu=1}^{10n-5}  c_\nu  \Big ( \epsilon_2  p_2^k \frac{\partial \bar\x_{2}}{\partial z_\nu}+(\mu_1+\lambda_2) p_1^k\frac{\partial \bar\x_1}{\partial z_\nu} \Big )(Z') -C'_{\bar\y^k_2}  =\bar\0,\\[2ex]
\sum\limits_{\nu=1}^{10n-5}  c_\nu  \Big ((\mu_2+\lambda_3)\,p_1^k\frac{\partial \bar\x_1}{\partial z_\nu} +\mu_3\, p_2^k\frac{\partial \bar\x_2}{\partial z_\nu} \Big )(Z')  + \epsilon_3  p_3^k C'_{\bar\x_3} -C'_{\bar\y^k_3}   =\bar\0,\\[2ex]
 \sum\limits_{\nu=1}^{10n-5}  c_\nu \Big ( \lambda_4 p_1^k\frac{\partial \bar\x_1}{\partial z_\nu}-  \frac{\partial \bar\y^k_4}{\partial z_\nu}\Big )  (Z')   + \epsilon_4  p_4^k C'_{\bar\x_{4}}  =\bar\0,\\[2ex]
 \sum\limits_{\nu=1}^{10n-5}  c_\nu  \lambda_5 p_1^k\frac{\partial \bar\x_1}{\partial z_\nu}(Z') - C'_{\bar\y^k_{5}}  =\bar\0.   \end{cases}\end{equation}
 As above, using   (\ref{sp-form-1}), (\ref{p3p4}) and (\ref{d-x1x2}),   system (\ref{n-sp-b3}) becomes, in $\mathcal{S}\times \mathrm{K}$, for each $1\le k\le 2,$
\begin{equation}\label{n-sp-b30}
\begin{cases}
 \epsilon_1p_1^k q_3^1 C'_{\bar\x_3} +\epsilon_1p_1^k q_4^1 C'_{\bar\x_4} +C'_{\bar\y_1^k}  =\bar\0, \\
[ \epsilon_2p_2^k  q_3^2 +(\mu_1+\lambda_2) p_1^kq_3^1] C'_{\bar\x_3} +[\epsilon_2p_2^k   q_4^2+(\mu_1 +\lambda_2)  p_1^kq_4^1] C'_{\bar\x_4} +C'_{\bar\y_2^k}  =\bar\0,  \\
  [(\mu_2+\lambda_3-\epsilon_3) p_1^kq_3^1 +(\mu_3-\epsilon_3)p_2^kq_3^2] C'_{\bar\x_3} +[(\mu_2+\lambda_3)p_1^kq_4^1 +\mu_3p_2^kq_4^2] C'_{\bar\x_4} +C'_{\bar\y_3^k}  =\bar\0, \\
   \lambda_4 p_1^k q_3^1 C'_{\bar\x_3} +[(\lambda_4   + \epsilon_4)q_4^1p^k_1+\epsilon_4 q_4^2p^k_2] C'_{\bar\x_4} +C'_{\bar\y_1^k}+C'_{\bar\y_2^k}+C'_{\bar\y_3^k} +C'_{\bar\y_5^k} =\bar\0,\\
 \lambda_5 p_1^k q_3^1 C'_{\bar\x_3} +\lambda_5 p_1^k q_4^1 C'_{\bar\x_4} +C'_{\bar\y_5^k} =\bar\0.
\end{cases}
\end{equation}
Again, after simple eliminations, (\ref{n-sp-b30})  reduces to the system:
\begin{equation}\label{eq-b150}
\begin{cases} (\alpha p_1^1q_3^1 + \beta p_2^1q_3^2) C'_{\bar\x_3} +(\gamma p_1^1q_4^1+\delta p_2^1q_4^2) C'_{\bar\x_4}=\bar\0,\\
(\alpha p_1^2q_3^1 + \beta p_2^2q_3^2) C'_{\bar\x_3} +(\gamma p_1^2q_4^1+\delta p_2^2q_4^2) C'_{\bar\x_4}=\bar\0,\end{cases}
\end{equation}
where $\alpha, \beta,\gamma$ and $\delta$ are the numbers given by
\[
\begin{cases}  \alpha=\epsilon_1 +\mu_1 + \mu_2- \epsilon_3+\lambda_2+\lambda_3+\lambda_5-\lambda_4,\quad
 \beta=\epsilon_2  +\mu_3 -\epsilon_3,\\
 \gamma=\epsilon_1 +\mu_1  + \mu_2  -  \epsilon_4+\lambda_2+\lambda_3+\lambda_5-\lambda_4,\quad
\delta=\epsilon_2 +\mu_3  -\epsilon_4,\end{cases}
\]
 with $\mu_i$'s and $\lambda_j$'s defined  by (\ref{def-mui}) and (\ref{def-Li}).
 Thus  (\ref{eq-b150}) admits only the trivial solution in $(C'_{\bar\x_3},C'_{\bar\x_4})$ provided that
  \[
   \det \begin{bmatrix} \alpha p_1^1 q_3^1 + \beta  p_2^1 q_3^2  &\gamma p_1^1q_4^1+\delta p_2^1q_4^2 \\
 \alpha p_1^2 q_3^1 + \beta p_2^2 q_3^2     & \gamma p_1^2q_4^1+\delta p_2^2q_4^2
\end{bmatrix} =\det(p_1p_2) (\alpha \delta q_3^1q_4^2-\beta \gamma q_3^2q_4^1) \ne 0;
\]
in this case, from (\ref{n-sp-b30}), we also have $C'_{\bar\y^k_i}=\bar\0$ for all $k=1,2$ and $i=1,2,3,5.$
Therefore, if $Z=(Z',\bm\kappa)\in\mathcal{S}\times\mathrm{K}$ and $\epsilon_j$ $(1\le j\le 5)$ satisfy
\begin{equation}\label{det-not0b}
 Q_1(p_1,p_2,q_3,q_4,\bm\kappa, \epsilon_1,\dots, \epsilon_5) =\det(p_1p_2) (\alpha \delta q_3^1q_4^2-\beta \gamma q_3^2q_4^1)  \ne 0,
 \end{equation}
then   $ C'_{\bar\x_3}= C'_{\bar\x_4}=C'_{\bar\y^k_i}=\bar\0$ for all $k=1,2$ and $i=1,2,3,5.$
Clearly, $Q_1$ is a nonzero rational function, with denominator only a polynomial of $\bm\kappa$; thus  the set $\mathcal R_1$ of $(\bm\varepsilon,Z',\bm\kappa)\in\R^{25}\times\mathcal{S}\times\mathrm{K}$ satisfying (\ref{det-not0b})   is an open and dense subset of the ($45$-dimensional)  open set $\R^{25}\times\mathcal{S}\times\mathrm{K}$.

Now assume that $n\ge 2.$ Then  (\ref{null-sp-b1}) reduces in $\mathcal{R}_1$ to a $20\times 20$ linear system
\begin{equation}\label{null-sp-b22}
\begin{cases} \sum\limits_{\nu=1}^{15}   \frac{\partial (\psi'_1)_{k\ell}}{\partial z_\nu} c_\nu  + (\phi_1)_{k\ell} \tau_1  =0,\\[2ex]
\sum\limits_{\nu=1}^{15}   \frac{\partial (\psi'_2)_{k\ell}}{\partial z_\nu} c_\nu + (\phi_2)_{k\ell} \tau_1 +  (\rho_2)_{k\ell} \tau_2=0,\\[2ex]
\sum\limits_{\nu=1}^{15}   \frac{\partial (\psi'_3)_{k\ell}}{\partial z_\nu} c_\nu  + (\phi_3)_{k\ell} \tau_1 -\frac{1}{\kappa_3}(\rho_2)_{k\ell} \tau_2+  (\rho_3)_{k\ell} \tau_3=0,\\[2ex]
\sum\limits_{\nu=1}^{15}   \frac{\partial (\psi'_4)_{k\ell}}{\partial z_\nu} c_\nu + (\phi_4)_{k\ell} \tau_1 + (\rho_4)_{k\ell}   \tau_4 +\frac{1}{\kappa_4-1}(\rho_5)_{k\ell} \tau_5 =0,\\[2ex]
\sum\limits_{\nu=1}^{15}   \frac{\partial (\psi'_5)_{k\ell}}{\partial z_\nu} c_\nu   + (\phi_5)_{k\ell} \tau_1 + (\rho_5)_{k\ell} \tau_5 =0 \qquad
\forall\, 1\le k,\ell \le 2, \end{cases}
\end{equation}
 in
$(c_1,\dots,c_{15},\tau_1,\dots,\tau_5)\in\R^{20}$. The determinant of the coefficient matrix of  this system  is a rational function
\[
Q_2(p_1,p_2,q_3,q_4,x_{3,1}, y^1_{1,1}, y^2_{1,1},y^1_{2,1},y^2_{2,1},y^1_{3,1},y^2_{3,1}, \bm\kappa,(\epsilon^j_{k\ell})_{k,\ell=1,2;1\le j\le 5}),
\]
whose denominator is a polynomial of $\bm\kappa.$ Calculations using specific parameters, subject to the restriction $\bm\kappa\in \mathrm{K}$,   will  show  that $Q_2\not\equiv  0.$ Consequently, the set $\mathcal R_2$ of $(\bm\varepsilon,Z',\bm\kappa)\in\R^{25}\times\mathcal{S}\times\mathrm{K}$ satisfying $Q_2\ne 0$   is an open and dense subset of $\R^{25}\times\mathcal{S}\times\mathrm{K}$.
 For any parameter in $\mathcal R_2$, system (\ref{null-sp-b22}) admits only the trivial solution $(c_1,\dots,c_{15},\tau_1,\dots,\tau_5)={\mathbf 0}\in\R^{20}.$

Finally, any parameter $\mathcal R_1\cap\mathcal R_2$ will make the system (\ref{null-sp-b}) to admit  only the trivial solution; hence $\mathcal B(\bm\varepsilon,Z)$ is nonsingular, proving  $\mathcal N_1\not\equiv 0.$
\end{proof}


\section{Proof of Theorem \ref{mainthm2}} \label{s-6}

In this final section, we prove Theorem \ref{mainthm2} by constructing smooth, strongly polyconvex functions $F$ on $\R^{2\times n}$ whose gradient  $\sigma=DF$ satisfies  Condition $O_5.$

 \subsection{Construction of Convex Functions}
We first prove a useful result for constructing smooth convex functions.

\begin{lemma} \label{construct-conv}
Let $q\ge 1$ and $N\ge 2$ be integers. Suppose that $c_j\in \R$ and $S_j,T_j\in\R^q$ $(j=1, \dots, N)$ satisfy the strict inequalities:
\begin{equation}\label{convex-01}
c_i-c_j+S_i \cdot (T_j-T_i)<0 \quad \forall\, 1\le i\ne j \le N.
\end{equation}
Then   there exists a smooth convex function $g\colon  \R^q\to \R$ satisfying the following properties:
\begin{itemize}
\item[(i)]  For some number $0<\epsilon<\min_{1\le i\ne j\le N}|T_i-T_j|,$
\[
g(t)=c_j +S_j\cdot (t-T_j)  \quad  \forall\, t\in  {\mathbb B_\epsilon(T_j),}\;\;1\le j\le N.
\]
\item[(ii)] There exists a constant $M>1$ such that for each $k=1,2,\dots,$
\[
 |D^k g(t)|\le L_k |t|^{1-k}\quad \forall\, t\in\R^q,\;\; |t| \ge  M,
\]
where  $D^k g$  is any $k$th-order  partial derivative of $g$ and $L_k>0$ is a constant.
\end{itemize}
\end{lemma}

\begin{proof}   For each $1\le j\le N$, let
\[
h_j(t)= c_j +S_j\cdot (t-T_j)\quad \forall\, t\in\R^q.
\]
Then, condition (\ref{convex-01}) becomes:
\begin{equation}\label{convex-11}
 h_j(T_j)>h_i(T_j) \quad \forall\, 1\le i\ne j \le N.
\end{equation}
Let
\[
A=\max_{1\le i\le N}|S_i|, \;\;  B=\max_{1\le i\le N}|T_i|, \;\;  K=1+(A+1)(B+1) - {\min_{1\le i\le N,\,|s|\le B+1} h_i(s)},
\]
and define
\[
\begin{cases} h_{N+1}(t)=(A+1)|t|-K,\\
 f(t) =\max \{h_1(t),\dots  ,h_N(t),h_{N+1}(t)\} \end{cases} \forall\, t\in\R^q.
 \]
Then  $f\colon\R^q\to\R$ is convex.  Moreover, for all $ |t|\le B+1,$
\[
h_{N+1}(t)\le (A+1)(B+1)-K< \min_{1\le i\le N,\,|s|\le B+1} h_i(s);
\]
hence,
\begin{equation}\label{f-1}
 f(t)=\max\limits_{1\le i\le N} h_i(t) \quad \forall\, |t|\le B+1.
 \end{equation}
 Let $C=\max_{1\le i\le N}|c_i|.$ Then,
for all $|t|\ge K+AB+C$ and  $1\le j\le N$, we have that
\[
\begin{split} h_{N+1}(t)&=A|t|+|t|  -K \ge  A|t| +AB+C\ge A|t-T_j|-A|T_j|+AB+C\\
&\ge S_j\cdot(t-T_j) +C = h_j(t)- c_j+C\ge  h_j(t),\end{split}
\]
which shows that
\begin{equation}\label{f-2}
f(t)=h_{N+1}(t) \quad \forall\, |t|\ge K+AB+C.
 \end{equation}
 Therefore, in view of  (\ref{convex-11}), \eqref{f-1} and \eqref{f-2},   there exists  $\delta\in (0,1)$ such that
\begin{equation}\label{3-conv}
\begin{cases} \B_\delta(T_i)\cap \B_\delta(T_j) =\emptyset &  \forall \;
 1\le i\ne j\le N,\\
 f(t) = h_j(t) &  \forall \; t\in \B_\delta(T_j),\; 1\le j\le N,\\
  f(t)=h_{N+1}(t) & \forall\, |t|\ge K+AB+C. \end{cases}
\end{equation}

Let $\rho\in C^\infty_c(\R^q)$ be such that
\[
\rho(s)\ge 0,\;\; \supp (\rho)\subset \B_1(0),\;\; \int_{\R^q}\rho(s)\,ds=1,\;\;  \int_{\R^q} \rho(s)s\,ds=0,
\]
and let $\rho_\epsilon(s)=\epsilon^{-q}\rho(s/\epsilon)$, where $0<\epsilon<\delta/2.$
Define
\[
g(t)=\int_{\R^q} f(t-s)\rho_\epsilon(s)\,ds =\int_{\R^q} f(s)\rho_\epsilon(t-s)\,ds  =\int_{\B_\epsilon(t)} f(s)\rho_\epsilon(t-s)\,ds \quad \forall\,t\in\R^q.
\]
Then $g\colon\R^q\to\R$ is a $C^\infty$  convex function.
We claim that $g$ satisfies (i) and (ii).

First, let $j\in \{1,\dots  ,N\}$ and $t\in \B_\epsilon(T_j);$  then  $\B_\epsilon(t)\subset \B_\delta(T_j).$ Hence, by (\ref{3-conv}),
\[
g(t)=\int_{\B_\epsilon(t)} f(s)\rho_\epsilon(t-s)\,ds    =\int_{\B_\epsilon(t)} h_j(s)\rho_\epsilon(t-s)\,ds  =h_j(t),
\]
verifying  (i).

Second, let $ M= K+AB+C+1,$ and let $t\in\R^q$ be such that $|t|\ge  M.$ Then
\[
\B_\epsilon(t) \subset \{s\in\R^q: |s|\ge K+AB+C\};
\]
 thus, by (\ref{3-conv}),   $f |_{\B_\epsilon(t)} =h_{N+1}|_{\B_\epsilon(t)}$. Hence,
\[
 g(t)=\int_{\B_\epsilon(t)} h_{N+1}(s)\rho_\epsilon(t-s)\,ds    =(A+1)\int_{\B_\epsilon(0)} |t-s|\rho_\epsilon(s)\,ds-K \quad \forall\, |t|\ge M.
\]
Therefore, for each $k=1,2,\dots,$ we have
\[
D^k g(t)=(A+1)\int_{\B_\epsilon(0)} D_t^k(|t-s|)\rho_\epsilon(s)\,ds \quad \forall \, |t|>M.
\]
Note that
\[
|D_t^k(|t-s|)|\le C_k |t-s|^{1-k}\le C_k (|t|-|s|)^{1-k} \le 2^{k-1}C_k |t|^{1-k}
\]
 for all $|t|\ge 1$ and $|s|\le \frac12,$ where $C_k>0$ is a constant; hence, as $0<\epsilon<1/2,$ it follows that,
\[
|D^k g(t)|\le (A+1)\int_{\B_\epsilon(0)} |D_t^k(|t-s|)|\rho_\epsilon(s)\,ds \le (A+1) 2^{k-1}C_k |t|^{1-k} \quad \forall\, |t|\ge M,
\]
which  proves  (ii) with $L_k=(A+1) 2^{k-1}C_k,$ completing the proof.
\end{proof}

\subsection{Polyconvex Functions Supporting  $\T_5$-configurations}

We use the notations of Section \ref{s-5}.

From the definition  of the set $\Z_1$ in (\ref{def-setZ1}), we select a nonempty bounded open set $\Z_2\subset\subset  \Z_1$   such that, for all $Z\in\Z_2,$ the condition
\begin{equation}\label{emb-2}
 \begin{cases}
 |\omega^1_i(Z')- \omega^1_j(Z')|>\sigma_0,\;\;   |\zeta^1_i(Z)-\zeta_j^1(Z)| >\sigma_0,\\
  c_i-c_j  + d_i [\delta(\zeta_j^1(Z))-\delta(\zeta_i^1(Z))]\\
   \;\; +\langle \zeta_i^2(Z)  - d_iD\delta(\zeta_i^1(Z)), \zeta_j^1(Z)-\zeta_i^1(Z)\rangle  <-\sigma_0  \end{cases}
\forall\, 1\le i\ne j\le 5,
\end{equation}
 holds for some constant  $\sigma_0>0.$

\begin{lemma}\label{construct-1}     There exists $\epsilon_0\in (0,1)$ such that for each $\epsilon\in (0,\epsilon_0)$ and $Z\in\Z_2,$ one can find a smooth convex function $G(A,\delta)$ on $\R^{2\times n}\times \R$ and a number
\[
0<\tau<\min_{1\le i\ne j\le 5}|\tilde  \zeta_i^1(Z)-\tilde  \zeta_j^1(Z)|
\]
 such that
\begin{equation}\label{dg-00}
G(\tilde  \zeta_i^1(Z))=c_i\;\;\mbox{and}\;\;
(G_A,G_\delta)|_{\B_\tau(\tilde  \zeta_i^1(Z))}=(Q_i^\epsilon,d_i)
\quad (1\le i\le 5),
\end{equation}
where
$
Q^\epsilon_i=  \zeta_i^2(Z)-\epsilon \zeta_i^1(Z)-d_i   D\delta (\zeta_i^1(Z)).
$
Moreover,
\begin{equation}\label{dg-0}
|D^k G(A,\delta)|\le L_k (|A|+|\delta |)^{1-k}\quad \forall\, |A|+|\delta |\ge M ,\;\; k\ge1,
\end{equation}
where $M>1$ and $L_k>0$ are some constants.
\end{lemma}

\begin{proof} By (\ref{emb-2}),  for each  $1\le i\ne j\le 5$ and $Z\in \Z_2,$ we have
\[
\begin{split}   c_i -c_j  & +  \langle Q^\epsilon_i,  \zeta_j^1(Z)   -\zeta_i^1(Z) \rangle +d_i [\delta(\zeta_j^1(Z))-\delta(\zeta_i^1(Z))]
 \\  &= c_i-c_j + \langle \zeta_i^2(Z)-d_i  D\delta (\zeta_i^1(Z)), \zeta_j^1(Z) -\zeta_i^1(Z) \rangle \\
  &\;\;\;\,\,+d_i [\delta(\zeta_j^1(Z))-\delta(\zeta_i^1(Z))] -\epsilon \langle \zeta_i^1(Z), \zeta_j^1(Z) -\zeta_i^1(Z) \rangle
 \\ &\le -\sigma_0 -\epsilon \langle \zeta_i^1(Z), \zeta_j^1(Z) -\zeta_i^1(Z) \rangle \\
 &\le -\sigma_0+J\epsilon<0 \quad \mbox{if\;\;$0<\epsilon<\epsilon_0,$}
   \end{split}
\]
where
\[
J=\max_{1\le i\ne j\le 5}\max_{\bar{\mathcal{Z}}_2}|\langle \zeta^1_i,\zeta^1_j-\zeta^1_i\rangle|+\sigma_0+1,\quad \epsilon_0=\sigma_0/J\in(0,1).
\]
Hence, for each $\epsilon\in (0,\epsilon_0)$ and $Z\in\Z_2,$ the existence of a smooth   convex function $G:\R^{2\times n}\times\R\to\R$ and a number $0<\tau <\min_{1\le i\ne j\le 5}|\tilde  \zeta_i^1(Z)-\tilde  \zeta_j^1(Z)|$ satisfying \eqref{dg-00} and \eqref{dg-0}  follows  from  Lemma \ref{construct-conv}.
\end{proof}

\begin{lemma}\label{fun-F0} Let $\epsilon\in (0,\epsilon_0)$ and $Z\in\Z_2,$  where $\epsilon_0\in(0,1)$ is given by Lemma \ref{construct-1}. Let $G$  be determined by Lemma \ref{construct-1} and define
\[
F_0(A)=\frac{\epsilon}{2} |A|^2+G(\tilde A) \quad \forall\, A\in\R^{2\times n}.
\]
 Then  there exists a number $0<s_0 <\min_{1\le i\ne j\le 5}| \zeta_i^1(Z)- \zeta_j^1(Z)|$  such that  for each $1\le i\le 5,$
\begin{equation}\label{linear-DF}
DF_0 (A)=\zeta_i^2(Z)+\epsilon (A- \zeta_i^1(Z))+d_i D\delta (A- \zeta_i^1(Z)) \quad \forall\, A\in B_{s_0}( \zeta_i^1(Z)).\end{equation}
 In particular,  $\zeta_i(Z) \in \K_{F_0}$ for all $1\le i\le 5.$

Moreover, there exist constants $C_k>0$ $(k=1,2,\dots)$ such that
\begin{equation}\label{growth-DF}
|D^k F_0(A)|\le C_k(|A|+1) \quad\forall\, A\in \R^{2\times n}.
\end{equation}
\end{lemma}

\begin{proof} By  Lemma  \ref{construct-1}, for each $1\le i\le 5$, we have $(G_A, G_\delta)|_{\B_\tau(\tilde \zeta_i^1(Z))}=(Q_i^\epsilon, d_i),$ where $0<\tau <\min_{1\le i\ne j\le 5}| \zeta_i^1(Z)- \zeta_j^1(Z)|$  is the number in  Lemma  \ref{construct-1}.
Thus, for each $A\in\R^{2\times n}$ with $\tilde A\in \B_\tau(\tilde \zeta_i^1(Z)),$  we have that
  \[
 DF_0(A)
 =\epsilon  A+Q^\epsilon_i + d_i D\delta(A)  =\zeta_i^2(Z)+\epsilon (A- \zeta_i^1(Z))+d_i D\delta (A- \zeta_i^1(Z)).
\]

Now let $0<s_0 <\tau$ be such that, for each $1\le i\le 5,$
\[
A\in B_{s_0}(\zeta_i^1(Z)) \;\; \imply \;\;  \tilde A\in \B_\tau(\tilde \zeta_i^1(Z)),
\]
from which (\ref{linear-DF}) follows. In particular, for each $1\le i\le 5$, since $DF_0(\zeta_i^1(Z))=\zeta_i^2(Z),$  we have  $\zeta_i(Z)=(\zeta^1_i(Z),\zeta^2_i(Z))\in \K_{F_0}.$

To prove (\ref{growth-DF}),  let $f(A)=G(\tilde A)=G(A,\delta(A))$ for all $A\in \R^{2\times n}\approx \R^q,$ where $q=2n.$ Then $F_0(A)=\frac{\epsilon}{2}|A|^2+ f(A) .$

Let $k\ge 1$ be an integer and  $\alpha=(\alpha_1,\dots,\alpha_q)\in \mathbb Z_+^q$ be a multi-index with
\[
|\alpha|= \alpha_1+\dots+\alpha_q=k;
\]
 here $\mathbb Z_+$ stands for the set of all nonnegative integers.
  It suffices to show that  there exists a  constant $\tilde L_k>0$ such that
\begin{equation}\label{gr-df}
|D^\alpha f(A)|\le \tilde L_k (|A|+1)  \quad \forall\, A\in\R^q.
\end{equation}
Since $D^\gamma \delta (A)=0$ for all  $\gamma\in\mathbb Z_+^q$ with $|\gamma|>2$,   it is easily seen that
\begin{equation}\label{gr-df-0}
D^\alpha f(A)=D^{(\alpha,0)}G(\tilde A)  + \sum   D^{(\beta,\ell)}G(\tilde A) D^{\gamma_1} \delta (A) \cdots D^{\gamma_\ell} \delta (A),
\end{equation}
where $D^{(\beta,\ell)}G(A,\delta)= D^\beta_A D^\ell_\delta G(A,\delta)$ for all $(A,\delta)\in\R^q\times \R,$ and the summation is taken over terms with $1\le \ell \le k $ and $\beta,\gamma_1,\dots,\gamma_\ell\in\mathbb Z_+^q$ satisfying
\[
 \beta+\gamma_1+\cdots+\gamma_\ell=\alpha \;\; \mbox{and} \;\; 1\le |\gamma_j|\le 2 \quad \forall \, 1\le j\le \ell.
\]
Note that $|D^\gamma \delta(A)|\le c_{|\gamma|} |A|^{2-|\gamma|}$ for all $A\in\R^q$ and $\gamma\in\mathbb Z_+^q$ with $1\le |\gamma|\le 2.$ Hence, by (\ref{dg-0}), for all $A\in \R^q$ with $|\tilde A|\ge M,$ we have
\[
\begin{split}
|D^\alpha f(A)| & \le L_k |\tilde A|^{1-k} + \sum L_{|\beta|+\ell} |\tilde A|^{1-|\beta|-\ell} c_{|\gamma_1|} |A|^{2-|\gamma_1|}\cdots c_{|\gamma_\ell|} |A|^{2-|\gamma_\ell|}\\
& \le L_k | A|^{1-k} + \sum L_{|\beta|+\ell} |A|^{1-|\beta|-\ell} c_{|\gamma_1|} | A|^{2-|\gamma_1|}\cdots c_{|\gamma_\ell|} | A|^{2-|\gamma_\ell|} \\
& =L_k | A|^{1-k} + \sum L_{|\beta|+\ell}c_{|\gamma_1|}  \cdots c_{|\gamma_\ell|} | A|^{1-|\beta|-\ell}  |  A|^{2\ell -|\gamma_1|-\cdots-|\gamma_\ell|}\\
& =L_k | A|^{1-k} + \sum L_{|\beta|+\ell} c_{|\gamma_1|}  \cdots c_{|\gamma_\ell|} | A|^{1+\ell-k}:=\sum_{j=1}^{k+1}\tilde c_j |A|^{j-k}.
\end{split}
\]
Therefore, if $|A|\ge M>1$ then $|\tilde A|\ge M$ and thus
\[
|D^\alpha f(A)|\le \sum_{j=1}^{k+1}\tilde c_j |A|^{j-k} \le\Big (\sum_{j=1}^{k+1}\tilde c_j \Big ) |A|,
\]
  which proves (\ref{gr-df}) for all $A\in\R^q,$ with
\[
\tilde L_k=\Big (\sum_{j=1}^{k+1}\tilde c_j \Big ) +\max_{|\beta|=k, \; |A|\le M}  |D^\beta f(A)|.
\]
\end{proof}

\subsection{Perturbation of $F_0$ for Nondegeneracy}
 We  perturb the function $F_0$ in Lemma \ref{fun-F0} around each $\zeta_i^1(Z)$ to obtain  a new strongly  polyconvex function $F$ such that
\[
  (D^2F(\zeta_1^1(Z)), \dots,D^2F(\zeta_5^1(Z)),Z)\in \P_0,
\]
where $\P_0$ is the set defined in Theorem \ref{nondeg-thm}.

To this end, let $B_1(0)$ be the unit ball in $\R^{2\times n}$ and  $\omega\in C^\infty_c(B_1(0))$ be a cut-off function such that
 $0\le \omega\le 1$ and   $\omega|_{B_{1/2}(0)}=1.$

For each $r>0$ and  $H=(H_{pq}^{ij})\in \mathbb S^{2n\times 2n}$, let
\[
V_{H,r}(A)= \frac12 \omega(A/r) \sum_{i,p=1}^2\sum_{j,q=1}^n  H^{ij}_{pq}a_{ij}a_{pq}\quad \forall\, A=(a_{ij})\in\R^{2\times n};
\]
then $V_{H,r} \in C^\infty_c(\R^{2\times n})$ has support in $B_r(0),$
\begin{equation}\label{cut-off}
\begin{split} &  V_{H,r}(A)=\frac12   \sum_{i,p=1}^2\sum_{j,q=1}^n  H^{ij}_{pq}a_{ij}a_{pq} \quad \forall\, A=(a_{ij})\in  B_{r/2}(0),\;\;\mbox{and}\\
& |D^2 V_{H,r}(A)|\le C_0 |H| \quad \forall\, A \in\R^{2\times n}, \;\; \mbox{where $C_0>0$ is a constant.}  \end{split}
\end{equation}

By the density of set $\P_0$   in $\P_1$ (Theorem \ref{nondeg-thm}), we select  $(H_1^0,\dots,H_5^0,Z_0)$  such that
\begin{equation}\label{ImFT-1}
 (H_1^0,\dots,H_5^0,Z_0)\in \P_0,\;\; Z_0\in  \Z_2 \;\;\; \mbox{and}\;\;\;  \sum_{j=1}^5 |H_j^0-\epsilon I -d_jH_0|<\frac{\epsilon}{4C_0},
 \end{equation}
where $C_0>0$ is the constant in (\ref{cut-off}).

Let  $F_0$ be defined as  in Lemma \ref{fun-F0} with $Z=Z_0.$ Let
\[
\begin{cases}
\tilde H_j=H_j^0-D^2 F_0(\zeta_j^1(Z_0))=H_j^0-\epsilon I-d_j H_0\;  \; (1\le j\le 5),\\
s_1=\min\{s_0,\sigma_0\}>0,\end{cases}
\]
where $s_0>0$ is as given   in Lemma \ref{fun-F0} with $Z=Z_0$, and $\sigma_0>0$ is as defined in (\ref{emb-2}).

\begin{proposition}\label{prop-def-F} Let
\begin{equation}\label{def-F}
F(A)=F_0(A)+\sum_{j=1}^5 V_{\tilde H_j,s_1}(A-\zeta_j^1(Z_0)) \quad \forall\, A \in\R^{2\times n}.
\end{equation}
Then,
\begin{equation}\label{prop-F}
\begin{cases}
 DF(A)=\zeta_i^2(Z_0)+H_i^0  (A- \zeta_i^1(Z_0))  \quad
 \forall\, A\in B_{\frac{s_1}{2}}(\zeta_i^1(Z_0)), \;  1\le i \le 5; \\
 \mbox{in particular, $DF(\zeta_i^1(Z_0))= \zeta_i^2(Z_0)$ and  thus $\zeta_i(Z_0) \in \K_F \, (1\le i\le 5),$}\\[1ex]
 |D^k F(A)|\le C_k(|A|+1)\;\;\mbox{for  some $C_k>0$}   \quad  \forall\, k\ge 1, \;\; A\in \R^{2\times n}. \end{cases}
\end{equation}
Moreover,  the function
\[
 g(A)=\frac{\epsilon}{4}|A|^2+\sum\limits_{j=1}^5 V_{\tilde H_j,s_1}(A-\zeta_j^1(Z_0))\quad (A \in\R^{2\times n})
\]
  is   smooth  and convex  on $\R^{2\times n},$ and thus the function $\tilde G(A, \delta)= g(A)+G(A,\delta)$ is  smooth and convex  on $\R^{2\times n}\times \R;$ therefore,
\[
F(A)=\frac{\epsilon}{4}|A|^2 + \tilde G(A, \delta(A))\quad  (A \in\R^{2\times n})
\]
  is strongly polyconvex.
\end{proposition}

\begin{proof}  The properties in (\ref{prop-F}) follow  from (\ref{cut-off}) and properties (\ref{linear-DF})-(\ref{growth-DF}) in Lemma \ref{fun-F0}.

The convexity of $g$ follows since, by (\ref{cut-off}) and (\ref{ImFT-1}), $D^2g(A)\ge \frac{\epsilon}{4}I$ for all $A\in\R^{2\times n}.$
\end{proof}

\subsection{Proof of Theorem \ref{mainthm2}}

Let   $F$ be as defined by (\ref{def-F}).
By (\ref{prop-F}),  we have
\[
 D^2F(\zeta_i(Z_0))=H_i^0 \quad (1\le i\le 5) \quad\text{and}\quad  (H^0_1,\dots,H_5^0,Z_0)\in \P_0.
\]
Hence, the function $\sigma=DF$ satisfies the {\em nondegeneracy conditions} (\ref{NC-1}) and (\ref{NC-2}) with $N=5.$

Let
 \[
Z_1(\rho)=(Z_1'(\rho),\kappa_1(\rho),\dots,\kappa_5(\rho)): \bar\B_{\tau_0}(0) \to \mathcal B_{k_0}(Z_0)\subset \Z_2
\]
 be  as defined in Lemma \ref{ImFT-lem} with $\sigma=DF.$
 For
 each $1\le i \le 5$ and $ \rho\in\bar{\mathbb B}_{\tau_0}(0),$ define
\[
  \begin{cases}
    \gamma_i (\rho)=h_i(Z_1'(\rho)),\\
 \pi_1(\rho)=\rho,\;\;
   \pi_i(\rho) =\rho+\gamma_1(\rho)+\dots  +\gamma_{i-1}(\rho)  \;\;  (2\le i\le 5), \\
\xi_i (\rho) =   \pi_i(\rho) +\kappa_i(\rho) \gamma_i(\rho)  \;\;  (1\le i\le 5),
\end{cases}
\]
where the functions $\{h_i(Z')\}_{i=1}^5$ are as defined in \eqref{def-hzo} with $N=5.$

We now  verify that $\sigma=DF$ satisfies Condition $O_5$  with  the functions $\{(\kappa_i(\rho),\gamma_i(\rho))\}_{i=1}^5.$

 First, note that $\xi_i(\rho)$ and $\pi_i(\rho)$ coincide with the quantities defined in  (\ref{def-xp-i}) for  $N=5.$
Hence, by \eqref{xi-ij-K},
\[
\xi_i(\rho)\in \K_F \quad\forall\, 1\le i\le 5, \quad \rho\in\bar\B_{\tau_0}(0).
\]
Thus, property (P1)(i) in Definition \ref{O-N} holds for any  $0<r_0\le \tau_0.$

Let $\tau_2\in (0,\tau_0)$ and $\ell_0\in (\frac12,1)$ be the constants determined in Theorem \ref{thm-OCN}  with $\sigma=DF$ and $N=5.$  Define
 \begin{equation}\label{def-delta12}
 \delta_1=  \min\limits_{1\le i\le 5, \, \rho\in \bar\B_{\tau_0}(0)} \Big \{1-\ell_0, \; \frac{1}{\kappa_i(\rho)} \Big\} \;\;\;\mbox{and}\;\;    \delta_2=  \max\limits_{1\le i\le 5, \, \rho\in \bar\B_{\tau_0}(0)} \Big\{\ell_0, \;\frac{1}{\kappa_i(\rho)}\Big \}.
 \end{equation}
Then
\[
0<\delta_1\le 1-\ell_0<\ell_0\le \delta_2<1, \quad\text{and}\quad   1/\delta_2\le \kappa_i(\rho)\le 1/\delta_1 \quad\forall\,  \rho\in\bar\B_{\tau_0}(0),\; 1\le i \le 5.
 \]
Since  $\gamma_i(\rho)=h_i(Z_1'(\rho))\in\Gamma_{r_i}\subset \Gamma$ for $1\le i\le 5$  and
\[
\gamma_1(\rho)+\dots+\gamma_5(\rho)=0 \quad\forall\, \rho\in\bar\B_{\tau_0}(0),
\]
  it follows that  the functions
\[
 (\kappa_i,\gamma_i)\colon  \bar\B_{\tau_0}(0) \to [1/\delta_2,1/\delta_1]\times \Gamma,  \quad  1\le i\le 5,
\]
satisfy conditions  \eqref{fun-x-g} and \eqref{satisfy-g}  in Definition \ref{O-N} for any  $0<r_0\le \tau_0.$

Moreover,   since $Z_0\in\Z_2,$ we have $\pi^1_i(0)\ne \pi^1_j(0)$ and $\xi_i^1(0)\ne \xi_j^1(0)$ for all $1\le i\ne j\le 5.$  Hence, we can select $\tau_3 \in (0,\tau_2)$  such  that
 \begin{equation}\label{tau-4}
 \xi_i^1(\bar\B_{\tau_3}(0))\cap \xi_j^1(\bar\B_{\tau_3}(0)) =\pi_i^1(\bar\B_{\tau_3}(0))\cap\pi_j^1(\bar\B_{\tau_3}(0))=\emptyset\quad \forall\, 1\le i\ne j\le 5.
 \end{equation}
Thus,  property (P1)(ii) in Definition \ref{O-N} holds for any  $r_0\in (0,\tau_3].$

Next, since $0<\delta_1\le 1-\ell_0<\ell_0\le \delta_2<1,$   Theorem \ref{thm-OCN} implies that the sets $\{S^r(\lambda)\}_{i=1}^5 $  are open for each  $r\in (0,\tau_2]$ and $\lambda\in [0,\delta_1] \cup [\delta_2,1).$

Furthermore, because $\xi_i(0)\ne \xi_j(0)$ and $\pi_i(0)\ne \pi_j(0)$ for  all $1\le i\ne j  \le 5,$  we can choose   $\delta_0 \in (\delta_2, 1)$ and $r_0\in (0,\tau_3)$  such  that the sets $\{S^r(\lambda)\}_{i=1}^5 $ are pairwise disjoint for each $\lambda\in \{0\}\cup [\delta_0,1)$ and $r\in (0,r_0].$
With this choice,  property (P2)  of Definition \ref{O-N} is satisfied.

Consequently,  $\sigma=DF$ satisfies Condition $O_5$ with the functions $\{(\kappa_i,\gamma_i)\}_{i=1}^5$ and  the constants
\[
0<\delta_1<\delta_2<\delta_0<1, \quad r_0>0.
\]

\subsection*{Acknowledgements}

B. Guo was supported by the National Key R\&D Program of China (2024YFA 1013301). S. Kim was supported by the National Research Foundation of Korea (NRF) grant funded by the Korea government (MSIT) (No. RS-2023-00217116) and (No. RS-2025-16064757). S. Kim is also grateful for support by the Open KIAS Center at Korea Institute for Advanced Study.


\begin{thebibliography}{99}

\bibitem{AF84}
E. Acerbi  and N. Fusco,  {\em Semicontinuity problems in the calculus of variations,}
Arch. Rational Mech. Anal., {\bf 86} (1984),  125--145.

  \bibitem{AGS05}
  L. Ambrosio, N.  Gigli and G. Savar\'e,  ``Gradient Flows in Metric Spaces and in the Space of Probability
Measures."  Birkhäuser Verlag, Basel, 2005.


        \bibitem{Ba77} J. M. Ball, {\em Convexity conditions and existence theorems in nonlinear elasticity,} Arch. Rational Mech. Anal., {\bf 63} (1977), 337--403.


  \bibitem{Bar}
{V. Barbu,} ``Nonlinear Differential Equations of Monotone Types in Banach Spaces." Springer, New York, 2010.

\bibitem{BC}
{H. Bauschke and P. Combettes,} ``Convex Analysis and Monotone Operator Theory in Hilbert Spaces." Second Edition. Springer, Cham, Switzerland. 2017.


\bibitem{BDDMS20}
{V. B\"ogelein, B. Dacorogna,  F. Duzaar, P. Marcellini and C. Scheven,} {\em Integral convexity and parabolic systems,} SIAM J. Math. Anal., {\bf 52}(2) (2020), 1489--1525.

\bibitem{BDM13}
{V. B\"ogelein, F. Duzaar and G. Mingione,} {\em The regularity of general parabolic systems with degenerate diffusion,} Memoirs of Amer. Math. Soc., {\bf 221} No. 1041, 2013.

\bibitem{Br}
{H. Br\'ezis,} ``Op\'eratuers Maximaux Monotones et Semi-Groupes de Contractions dans les Espaces de Hilbert." North-Holland Math. Stud. {\bf 5}, North-Holland, Amsterdam, 1973.





  \bibitem{BV19}
  T. Buckmaster and V. Vicol, {\em Nonuniqueness of weak solutions to the
Navier-Stokes equation,} Ann. of Math. {\bf 189} (2019), 101--144.



\bibitem{CT22}   M. Colombo and R. Tione, {\em Non-classical solutions of the $p$-Laplace equation,} (January 19, 2022), available at arXiv: 2201.07484.



  \bibitem{CFG11}
D. Cordoba, D. Faraco and  F. Gancedo, {\em Lack of uniqueness for weak solutions of the incompressible porous media equation},  Arch. Ration. Mech. Anal. {\bf 200} (3) (2011), 725--746.



 \bibitem{Da08}
B. Dacorogna, ``Direct Methods in the Calculus of Variations," Second Edition. Springer-Verlag, Berlin, Heidelberg, New York, 2008.

\bibitem{DM97} B. Dacorogna and P. Marcellini, {\em General existence theorems for Hamilton-Jacobi equations in the scalar and vectorial cases,} Acta Math., {\bf 178}(1) (1997), 1--37.

\bibitem{DM99}
B. Dacorogna and P. Marcellini, ``Implicit Partial Differential Equations."  Birkh\"auser Boston, Inc., Boston, MA, 1999.




\bibitem{DLSz09}
{C. De Lellis and L. Sz\'ekelyhidi}, {\em The Euler equations as a differential inclusion,} Ann. of Math. {\bf 170}(3) (2009), 1417--1436.

\bibitem{DLSz17}
{C. De Lellis and L. Sz\'ekelyhidi}, {\em High dimensionality and $h$-principle in PDE,} Bull. Amer. Math. Soc. (N.S.), {\bf 54}(2) (2017), 247--282.

\bibitem{DiB93}
E.~DiBenedetto, ``Degenerate Parabolic Equations." Springer, 1993.

\bibitem{Ev86}
L.~C.~Evans, {\em Quasiconvexity and partial regularity in the
calculus of variations,} Arch. Rational Mech.
Anal., {\bf 95} (1986), 227--252.

\bibitem{ESG05} L.C. Evans, O. Savin and W. Gangbo, {\em Diffeomorphisms and nonlinear heat flows,}  SIAM J. Math. Anal., {\bf 37}(3) (2005), 737--751.


 \bibitem{Gr86}
{ M. Gromov,} ``Partial Differential Relations."   Springer-Verlag, Berlin, 1986.

\bibitem{Ha95}
{ C. Hamburger,} {\em Quasimomotonicity, regularity and duality for nonlinear systems of partial differential equations,} Annali di Matematica pura ed applicata (IV), Vol. {\bf CLXIX} (1995),  321--354.

\bibitem{Ha03}
{ C. Hamburger,} {\em Partial regularity of minimizers
of polyconvex variational integrals,} Calc. Var. {\bf 18}  (2003), 221--241.

\bibitem{Is18}
P. Isett, {\em A proof of Onsager's conjecture,}  Ann. of Math. {\bf 188}(3) (2018), 871--963.

 \bibitem{Iw92}
T. Iwaniec, {\em p-Harmonic tensors and quasiregular mappings,} Ann. Math., {\bf 136} (1992), 589--624.


\bibitem{Jo23}
C. J. P. Johansson, {\em Wild solutions to the scalar Euler-Lagrange equations,} Available at arXiv:2303.07298v1 [math.AP] 13 Mar 2023


\bibitem{KY15}
S. Kim and B. Yan, {\em Convex integration and infinitely many weak solutions to the Perona-Malik equation in all dimensions}, SIAM J. Math. Anal. {\bf 47}(4) (2015), 2770--2794.


\bibitem{KY17}
S. Kim and B. Yan, {\em On Lipschitz solutions for some  forward-backward parabolic  equations. II: the case against Fourier}, Calc. Var. {\bf 56}(3) (2017), Art. 67, 36 pp.

\bibitem{KY18}
S. Kim and B. Yan, {\em On Lipschitz solutions for some  forward-backward parabolic  equations}, Ann. Inst. H. Poincar\'e Anal. Non Lin\'eaire {\bf 35}(1) (2018), 65-100.


\bibitem{KY25}
S. Kim and B. Yan, {\em On integral convexity, variational solutions and nonlinear semigroups,} J. Math. Pures Appl. (9) {\bf 194} (2025), Paper No. 103662, 37 pp.


\bibitem{KMS03} B. Kirchheim, S. M\"uller and V. \v Sver\'ak, {\em Studying nonlinear pde by geometry in matrix space,} in ``Geometric analysis and nonlinear partial differential equations", 347--395, Springer, Berlin, 2003.

\bibitem{LSU}
O. A. Lady\v{z}enskaja and V. A. Solonnikov and N. N. Ural'ceva, ``Linear and quasilinear equations of parabolic type. (Russian)," Translated from the Russian by S. Smith. Translations of Mathematical Monographs, Vol. 23 American Mathematical Society, Providence, R.I. 1968.


\bibitem{LP17} M.  Lewicka and M. R. Pakzad, {\em
Convex integration for the Monge--Amp\'ere equations in two dimensions,} Anal. PDE {\bf 10}(3) (2017), 695--727.

\bibitem{Ln}
G. M. Lieberman,  ``Second order parabolic differential equations,"  World Scientific Publishing Co., Inc., River Edge, N.J., 1996.

  \bibitem{Mo52}
C.~B.~Morrey,
{\em Quasiconvexity and the lower semicontinuity of multiple integrals,} Pacific J. Math., {\bf 2} (1952), 25--53.


\bibitem{MRS05} S. M\"uller, M. Rieger and V. \v Sver\'ak, {\em Parabolic equations with nowhere smooth solutions}, Arch. Rational Mech. Anal.  {\bf  177}(1) (2005), 1--20.




\bibitem{MSv03} {S. M\"uller and V. \v Sver\'ak}, {\em Convex integration for Lipschitz mappings and counterexamples to regularity},  Ann. of Math. (2) {\bf 157}(3)   (2003), 715--742.



\bibitem{Sh11}
{R. Shvydkoy}, {\em Convex integration for a class of active scalar equations}, J. Amer. Math. Soc.  {\bf 24}(4) (2011), 1159--1174.


  \bibitem{Sz04}
 L. Sz\'ekelyhidi, {\em The regularity of critical points of polyconvex functionals,} Arch. Rational Mech.  Anal.  {\bf 172}(1) (2004), 133--152.


 \bibitem{Ta93} L. Tartar, {\em Some remarks on separately convex functions,}  in ``Microstructure and
Phase Transitions," IMA Vol. Math. Appl. 54 (D. Kinderlehrer, R. D. James,
M. Luskin and J. L. Ericksen, eds.), Springer-Verlag, New York (1993), 191--204.


 \bibitem{Ya20}
B.~Yan, {\em Convex integration for diffusion equations and Lipschitz
solutions of polyconvex gradient flows,} Calc. Var.  (2020), 59:123. https://doi.org/10.1007/s00526-020-01785-7

 \bibitem{Ya22}
B.~Yan, {\em On nonuniqueness and nonregularity for gradient flows of polyconvex functionals,}
Calc. Var. (2024) 63:4
https://doi.org/10.1007/s00526-023-02609-0

 \bibitem{Ya23}
B.~Yan, {\em The {$\tau_N$}-configurations and polyconvex gradient flows,} Pure Appl. Funct. Anal. {\bf10}(1) (2025),   173–191.

\bibitem{Zh86}
 K. Zhang, {\em On the Dirichlet problem for a class of quasilinear elliptic systems of
partial differential equations in divergence form,} in ``Proceedings  of Tianjin Conference on Partial Differential
Equations in 1986," (S. S. Chern ed.) Lecture Notes in Mathematics, {\bf 1306}, pp.
262--277. Springer, Berlin, Heidelberg, New York, 1988.



\bibitem{Zh06}
K. Zhang, {\em On existence of weak solutions for one-dimensional forward-backward diffusion equations}, J. Differential Equations {\bf 220} (2) (2006), 322--353.


\end{thebibliography}
\end{document}